\documentclass[11pt,a4paper]{article}
\usepackage{multirow}
\usepackage{graphicx}
\usepackage{subfigure}
 \usepackage{epsfig}
 \usepackage{epstopdf}
\usepackage[T1]{fontenc}
\usepackage{extarrows}
\usepackage{geometry}
\usepackage{amsbsy,amsmath,latexsym,amsfonts, epsfig, color, authblk, amssymb, graphics, bm}
\usepackage{cases}
\usepackage{epsf,slidesec,epic,eepic}
\usepackage{fancybox}
\usepackage{fancyhdr}
\usepackage{setspace}
\usepackage{nccmath}
\usepackage{extarrows}
\usepackage{multirow}
\usepackage{extarrows}
\usepackage{graphicx}
\usepackage{algorithm}
\usepackage{algorithmic}
\usepackage{color}
\usepackage{setspace}
\usepackage{subfigure}
\usepackage{graphicx}
\usepackage{amsmath}
\usepackage{booktabs}
\usepackage{multirow}
\usepackage[colorlinks, citecolor=blue]{hyperref}

\newtheorem{theorem}{Theorem}[section]
\newtheorem{lemma}{Lemma}[section]
\newtheorem{definition}{Definition}[section]

\newtheorem{proposition}{Proposition}[section]

\newtheorem{assumption}{Assumption}[section]
\newtheorem{remark}{Remark}[section]

\usepackage[figuresright]{rotating}
\usepackage{adjustbox}
\newenvironment{proof}{{\noindent \bf Proof:}}{\hfill$\Box$\medskip}

\definecolor{lred}{rgb}{1,0.8,0.8}
\definecolor{lblue}{rgb}{0.8,0.8,1}
\definecolor{dred}{rgb}{0.6,0,0}
\definecolor{dblue}{rgb}{0,0,0.5}
\definecolor{dgreen}{rgb}{0,0.5,0.5}
\title{An Inexact Projected Regularized Newton Method for Fused Zero-norms Regularization  Problems}
\author{
 Yuqia Wu\footnote{Department of Applied Mathematics, The Hong Kong Polytechnic University, 
 	Hong Kong (\href{mailto:yuqia.wu@connect.polyu.hk}{yuqia.wu@connect.polyu.hk}).} ,\ \ 
 Shaohua Pan\footnote{School of Mathematics, South China University of Technology, Guangzhou (\href{mailto:shhpan@scut.edu.cn}{shhpan@scut.edu.cn}).} ,\ \
 Xiaoqi Yang\footnote{Department of Applied Mathematics, The Hong Kong Polytechnic University, 
 	Hong Kong (\href{mailto:mayangxq@polyu.edu.hk}{mayangxq@polyu.edu.hk}).} .
 }
\date{}
 
\begin{document}
 \maketitle
\begin{abstract}
We are concerned with structured $\ell_0$-norms regularization problems, with a twice continuously differentiable loss function and a box constraint. This class of problems have a wide range of applications in statistics, machine learning and image processing. To the best of our knowledge, there is no effective algorithm in the literature for solving them. In this paper, we first obtain a polynomial-time algorithm to find a point in the proximal mapping of the fused $\ell_0$-norms with a box constraint based on dynamic programming principle. We then propose a hybrid algorithm of proximal gradient method and inexact projected regularized Newton method to solve structured $\ell_0$-norms regularization problems. The whole sequence generated by the algorithm is shown to be convergent by virtue of a non-degeneracy condition, a curvature condition and a Kurdyka-{\L}ojasiewicz property. A superlinear convergence rate of the iterates is established under a locally H\"{o}lderian error bound condition on a second-order stationary point set, without requiring the local optimality of the limit point. Finally, numerical experiments are conducted to highlight the features of our considered model, and the superiority of our proposed algorithm.
\end{abstract}


\section{Introduction}\label{sec1}
Given a matrix $B\in\mathbb{R}^{p\times n}$, $\lambda_1 > 0$, $\lambda_2> 0$, $l \in \mathbb{R}^n_-$ and $ u\in\mathbb{R}^n_+$, we are interested in the following structured $\ell_0$-norms regularization problem with a box constraint:
\begin{equation}\label{model}
 \min_{x\in\mathbb{R}^n}\ F(x)\!:= f(x) + \lambda_1 \|Bx\|_0 + \lambda_2\|x\|_0 \ \ {\rm s.t.} \ \ l \leq x \leq u,
\end{equation}
where $f\!:\mathbb{R}^n\rightarrow \mathbb{R}$ is a twice continuously differentiable function, $\|\cdot\|_0$ denotes the $\ell_0$-norm (or cardinality) function. This model encourages sparsity of both variable $x$ and its linear transformation $Bx$. Throughout this paper, we define $g(\cdot) := \lambda_1\|B\cdot\|_0 +\lambda_2 \|\cdot\|_0 + \delta_{\Omega}(\cdot)$ and $\Omega:= \{x\in\mathbb{R}^n \ \lvert \ l\leq x\leq u\}$, where $\delta_\Omega(\cdot)$ denotes the indicator function of $\Omega$. 

\subsection{Motivation}
Given a data matrix $A\in\mathbb{R}^{m\times n}$ and its response $b\in\mathbb{R}^m$, the usual regression model is to minimize $f(x)\!:=\!h(Ax-b)$, where $h$ is a smooth function attaining its optimal solution at the origin. For the case where $h(\cdot) \!=\! \frac{1}{2}\|\cdot\|^2$, $f$ is the least-squares loss function of the linear regression. It is known that one of the formulations for finding a sparse vector while minimizing $f$ is the following $\ell_0$ regularization problem
\begin{equation}\label{L0-model}
	\min_{x\in\mathbb{R}^n} \ f(x) + \lambda_2 \|x\|_0,
\end{equation}
where the $\ell_0$-norm term shrinks some small coefficients to $0$ and identifies a set of influential components. However, the $\ell_0$-norm penalty only takes the sparsity of $x$ into consideration, but ignores its spatial nature, which sometimes needs to be considered in real-world applications. 
For example, in the context of image processing, the variables often represent the pixels of images, which are correlated with their neighboring ones. To recover the blurred images, Rudin et al. \cite{rudin92} took into account the differences between adjacent variables and used the total variation regularization, which penalizes the changes of the neighboring pixels and hence encourages smoothness in the solution.
Moreover, Land and Friedman \cite{land97} studied the phoneme classification on TIMIT database (Acpistoc-Phonetic Continuous Speech Corpus, NTIS, US Dept of Commerce), which consists of 4509 32ms speech frames and each speech frame is represented by 512 samples of 16 KHz rate. This database is collected from 437 male speakers. Every speaker provided approximately two speech frames of each of five phonemes, where the phonemes are ``sh'' as in ``she'', ``dcl'' as in ``dark'', ``iy'' as the vowel in ``she'', ``aa'' as the vowel in ``dark'', and ``ao'' as the first vowel in ``water''. Since each phoneme is composed of a series of consecutively sampled points, there is a high chance that each sampled point is close or identical to its neighboring ones. For this reason, Land and Friedman \cite{land97} considered imposing a fused penalty on the coefficients vector $x$, and proposed the following problems with zero-order variable fusion and first-order variable fusion respectively to train the classifier:
\begin{align}
	\min_{x\in\mathbb{R}^n} \frac{1}{2} \|Ax-b\|^2 + \lambda_1\|\widehat{B}x\|_0, \label{fused-L0} \\
	\min_{x\in\mathbb{R}^n} \frac{1}{2} \|Ax-b\|^2 + \lambda_1\|\widehat{B}x\|_1, \label{fused-L1}
\end{align}
where $A\in\mathbb{R}^{m\times n}$ represents the phoneme data, $b\in\mathbb{R}^m$ is the label vector, $\widehat{B}\in\mathbb{R}^{(n\!-\!1)\times n}$ with $\widehat{B}_{ii}= 1 \ {\rm and} \  \widehat{B}_{i,i+1} =-1$ for all $i \in\{1,\ldots,n\!-\!1\}$ and $\widehat{B}_{ij} = 0$ otherwise. If $f(\cdot) \!=\! \frac{1}{2}\|A\cdot-b\|^2$ and $B = \widehat{B}$, then we call \eqref{model} a fused $\ell_0$-norms regularization problem with a box constraint. 

Additionally taking the sparsity of $x$ into consideration, Tibshirani et al. \cite{tibshirani05} proposed the fused Lasso, given by
\begin{equation}\label{fused-lasso-Lag}
	\min_{x\in\mathbb{R}^n} \ \frac{1}{2} \|Ax-b\|^2   +\lambda_1 \|\widehat{B}x\|_1 + \lambda_2 \|x\|_1,
\end{equation}
and presented its nice statistical properties. Friedman et al. \cite{friedman07} demonstrated that the proximal mapping of $\lambda_1 \|\widehat{B}x\|_1 + \lambda_2 \|x\|_1$ can be obtained through a process, which is known as ``prox-decomposition" later. Based on the accessibility of this proximal mapping, various algorithms can efficiently address model \eqref{fused-lasso-Lag}, see \cite{liu09,liu10,li18b,molinari19}. In particular, Li et al. \cite{li18b} proposed a semismooth Newton augmented Lagrangian method (SSNAL) to solve the dual of \eqref{fused-lasso-Lag}. The numerical results presented in their study indicate that SSNAL is highly efficient.

It was claimed in \cite{land97} that both \eqref{fused-L0} and \eqref{fused-L1} perform well in signal regression, but the zero-order fusion one produces simpler estimated coefficient vectors. This observation suggests that model \eqref{model} with $f = \frac{1}{2}\|A\cdot -b\|^2$ and $B = \widehat{B}$ may be able to effectively find a simpler solution while performs well as the fused Lasso does. Compared with regularization problems using $\ell_0$-norm, those using $\|Bx\|_0$ regularization remain less explored in terms of algorithm development. 
According to \cite{land97}, the global solution of \eqref{fused-L0} is unattainable. 
However, one of its stationary points can be obtained. In fact, Jewell et al. \cite{jewell20} has revealed by virtue of dynamic programming principle that a point in the proximal mapping of $\lambda_1\|\widehat{B}\cdot\|_0$ can be exactly determined within polynomial time, which allows one to use the well-known proximal gradient (PG) method to find a stationary point of problem \eqref{fused-L0}. 
However, the highly nonconvex and nonsmooth nature of model \eqref{model} presents significant challenges in computing the proximal mapping of $g$ when $B=\widehat{B}$ and in developing effective optimization algorithms for solving it. As far as we know, no specific algorithms have yet been designed to solve these challenging problems. 

Another motivation for this work comes from our previous research \cite{wu22}. In that work, we considered a model which replaces the $\ell_0$-norm in \eqref{L0-model} by $\|x\|_q^q$, where $q\in(0,1)$ and $\|x\|_q$ is the $\ell_q$ quasi-norm of $x$, defined as $\|x\|_q = \big( \sum_{i=1}^n \lvert x_i\rvert^q\big)^{1/q}$. For this class of nonconvex and nonsmooth problems, we proposed a hybrid of PG and subspace regularized Newton methods (HpgSRN), which in Newton step restricts the subproblem on a subspace within which the objective function is smooth, and thus a regularized Newton method can be applied. It is worth noting that the subspace is induced by the support of the current iterate $x^k$. PG step is executed in every iteration, but it does not necessarily run a Newton step unless a switch condition is satisfied. A global convergence was established under a curvature condition and the Kurdyka-{\L}ojasiewicz (KL) property \cite{Attouch10} of the objective function, and a superlinear convergence rate was achieved under an additional local 
second-order error bound condition. Due to the desirable convergence result and numerical performance of HpgSRN, we aim to adapt a similar subspace regularized Newton algorithm to solve \eqref{model}, in which the subspace is induced by the combined support of $Bx^k$ and $x^k$.

\subsection{Related work}
In recent years, many optimization algorithms have been well developed to solve the $\ell_0$-norm regularization problems of the form \eqref{L0-model}, which includes iterative hard thresholding \cite{herrity06, blumensath08, blumensath10, Lu14B}, the penalty decomposition \cite{lu13}, the smoothing proximal gradient method \cite{bian20}, the accelerated iterative hard thresholding \cite{wu2020} and NL0R \cite{zhou21}. Among all these algorithms, NL0R is the only Newton-type method, which employs Newton method to solve a series of stationary equations confined within the subspaces identified by the support of the solution obtained by the proximal mapping of $\lambda_2\|\cdot\|_0$.

The PG method is able to effectively process model \eqref{model} if the proximal mapping of $g$ can be precisely computed. The PG method belongs to first-order methods, which have a low computation cost and weak global convergence conditions, but they achieve at most a linear convergence rate. On the other hand, the Newton method has a faster convergence rate, but it can only be applied to minimize sufficiently smooth objective functions. In recent years, there have been active investigations into the Newton-type methods for nonsmooth composite optimization problems of the form
\begin{equation}\label{composite}
	\min_{x\in \mathbb{R}^n} \ \Psi(x):=\psi(x) + \varphi(x),
\end{equation}
where $\varphi : \mathbb{R}^n \rightarrow \overline{\mathbb{R}}:=(-\infty, + \infty]$ is proper lower semicontinuous, and $\psi(x)$ is twice continuously differentiable on an open subset of $\mathbb{R}^n$ containing the domain of $\varphi$. The proximal Newton-type method is able to address \eqref{composite} with convex $\varphi$ and convex or weakly convex $\psi$ (see \cite{bertsekas82,Lee14,Yue19,Mordu23,liu22}). Another popular second-order method for solving \eqref{composite} is to minimize the forward-backward envelop (FBE) of $\Psi$, see \cite{stella17, Themelis18, Themelis19, themelis21}. In particular,
for those $\Psi$ with the proximal mapping of $\varphi$ being available, Themelis et al. \cite{Themelis18} proposed an algorithm called ZeroFPR, based on the quasi-Newton method, for minimizing the FBE of $\Psi$. They achieved the global convergence of the iterate sequence by means of the KL property of the FBE and its local superlinear rate under the Dennis-Mor\'{e} condition and the strong local minimality property of the limit point. An algorithm similar to ZeroFPR but minimizing the Bregman FBE of $\Psi$ was proposed in \cite{themelis21}, which achieves a superlinear convergence rate without requiring the strong local minimality of the limit point.
For the case that $\psi$ is smooth and $\varphi$ admits a computable proximal mapping, Bareilles et al. \cite{bareilles22} proposed an algorithm, alternating between a PG step and a Riemannian Newton method, which was proved to have a quadratic convergence rate under a positive definiteness assumption on the Riemannian Hessian at the limit point.

\subsection{Main contributions}
In this paper, we aim to design a hybrid of PG and inexact projected regularized Newton methods (PGiPN) to solve the structured $\ell_0$-norms regularization problem \eqref{model}. Let $x^k\in\Omega$ be the current iterate. Our method first runs a PG step with line search at $x^k$ to produce $\overline{x}^k$ with 
 \begin{equation}\label{def-muk}
 \overline{x}^k\in {\rm prox}_{\overline{\mu}_k^{-1}g}(x^k - \overline{\mu}_k^{-1}\nabla f(x^k)),
 \end{equation}
 where ${\rm prox}_{\overline{\mu}_k^{-1}g}(x)$ is the proximal mapping of $g$ defined by \eqref{pg}, $\overline{\mu}_k>0$ is a constant such that $F$ gains a sufficient decrease from $x^k$ to $\overline{x}^k$, and then judges whether the iterate enters Newton step or not in terms of some switch condition, which takes the following forms of structured stable supports:
\begin{equation}\label{switch-condition}
 {\rm supp}(x^k) = {\rm supp}(\overline{x}^k)\  \ {\rm and} \ \ {\rm supp}(Bx^k) = {\rm supp}(B\overline{x}^k).
 \end{equation}
 If this switch condition does not hold, we set $x^{k+1} = \overline{x}^k$ and return to the PG step. 
 Otherwise, due to the nature of $\ell_0$-norm, the restriction of $\lambda_1 \|Bx\|_0 + \lambda_2 \|x\|_0$ on the supports ${\rm supp}(Bx^k)$ and ${\rm supp}(x^k)$, i.e., $\lambda_1 \|(Bx)_{{\rm supp}(Bx^k)}\|_0 + \lambda_2 \|x_{{\rm supp}(x^k)}\|_0$, is a constant near $x^k$ and does not provide any useful information at all and thus, unlike dealing with the $\ell_q$ regularization problem in \cite{wu22}, we introduce the following multifunction $\Pi:\mathbb{R}^n \rightrightarrows \mathbb{R}^n$:
 \begin{align}\label{Pi-map}
	\Pi(z)&:=\left\{ x \in \Omega \ \lvert \ {\rm supp}(x) \subset {\rm supp}(z), \ {\rm supp}(Bx) \subset {\rm supp}(Bz) \right\} \nonumber\\
	&= \left\{ x\in\Omega \ \lvert \ x_{[{\rm supp}(z)]^c} = 0, \ (Bx)_{[{\rm supp}(Bz)]^c} = 0 \right\},
\end{align}
and consider the associated subproblem
 \begin{equation}\label{formulated-prob}
  \min_{x\in\mathbb{R}^n}\,f(x) + \delta_{\Pi_k}(x)\ \ {\rm with}\ \ \Pi_k = \Pi(x^k).
 \end{equation}
 It is noted that the set $\Pi(x^k)$ contains all the points whose supports are a subset of the support of $x^k$ as well as the supports of their linear transformation are a subset of the support of the linear transformation of $x^k$. It is worth pointing out that the multifunction $\Pi$ is not closed but closed-valued. 

We will show that a stationary point of \eqref{formulated-prob} is one for problem \eqref{model}.
Thus, instead of a subspace regularized Newton step in \cite{wu22}, following the projected Newton method in \cite{bertsekas82} and the proximal Newton method in \cite{Lee14,Yue19,Mordu23} and \cite{liu22}, our projected regularized Newton step minimizes the following second-order approximation of \eqref{formulated-prob} on $\Pi_k$:
 \begin{equation}\label{subp}
  \mathop{\arg\min}_{x\in\mathbb{R}^n}
  \Big\{\Theta_k(x)\!:= f(x^k) + \langle\nabla\!f(x^k),x-\!x^k\rangle
  +\frac{1}{2}\langle x\!-x^k, G_k(x\!-\!x^k)\rangle+\delta_{\Pi_k}(x)\Big\},
 \end{equation}
 where $G_k$ is an approximation to the Hessian $\nabla^2\!f(x^k)$, satisfying the following positive definiteness condition:
 \begin{equation}\label{con-Gk}
     G_k\succeq b_1\|\overline{\mu}_k(x^k\!-\!\overline{x}^k)\|^{\sigma} I,
 \end{equation}
 where $b_1 >0$, $\sigma \in (0,\frac{1}{2})$ and $\overline{\mu}_k$ is the one in \eqref{def-muk}.
 To cater for the practical computation, our Newton step seeks an inexact solution $y^k$ of \eqref{subp} satisfying
 \begin{numcases}{}\label{inexact-cond1}
 \Theta_k(y) - \Theta_k(x^k) \leq 0, \\ 
  \label{inexact-cond2}
  {\rm dist}(0, \partial \Theta_k(y)) 
  \le\frac{\min \{\overline{\mu}_k^{-1}, 1\}}{2}\min \left\{\|\overline{\mu}_k(x^k\!-\! \overline{x}^k)\|,\|\overline{\mu}_k(x^k\!-\overline{x}^k)\|^{1+\varsigma}\right\}
 \end{numcases}
 with $\varsigma\in (\sigma, 1]$. Setting the direction $d^k\!:=y^k-x^k$, a step size $\alpha_k \in (0,1]$ is found in the direction $d^k$ via backtrackings, and set $x^{k+1}:= x^k + \alpha_kd^k$. To ensure the global convergence, the next iterate still returns to the PG step. The details of the algorithm are given in Section \ref{sec3}.\\

The main contributions of the paper are as follows:

$\bullet$ Based on dynamic programming principle, we develop a polynomial-time ($O(n^{3+o(1)})$) algorithm for seeking a point $\overline{x}_k$ in the proximal mapping \eqref{def-muk} of $g$ when $B=\widehat{B}$.  This generalizes the corresponding result in \cite{jewell20} for finding $\overline{x}^k$ in \eqref{def-muk} from $g(\cdot) = \lambda_1\|B\cdot\|_0$ to $g(\cdot) = \lambda_1\|B\cdot\|_0 +\lambda_2 \|\cdot\|_0 + \delta_{\Omega}(\cdot)$. This also provides a PG algorithm for solving \eqref{model}. We establish a uniform lower bound on ${\rm prox}_{\mu^{-1} g}(x)$ for $x$ on a compact set and $\mu$ on a closed interval. This generalizes the corresponding results in \cite{Lu14B} for $\ell_0$-norm and in \cite{wu22} for $\ell_q$-norm with $0\!<\! q\! < \!1$, respectively. 

$\bullet$ We design a hybrid algorithm (PGiPN) of PG and inexact projected regularized Newton method to solve the structured $\ell_0$-norms regularization problem \eqref{model}, which includes the fused $\ell_0$-norms regularization problem with a box constraint as a special case. We obtain the global convergence of the algorithm by showing that the structured stable supports \eqref{switch-condition} hold when the iteration number is sufficiently large. Moreover, we establish a superlinear convergence rate under a H\"{o}lderian error bound on a second-order stationary point set, without requiring the local optimality of the limit point.

$\bullet$ The numerical experiments show that our PGiPN is more effective than some existing algorithms in the literature in terms of solution quality and efficiency.\\

The rest of the paper is organized as follows. In Section \ref{sec2} we give some preliminaries and characterizations of the optimality condition of model \eqref{model}. In Section \ref{sec3}, we give some results related to $g$, including a lower bound of proximal mapping of $g$, and an algorithm for finding a point in the proximal mapping of $\lambda_1\|\widehat{B}x\|_0 + \lambda_2\|x\|_0+\delta_{\Omega}(x)$. In Section \ref{sec4}, we introduce our algorithm and show that it is well defined. In Section \ref{sec5} we present the convergence analysis of our algorithm. The implementation scheme of our algorithm and the numerical experiments are presented in Section \ref{sec7}. 

\section{Preliminaries}\label{sec2}

\subsection{Notations}
For any $x\in\mathbb{R}^n$ and $\epsilon >0$, $\mathbb{B}(x, \epsilon):=\{z\ \lvert \ \|z-x\|\leq \epsilon\}$ denotes the ball centered at $x$ with radius $\epsilon$. Let ${\bf B}:=\mathbb{B}(0,1)$. For a closed and convex set $\Xi\subset \mathbb{R}^n$, ${\rm proj}_{\Xi}(z)$ denotes the projection of a point $z\in\mathbb{R}^n$ onto $\Xi$. We denote by $\mathcal{N}_{\Xi}(x)$ and $\mathcal{T}_{\Xi}(x)$ the normal cone and tangent cone of $\Xi$ at $x$, respectively. For a closed set $\Xi'\subset \mathbb{R}^n$, ${\rm dist}(z, \Xi') := \min_{x\in\Xi'}\|x-z\|$. For $t\in\mathbb{R}$, $t_+:= \max\{t, 0\}$. Fix any two nonnegative integers $j<k$, define $[j\!:\!k]:=\{j, j\!+\!1,...,k\}$ and $[k]:= [1\!:\!k]$. For an index set $T \subset [n]$, write $T^c := [n]\backslash T$ and $\lvert T\rvert$ is the number of the elements of $T$. 
Given any $x\in\mathbb{R}^n$, ${\rm supp}(x) : = \{ i \in [n] \ \lvert \ x_i \neq 0 \},$ $\lvert x\rvert_{\min} := \min_{i\in{\rm supp}(x)} \lvert x_i\rvert$ and $x_T\in\mathbb{R}^{\lvert T\rvert}$ is the vector consisting of those $x_j$'s with $j\in T$, and $x_{j:k}:= x_{[j:k]}$. {\bf 1} and $I$ are the vector of ones and the identity matrix, respectively, whose dimension is adaptive to the context.
Given a real symmetric matrix $H$, $\lambda_{\min}(H)$ denotes the smallest eigenvalue of $H$, and $\|H\|_2$ is the spectral norm of $H$.
For a matrix $A \in \mathbb{R}^{m \times n}$ and $S \subset [m]$, $A_{S\cdot}$(resp. $A_{\cdot T}$) denotes the matrix consisting of the rows (resp. the columns) of $A$ whose indices correspond to $S$(resp. $T$). We write the range of $A$ by ${\rm Range}(A) = \{ Ax \ \lvert \ x \in \mathbb{R}^n\}$ and the null space of $A$ by ${\rm Null}(A) = \{ y \in \mathbb{R}^n \ \lvert \ Ay = 0\}$. For another matrix $C\in\mathbb{R}^{m\times p}$, $[A\ C] \in \mathbb{R}^{m\times (n+p)}$ is defined as a matrix composed of two matrices, $A$ and $C$, placed side by side. For any $D\in \mathbb{R}^{p\times n}$, $[A; D]:=[A^{\top} \ D^{\top}]^{\top}$. For a proper lower semicontinuous function $h:\mathbb{R}^n\rightarrow \overline{\mathbb{R}}:=\mathbb{R} \cup {\pm \infty}$, its domain is denoted as ${\rm dom}h\!:=\{x\in\mathbb{R}^n\ \lvert\ h(x)<\infty\}$, and the proximal mapping of $h$ is defined as
\begin{equation}\label{pg}
{\rm prox}_{\mu h} (z) : = \mathop{\arg\min}_{x\in\mathbb{R}^n} \frac{1}{2\mu} \|x-z\|^2 + h(x),
\end{equation}
with $\mu > 0$ and $z\in\mathbb{R}^n$.

\subsection{Stationary conditions}\label{sec2.1}

Since problem \eqref{model} involves a box constraint and the structured $\ell_0$-norms function is lower semicontinuous, the set of global optimal solutions of model \eqref{model} is nonempty and compact. Moreover, by the continuity of $\nabla^2 \!f$ and the compactness of $\Omega$, we have $\nabla\!f$ is Lipschitz continuous on $\Omega$, i.e., there exists $L_1>0$ such that
\begin{equation}\label{fact1}
	\|\nabla f(x) - \nabla f(y)\|\leq L_1 \|x-y\| \ \ {\rm for\  all} \  x,y\in\Omega.
\end{equation}

For a point $\overline{x}\in{\rm dom}h$, we denote the basic (or limiting) subdifferential of $h$ at $\overline{x}$ by $\partial h(\overline{x})$  \cite[Definition 8.3]{RW09}. Next, we introduce two classes of stationary points for the general composite problem \eqref{composite}, which includes \eqref{model} as a special case. 
\begin{definition}\label{def1-Spoint}
	A vector $x\in\mathbb{R}^n$ is called a stationary point of problem \eqref{composite} if $0\in\partial\Psi(x)$. 
	A vector $x\in\mathbb{R}^n$ is called an $L$-stationary point of problem \eqref{composite}
	if there exists a constant $\mu>0$ such that
	$x\in{\rm prox}_{\mu^{-1}\varphi}(x-\!\mu^{-1}\nabla \psi(x))$. 
\end{definition}

Recall that $\Psi= \psi + \varphi$, where $\psi$ is twice continuously differentiable and $\varphi$ is proper and lower semicontinuous. If in addition $\varphi$ is assumed to be convex, 
then 
$$0\in \partial \Psi(x) \Leftrightarrow 0\in \mu(x-(x-\mu^{-1}\nabla \!\psi(x)))+\partial \varphi(x) \Leftrightarrow x={\rm prox}_{\mu^{-1}\varphi}(x\!-\mu^{-1}\nabla \psi(x)),$$
which means that for problem \eqref{composite} the $L$-stationarity of a point $x$ is equivalent to its stationarity. To extend this equivalence to a broader class of functions, we recall the definition of prox-regularity, which acts as a surrogate of convexity.
\begin{definition}\label{def-proxregular}
\cite[Definition 13.27]{RW09} A function $h\!:\mathbb{R}^n \rightarrow \overline{\mathbb{R}}$ is prox-regular at a point $\overline{x}\in{\rm dom}h$ for $\overline{v} \in \partial h(\overline{x})$ if $h$ is locally lower semicontinuous at $\overline{x}$, and there exist $r \geq 0$ and $\varepsilon > 0$ such that $h(x') \geq h(x) + v^{\top}(x' -x) - \frac{r}{2} \|x'-x\|^2$ for all $\|x'-\overline{x}\|\leq \varepsilon$, whenever $v\in\partial h(x)$, $\|v-\overline{v}\|<\varepsilon$, $\|x-\overline{x}\|<\varepsilon$ and $h(x)<h(\overline{x}) + \varepsilon$. If $h$ is prox-regular at $\overline{x}$ for all $\overline{v} \in \partial h(\overline{x})$, we say that $h$ is prox-regular at $\overline{x}$. 
\end{definition}

The following proposition reveals that under the assumption of the prox-regularity of $\varphi$, these two classes of stationary points for $\Psi$ coincide. The proof is similar to that in \cite[Remark 2.5]{wu22}, and the details are omitted here.
\begin{proposition}\label{prop-opt}
	If $\overline{x}$ is an $L$-stationary point of problem \eqref{composite}, then $0\in\partial\Psi(\overline{x})$. If $\varphi$ is prox-regular at $\overline{x}$ for $-\nabla \psi(\overline{x})$ and prox-bounded\footnote{For the definition of prox-boundedness, see \cite[Definitions 1.23]{RW09}.}, the converse is also true.
\end{proposition}

Next we focus on the stationary conditions of problem \eqref{model}. Recall that the multifunction $\Pi:\mathbb{R}^n \rightrightarrows \mathbb{R}^n$ defined in \eqref{Pi-map} is closed-valued. The following lemma characterizes the subdifferential of function $F$.
\begin{lemma}\label{subdiff-hpconvex}
	Fix any $z\in\Omega$. The following statements are true.
	\begin{description}
		\item[(i)] $\partial F(z) = \nabla\!f(z) + \partial g(z) = \nabla\!f(z) + \mathcal{N}_{ \Pi(z)}(z)$.  
		
		\item[(ii)]   $0\in \nabla\!f(x) + \mathcal{N}_{\Pi(z)}(x)$ implies that $0\in\partial F(x)$.
	\end{description}     
\end{lemma}
\begin{proof}
	The first equality of part (i) follows by \cite[Exercise 8.8]{RW09}, and the second one uses \cite[Lemma 2.2 (i)]{pan22}. Next we consider part (ii). Let $x\in \Pi(z)$. From the definition of $\Pi(\cdot)$, we have $\Pi(x) \subset \Pi(z)$, which along with $x \in \Pi(x)$ implies that $\mathcal{N}_{\Pi(z)}(x) \subset\mathcal{N}_{\Pi(x)}(x)$. Combining part (i), we obtain the desired result. 
\end{proof}

\begin{remark}
 Lemma \ref{subdiff-hpconvex} (ii) provides a way to seek a stationary point of $F$. Indeed, for any given $z\in\mathbb{R}^n$, if $x$ is a stationary point of problem $\mathop{\arg\min}\{ f(y) \ \lvert \ y\in\Pi(z)\}$, i.e., $0\in \nabla f(x) + \mathcal{N}_{\Pi(z)}(x)$, then by Lemma \ref{subdiff-hpconvex} (ii) it necessarily satisfies $0\in\partial F(x)$. This technique will be utilized in the design of our algorithm. In particular, when obtaining a good estimate of the stationary point, say $x^k$, we use a Newton step to minimize $f$ over the polyhedral set $\Pi(x^k)$, so as to enhance the speed of the algorithm.
\end{remark}
\subsection{Kurdyka-{\L}ojasiewicz property}\label{sec2.3}

In this subsection, we introduce the concept of Kurdyka-{\L}ojasiewicz (KL) property.
The KL property of a function plays an important role in the convergence analysis of first-order algorithms for nonconvex and nonsmooth optimization problems (see, e.g., \cite{Attouch10,Attouch13}). In this work, we will use it to establish the global convergence property of our algorithm.
\begin{definition}\label{KL-Def}
For any $\eta>0$, we denote by $\Upsilon_{\!\eta}$ the set consisting of all continuous concave $\varphi\!:[0,\eta)\to\mathbb{R}_{+}$ that are continuously
differentiable on $(0,\eta)$ with $\varphi(0)=0$ and $\varphi'(s)>0$ for all $s\in(0,\eta)$.
	A proper function $h\!:\mathbb{R}^n\!\to\overline{\mathbb{R}}$ is said to have the KL property at $\overline{x}\in{\rm dom}\,\partial h$
	if there exist $\eta\in(0,\infty]$, a neighborhood $\mathcal{U}$ of $\overline{x}$ and a function $\varphi\in\Upsilon_{\!\eta}$ such that for all $x\in\mathcal{U}\cap\big[h(\overline{x})<h<h(\overline{x})+\eta\big]$, $\varphi'(h(x)-h(\overline{x})){\rm dist}(0,\partial h(x))\ge 1.$ If $h$ has the KL property at each point of ${\rm dom} \partial h$, then $h$ is called a KL function.
\end{definition}
\section{Prox-regularity and proximal mapping of $g$} \label{sec3}

\subsection{Prox-regularity of $g$}\label{sec3.1}
In this subsection, we aim at proving the prox-regularity of $g$, which together with Proposition \ref{prop-opt} and the prox-boundedness of $g$ indicates that the set of stationary points of problem \eqref{model} coincides with that of its $L$-stationary points.

We remark here that the prox-regularity of $g$ cannot be obtained from the existing calculus of prox-regularity. In fact, it was revealed in \cite[Theorem 3.2]{poliquin10} that, for proper $f_i, \ i=1,2$ with $f_i$ being prox-regular at $\overline{x}$ for $v_i \in \partial f_i(\overline{x})$ and let $v:=v_1+v_2$, and $f_0:= f_1+f_2$, a sufficient condition such that $f_0$ is prox-regular at $\overline{x}$ for $v$ is 
\begin{equation}\label{constraint-qualification}
     w_1+w_2 = 0 \ {\rm with} \ w_i\in \partial^{\infty} f_i(\overline{x}) \Longrightarrow w_i = 0, \ i=1,2,
\end{equation}
where $\partial^{\infty}$ denotes the horizon subdifferential \cite[Definition 8.3]{RW09}.
We give a counter example to illustrate that the above constraint qualification does not hold for $f_i:\mathbb{R}^4\rightarrow \mathbb{R}$ with $ f_1= \|\widehat{B} \cdot \|_0$ and $f_2 = \| \cdot \|_0$. Let $\overline{x} = (0,0,0,1)^{\top}$. Then,  
$$ \partial^{\infty} f_1(\overline{x}) = \partial f_1(\overline{x}) = {\rm Range}((\widehat{B}_{[2]\cdot})^{\top}), \ \partial^{\infty} f_2(\overline{x}) = \partial f_2(\overline{x}) = {\rm Range}((I_{[3]\cdot})^{\top}).$$
By the expressions of $\partial^{\infty} f_1(\overline{x}) $ and $\partial^{\infty} f_2(\overline{x})$, it is immediate to check that the constraint qualification in \eqref{constraint-qualification} does not hold. Next, we give our proof toward the prox-regularity of $g$.

\begin{lemma}\label{lemma-pregular}
The function $g$ is prox-regular on its domain $\Omega$. Consequently, the set of stationary points of model \eqref{model} coincides with its set of $L$-stationary points.
\end{lemma}
\begin{proof}
	Fix any $\overline{x}\in\Omega$ and pick any $\overline{v} \in \partial g(\overline{x})$. Let $\lambda := \min\{\lambda_1, \lambda_2\}$ and $C:=[B; I]$. Pick any $\varepsilon \in (0, \min\{\lambda,\|\overline{v}\|,\frac{\lambda}{5\|\overline{v}\|}\})$ such that for all $x\in \mathbb{B}(\overline{x}, \varepsilon)$, ${\rm supp}(Cx)\supset{\rm supp}(C\overline{x})$. Next we prove that 
    \begin{equation}\label{eq-proxregular}
        g(x') \geq g(x) + v^{\top}(x'-x), \ {\rm for\ all \ } \|x'-\overline{x}\|\leq \varepsilon,\ v\in\partial g(x),\ \|v-\overline{v}\|< \varepsilon \ {\rm and} \ x\in \Xi,
    \end{equation} 
    where $\Xi := \{x\ | \ \|x-\overline{x}\| < \varepsilon, \ g(x) <g(\overline{x})+\varepsilon\},$ which implies that $g$ is prox-regular at $\overline{x}$ for $\overline{v}$.

    We first claim that for each $x\in \Xi$, it holds that ${\rm supp}(Cx) = {\rm supp}(C\overline{x})$ and $x\in\Omega$. If fact, by the definition of $\varepsilon$, ${\rm supp}(Cx)\supset{\rm supp}(C\overline{x})$. If ${\rm supp}(Cx) \neq {\rm supp}(C\overline{x})$, we have $g(x) \geq g(\overline{x}) + \lambda > g(\overline{x})+\varepsilon$, which yields that $x\notin \Xi$. Therefore, ${\rm supp}(Cx)={\rm supp}(C\overline{x})$. The fact that $x\in \Xi$ implies $x\in\Omega$ is clear. Hence the claimed facts are true.
    
    Fix any $x\in\Xi$. Consider any $x'\in\mathbb{B}(\overline{x},\varepsilon)$. If $x' \notin \Omega$, since $g(x')=\infty$, it is immediate to see that \eqref{eq-proxregular} holds, so it suffices to consider $x' \in \mathbb{B}(\overline{x},\varepsilon) \cap \Omega$. Note that
	${\rm supp}(Cx')\supset{\rm supp}(C\overline{x})={\rm supp}(Cx)$. If ${\rm supp}(Cx') \neq {\rm supp}(Cx)$, then $ g(x')\geq  g(x)+\lambda$. For any $v\in \partial  g(x)$ with $v\in\mathbb{B}(\overline{v},\varepsilon)$, $\|v\| \leq \|\overline{v}\| + \varepsilon \leq 2\|\overline{v}\|$, which along with $\|x' - x\|\leq \|x'-\overline{x}\|+\|x- \overline{x}\| \leq 2\varepsilon$ implies that 
	$$  g(x') -  g(x) -  v^{\top} (x' -x) \geq \lambda - \|v\|\|x'-x\| \geq \lambda - 4\|\overline{v}\| \varepsilon >0.$$
    Equation \eqref{eq-proxregular} holds.
	Next we consider the case ${\rm supp}(Cx') = {\rm supp}(Cx)$. Define
   $$\Pi^1(x):= \big\{ z\in\mathbb{R}^n \ \lvert \ (Bz)_{[{\rm supp}(Bx)]^c} = 0\big\}, \ \Pi^2(x):= \big\{ z\in\mathbb{R}^n \ \lvert \ z_{[{\rm supp}(x)]^c} = 0\big\}.$$ 
 Clearly, $\Pi(x) = \Pi^1(x) \cap \Pi^2(x) \cap \Omega$ and $\Pi_1(x),\Pi_2(x)$ and $\Omega$ are all polyhedral sets.
By \cite[Theorem 23.8]{convexanalysis}, for any $v\in\mathcal{N}_{\Pi(x)}(x)=\partial g(x)$, there exist $v_1\in \mathcal{N}_{\Pi^1(x)}(x)$, $v_2\in \mathcal{N}_{\Pi^2(x)}(x)$ and $v_3 \in \mathcal{N}_{\Omega}(x)$ such that $v = v_1 + v_2 + v_3$. Then,
	\begin{align*}
		g(x') -  g(x) - & v^{\top} (x' -x) =   \lambda_1\|Bx'\|_0 -\lambda_1 \|Bx\|_0 - v_1^\top (x'-x) \\
		& +\lambda_2 \|x'\|_0 - \lambda_2 \|x\|_0 - v_2^\top (x'-x) - v_3^\top (x'-x) \geq 0,
	\end{align*}
 where the inequality follows from $\lambda_1\|Bx'\|_0 -\lambda_1 \|Bx\|_0=0,v_1^\top (x'-x)= 0$, $\lambda_2 \|x'\|_0 - \lambda_2 \|x\|_0 =0, v_2^\top (x'-x) =0$ and $v_3^\top (x'-x) \leq 0$. Equation \eqref{eq-proxregular} is true.
	Thus, by the arbitrariness of $\overline{x}\in \Omega$ and $\overline{v}\in\partial  g(\overline{x})$, we conclude that $ g$ is prox-regular on set $\Omega$. 
\end{proof}

\subsection{Lower bound of the proximal mapping of $g$}
Given $\lambda>0$ and $x\in\mathbb{R}^n$, for any $z \in {\rm prox}_{\lambda\|\cdot\|_0}(x)$, it holds that if $|z_i|>0$, then $|z_i| \geq \sqrt{2\lambda}$ \cite[Lemma 3.3]{Lu14B}. This indicates that $|z|_{\min}$ has a uniform lower bound on ${\rm prox}_{\lambda\|\cdot\|_0}(x)$.    Such a uniform lower bound is shown to hold for $\ell_q$-norm with $0\!<\! q\! < \!1$ and played a crucial role in the convergence analysis of the algorithms involving subspace Newton method (see \cite{wu22}). Next, we show that such a uniform lower bound exists for $g$.

\begin{lemma}\label{lemma-lb}
 For any given compact set $\Xi\subset\mathbb{R}^n$ and constants $0<\underline{\mu}<\overline{\mu}$, define 
 \[   \mathcal{Z}:={\textstyle\bigcup_{z\in\Xi,\mu\in[\underline{\mu},\overline{\mu}]}}\ {\rm prox}_{\mu^{-1}g}(z).
 \]
 Then, there exists $\nu>0$ (depending on $\Xi,\underline{\mu}$ and $\overline{\mu}$) such that  $\inf_{u\in\mathcal{Z}\backslash\{0\}}\,\lvert [B;I]u\rvert_{\min}\ge\nu$. 
\end{lemma}
\begin{proof}
 Write $C:=[B; I]$. By invoking \cite[Corollary 3]{bauschke99} and the compactness of $\Omega$, there exists $\kappa > 0$ such that for all index set $J\subset [n\!+\!p]$,
 \begin{equation}\label{bounded-linear-regular}
  {\rm dist}(x, {\rm Null}(C_{J\cdot})\cap \Omega)  \leq \kappa {\rm dist}(x, {\rm Null}(C_{J\cdot})) \ \ {\rm for\ any\ } x\in\Omega.
 \end{equation}
 In addition, there exists $\sigma > 0$ such that for any index set $J\subset [n\!+\!p]$ with $\{C_{j\cdot}\}_{j\in J}$ being linearly independent, 
 \begin{equation}\label{bound-lambdamin}
  \lambda_{\min}(C_{J\cdot}C_{J\cdot}^\top) \geq \sigma. 
 \end{equation}
 
For any $z\in \Xi$ and $\mu \in [\underline{\mu},\overline{\mu}]$, 
	define $h_{z,\mu}(x):= \frac{\mu}{2} \|x-z\|^2$ for $x\in\mathbb{R}^n$. By the compactness of $\Omega$, $[\underline{\mu},\overline{\mu}]$ and $\Xi$, there exists $\delta_0\in(0,1)$ such that for all $z\in\Xi$, $\mu\in [\underline{\mu},\overline{\mu}]$ and $x, y \in\Omega$ with $\|x - y\|<\delta_0$, $\overline{\mu}(\|x\|+\|y\|+2\|z\|)\|x-y\|<\lambda:=\min\{\lambda_1, \lambda_2\}$, and consequently,
	\begin{align}\label{temp-ineq0}
		\lvert h_{z,\mu}(x) - h_{z,\mu}(y)\rvert 
		&=\frac{\mu}{2} \lvert\langle x-y, x+y-2z\rangle \rvert \le \frac{\overline{\mu}}{2}(\|x\|+\|y\|+2\|z\|)\|x-y\|<\frac{\lambda}{2}. 
	\end{align}
	
	Now suppose that the conclusion does not hold. Then there is a sequence $\{\overline{z}^k\}_{k\in\mathbb{N}}\subset \mathcal{Z}\backslash\{0\}$ such that  $\lvert C\overline{z}^k\rvert_{\min}\le\frac{1}{k}$ for all $k\in\mathbb{N}$. Note that $C$ has a full column rank. We also have $\lvert C\overline{z}^k\rvert_{\min}>0$ for each $k\in\mathbb{N}$. By the definition of $\mathcal{Z}$, for each $k\in\mathbb{N}$, there exist $z^k\in\Xi$ and $\mu_k\in[\underline{\mu},\overline{\mu}]$ such that $\overline{z}^k \in {\rm prox}_{\mu_k^{-1}g}(z^k)$. Since $\lvert C\overline{z}^k\rvert_{\min}\in\big(0,\frac{1}{k}\big)$ for all $k\in\mathbb{N}$, there exist $\mathcal{K}\subset\mathbb{N}$ and an index $i\in [n\!+\!p]$ such that 
	\begin{equation}\label{temp-ineq1}
		0<\lvert (C\overline{z}^k)_i\rvert = \lvert C\overline{z}^k\rvert_{\min} < \frac{\delta_0\sigma}{\kappa \|C\|_2}\quad{\rm for\ each}\ k\in\mathcal{K}, 
	\end{equation}
	where $\kappa$ and $\sigma$ are the ones appearing in \eqref{bounded-linear-regular} and \eqref{bound-lambdamin}, respectively.
	Fix any $k\in\mathcal{K}$. Write $Q_k:=[n\!+\!p]\backslash{\rm supp}(C\overline{z}^k)$ and choose $J_k\subset Q_k$ such that the rows of $C_{J_k\cdot}$ form a basis of those of $C_{Q_k\cdot}$.
    Let $\widehat{J}_k:=J_k\cup\{i\}$. If $J_k=\emptyset$, then $C_{\widehat{J}_k\cdot}$ has a full row rank. If $J_k\ne\emptyset$, then $C_{J_k\cdot}\overline{z}^k=0$, which implies that $C_{\widehat{J}_k\cdot}$ also has a full row rank (if not, $C_{i\cdot}$ is a linear combination of the rows of $C_{J_k\cdot}$, which along with $C_{J_k\cdot}\overline{z}^k=0$ implies that $C_{i\cdot} \overline{z}^k = 0$, contradicting to $\lvert (C\overline{z}^k)_i\rvert=\lvert C\overline{z}^k\rvert_{\rm min} > 0$). 
	Let $\widetilde{z}^k := {\rm proj}_{{\rm Null}(C_{\widehat{J}_k\cdot})}(\overline{z}^k)$. Then, 
	$C_{\widehat{J}_k\cdot}\widetilde{z}^k = 0$ and $(\overline{z}^k-\widetilde{z}^k)\in{\rm Range}\,(C_{\widehat{J}_k\cdot}^{\top})$. The latter means that there exists $\xi^k\in\mathbb{R}^{\lvert\widehat{J}_k\rvert}$ such that $\overline{z}^k-\widetilde{z}^k=C_{\widehat{J}_k\cdot}^{\top}\xi^k$.
	Since  $C_{\widehat{J}_k\cdot}$ has a full row rank and $\| C_{\widehat{J}_k\cdot}\overline{z}^k\|=\lvert (C\overline{z}^k)_i\rvert$, we have
	\begin{equation}\label{Czbar-lb}
		\lvert (C\overline{z}^k)_i\rvert =
		\| C_{\widehat{J}_k\cdot}\overline{z}^k - C_{\widehat{J}_k\cdot} \widetilde{z}^k \|     =\|C_{\widehat{J}_k\cdot}C_{\widehat{J}_k\cdot}^{\top}\xi^k\|\ge \sigma \|\xi^k\|,
	\end{equation}
	where the last inequality is due to \eqref{bound-lambdamin}. Combining \eqref{Czbar-lb} with \eqref{temp-ineq1} yields
	$\|\xi^k\|<\kappa^{-1}\|C\|_2^{-1}\delta_0$. Therefore,
	\begin{equation}\label{bound-tildez}
		\|\overline{z}^k - \widetilde{z}^k\|=\|C_{\widehat{J}_k\cdot}^{\top}\xi^k\|\le \|C_{\widehat{J}_k\cdot}\|_2\|\xi^k\|\le\|C\|_2\|\xi^k\|<\kappa^{-1}\delta_0.  
	\end{equation}
	Let $\widehat{z}^k:= {\rm proj}_{{\rm Null}(C_{\widehat{J}_k\cdot})\cap \Omega}(\overline{z}^k)$. From \eqref{bounded-linear-regular} and \eqref{bound-tildez}, it follows that 
	\begin{equation}\label{bound}
		\|\overline{z}^k-\widehat{z}^k\| = {\rm dist}(\overline{z}^k, {\rm Null}(C_{\widehat{J}_k\cdot})\cap \Omega) \leq \kappa {\rm dist}(\overline{z}^k, {\rm Null}(C_{\widehat{J}_k\cdot})) = \kappa \|\overline{z}^k - \widetilde{z}^k\|<\delta_0. 
	\end{equation}
	Note that $\widehat{z}^k, \overline{z}^k \in \Omega$. From \eqref{bound} and \eqref{temp-ineq0}, it follows that
	\begin{equation}\label{temp-ineq3}
		\lvert h_{z^k,\mu_k}(\widehat{z}^k) - h_{z^k,\mu_k}(\overline{z}^k)\rvert < \frac{\lambda}{2}. 
	\end{equation}
    Next we claim that ${\rm supp}(C\widehat{z}^k) \cup \{i\} \subset {\rm supp}(C\overline{z}^k)$. Indeed, 
	since the rows of $C_{\widehat{J}_k\cdot}$ form a basis of those of $C_{[Q_k\cup\{i\}]\cdot}$ and $C_{\widehat{J}_k\cdot}\widehat{z}^k = 0$, $C_{[Q_k\cup \{i\}]\cdot}\widehat{z}^k = 0$. 
	Then, ${\rm supp}(C_{[Q_k\cup\{i\}]\cdot}\widehat{z}^k)\cup\{i\} = {\rm supp}(C_{[Q_k\cup\{i\}]\cdot}\overline{z}^k)$. Since all the entries of $C_{[Q_k\cup \{i\}]^c\cdot}\overline{z}^k$ are nonzero, it holds that ${\rm supp}(C_{[Q_k\cup \{i\}]^c\cdot}\widehat{z}^k) \subset {\rm supp}(C_{[Q_k\cup \{i\}]^c\cdot}\overline{z}^k)$, which implies that ${\rm supp}(C\widehat{z}^k) \cup \{i\} \subset {\rm supp}(C\overline{z}^k)$. Thus, the claimed inclusion follows, which implies that $g(\overline{z}^k) -  g(\widehat{z}^k) \geq \lambda$. This together with \eqref{temp-ineq3} yields
	$h_{z^k,\mu_k}(\overline{z}^k)+g(\overline{z}^k) - (h_{z^k,\mu_k}(\widehat{z}^k)+g(\widehat{z}^k)) \geq \lambda - \frac{\lambda}{2} = \frac{\lambda}{2}$, contradicting to $\overline{z}^k \in {\rm prox}_{\mu_k^{-1}g}(z^k)$. The proof is completed.
\end{proof}  

The result of Lemma \ref{lemma-lb} will be utilized in Proposition \ref{xk-prop} to justify the fact that the sequences $\{\lvert B\overline{x}^k\rvert_{\min}\}_{k\in\mathbb{N}}$ and $\{\lvert \overline{x}^k\rvert_{\min}\}_{k\in\mathbb{N}}$ are uniformly lower bounded, where $\overline{x}^k$ is obtained in \eqref{def-muk} (or \eqref{ls-PG} below). This is a crucial aspect in proving the stability of ${\rm supp}(x^k)$ and ${\rm supp}(Bx^k)$ when $k$ is sufficiently large.

	\subsection{Proximal mapping of a fused $\ell_0$-norms function with a box constraint}\label{sec3.2}
	Using the idea of \cite{killick12}, Jewell et al. \cite{jewell20} presented a polynomial-time algorithm for computing the proximal mapping of the fused $\ell_0$-norm $\lambda_1\|\widehat{B}\cdot\|_0$, where $\widehat{B}x = (x_1-x_2;...; x_{n-1}-x_n)$ for any $x\in\mathbb{R}^n$. We extend the result of \cite{jewell20} for computing the proximal mapping of the fused $\ell_0$-norms $\lambda_1\|\widehat{B}\cdot\|_0+\lambda_2\|\cdot\|_0+\delta_{\Omega}(\cdot)$, i.e., for any given $z\in\mathbb{R}^n$, seeking a global optimal solution of the problem  
	\begin{equation}\label{fusedl0-prox}
		\min_{x\in\mathbb{R}^n}\, h(x;z):= \frac{1}{2}\|x-z\|^2+\lambda_1 \|\widehat{B}x\|_0 + \lambda_2 \|x\|_0 + \delta_{\Omega}(x).
	\end{equation}
	To simplify the deduction, for each $i\in [n]$, we define $\omega_i:\mathbb{R}\rightarrow \overline{\mathbb{R}}$ by $\omega_i(\alpha) := \lambda_2 \lvert \alpha\rvert_0 + \delta_{[l_i, u_i]}(\alpha)$. It is clear that for all $x\in\mathbb{R}^n$, $\lambda_2 \|x\|_0 + \delta_{\Omega}(x) = \sum_{i=1}^n \omega_i(x_i)$. 
     Let $H(0):= -\lambda_1$, and for each $s\in [n]$, define 
 \begin{equation}\label{Hfun}
         H(s) := \min_{y\in\mathbb{R}^s}  h_s(y;z_{1:s}), \ {\rm where} \  h_s(y;z_{1:s}):=\frac{1}{2}\|y- z_{1:s}\|^2 + \lambda_1 \|\widehat{B}_{\cdot [s]}y\|_0 + \sum_{j=1}^s \omega_j(y_j).
 \end{equation}
 It is immediate to see that $H(n)$ is the optimal value to \eqref{fusedl0-prox}. For each $s\in[n]$, define function $P_{s}\!:[0\!:\!s\!-\!1]\times \mathbb{R} \to \overline{\mathbb{R}}$ by 
	\begin{align}\label{Ps-fun}
		P_s(i, \alpha) := H(i) + \frac{1}{2}  \|\alpha {\bf 1} - z_{i+1:s}\|^2 + \sum_{j=i+1}^s \omega_j(\alpha) +\lambda_1.
	\end{align}
    For any given $y\in \mathbb{R}^s$, if $i$ is the largest integer in $[0\!:\!s\!-\!1]$ such that $y_i\neq  y_{i+1}$, and $y_{i+1} = y_{i+2} = ... = y_s = \alpha$, then $y = (y_{1:i}; \alpha {\bf 1})$ and
    \begin{align*}
        h_s(y; z_{1:s})&  \!= \!\frac{1}{2}  \|y_{1:i} \!-\! z_{1:i}\|^2 \!+\! \lambda_1\|\widehat{B}_{\cdot [i]}y_{1:i}\|_0 \!+\! \sum_{j=1}^i \omega_j(y_j) \!+\! \frac{1}{2}  \|y_{i+1:s}\! -\! z_{i+1:s}\|^2 \!+\! \sum_{j=i+1}^s \omega_j(y_j) +\lambda_1\\
        & = h_i(y_{1:i}; z_{1:i})+ \frac{1}{2}  \|\alpha{\bf 1} - z_{i+1:s}\|^2 + \sum_{j=i+1}^s \omega_j(\alpha) +\lambda_1.
    \end{align*}
    If $y_{1:i}$ is optimal to $\min_{y'\in\mathbb{R}^i}  h_i(y';z_{1:i})$, then $H(i) = h_i(y_{1:i};z_{1:i})$, which by the definitions of $P_s$ and $h_s$ yields that $P_s(i, \alpha) = h_s(y; z_{1:s}).$
    In the following lemma, we prove that the optimal value of $\min_{i\in[0:s-1], \alpha\in\mathbb{R}}P_s(i,\alpha)$ is equal to $H(s)$, by which we give characterization to an optimal solution of $h_s(\cdot;z_{1:s})$.
    	\begin{lemma}\label{prop-transform}
     Fix any $s\in [n]$. The following statements are true.
     \begin{description}
         \item[(i)] $H(s) = \min_{i\in[0:s-1], \alpha \in \mathbb{R}} P_s(i,\alpha)$.
         \item[(ii)] Assume that $(i_s^*, \alpha_s^*) \in \mathop{\arg\min}_{i\in[0:s-1], \alpha \in \mathbb{R}} P_s(i,\alpha)$. Then $y^* = (y^*_{1:i_s^*}; \alpha_s^* {\bf 1})$ is a global solution of $\min_{y\in\mathbb{R}^s} h_{s}(y; z_{1:s})$ with $y_{1:i_s^*}^*\in \mathop{\arg\min}_{v\in\mathbb{R}^{i_s^*}} h_{i_s^*}(v;z_{1:i_s^*})$.
     \end{description}
	\end{lemma}
	\begin{proof}
		{\bf (i)} Let $y^*$ be an optimal solution to problem \eqref{Hfun}. If $y^*_i=y^*_j$ for all $i,j\in [s]$, let $i_s^* = 0$; otherwise, let $i^*_s$ be the largest integer such that $y^*_{i^*_s}\neq y^*_{i^*_s+1}$. Set $\alpha_s^* = y^*_{i^*_s+1}$. If $i^*_s \neq 0$, from the definition of $H(\cdot)$, $h_{i^*_s}(y^*_{1:i^*_s};z_{1:i^*_s}) \geq H(i^*_s)$, which implies that 
        \begin{align*}
           & \min_{i\in[0:s-1], \alpha \in \mathbb{R}} P_s(i,\alpha)  \leq H(i^*_s) + \frac{1}{2}  \|\alpha^*_{s}{\bf 1} - z_{i_s^*+1:s}\|^2 + \sum_{j=i^*_s+1}^s \omega_j(\alpha^*_{s}) + \lambda_1 \\
           & \leq  h_{i_{s}^*}(y^*_{1:i^*_s};z_{1:i^*_s})  + \frac{1}{2}  \|y^*_{i^*_s+1:s} - z_{i_s^*+1:s}\|^2 + \sum_{j=i^*_s+1}^s \omega_j(y^*_j) + \lambda_1 = h_s(y^*;z_{1:s}) = H(s), 
        \end{align*}
        where the first equality holds by $y^*_{i^*_s+1} \neq y^*_{i^*_s}$ and the expression of $h_s(y^*;z_{1:s})$. If $i^*_s = 0$, 
        \begin{align*}
        \min_{i\in[0:s-1], \alpha \in \mathbb{R}} P_s(i,\alpha) \leq H(0) + \frac{1}{2}  \|y^* - z_{1:s}\|^2 + \sum_{j=1}^s \omega_j(y^*_j) + \lambda_1   = H(s).
        \end{align*}
        Therefore, $\min_{i\in[0:s-1], \alpha \in \mathbb{R}} P_s(i,\alpha) \leq H(s)$ holds.
        On the other hand, let $(i_s^*,\alpha^*_s)$ be an optimal solution to $\min_{i\in[0:s-1],\alpha\in\mathbb{R}} P_s(i,\alpha)$. If $i_s^*\neq 0$, let  $y^* \in \mathbb{R}^s$ be such that $y^*_{1:i_s^*} \in \mathop{\arg\min}_{v\in\mathbb{R}^{i_s^*}} h_{i_s^*}(v;z_{1:i_s^*})$ and $y^*_{i_s^*+1:s}= \alpha_s^* {\bf 1}$. Then, it is clear that 
        \begin{align*}
            H(s) &\leq h_s(y^*;z_{1:s}) \leq h_{i_s^*}(y^*_{1:i_s^*};z_{1:i_s^*}) + \frac{1}{2}  \|y^*_{i_s^*+1:s} - z_{i_s^*+1:s}\|^2 + \sum_{j=i_s^*+1}^s \omega_j(y^*_j) + \lambda_1 \\
             & = H(i_s^*) + \frac{1}{2}  \|\alpha_s^*{\bf 1} - z_{i_s^*+1:s}\|^2 + \sum_{j=i_s^*+1}^s \omega_j(\alpha_s^*) + \lambda_1 = \min_{i\in [0:s-1],\alpha\in\mathbb{R}} P_s(i,\alpha).
        \end{align*}
        If $i_s^* = 0$, let $y^* = \alpha_s^*{\bf 1}$. We have 
        $$ H(s) \leq h_s(y^*;z_{1:s}) = H(0)+ \frac{1}{2}  \|y^* - z_{1:s}\|^2 + \sum_{j=1}^s \omega_j(\alpha_s^*) +\lambda_1 = \min_{i\in [0:s-1],\alpha\in\mathbb{R}} P_s(i,\alpha).$$
        Therefore, $H(s) \leq \min_{i\in [0:s-1],\alpha\in\mathbb{R}} P_s(i,\alpha)$.
        These two inequalities imply the result.

        \noindent
        {\bf (ii)} If $i_s^* \neq 0$, by part (i) and the definitions of $\alpha_s^*$ and $i_s^*$, 
        \begin{align*}
            H(s) & = \min_{i\in[0:s-1],\alpha \in \mathbb{R}} P_s(i,\alpha) = H(i_s^*) + \frac{1}{2}  \|\alpha_s^*{\bf 1} - z_{i_s^*+1:s}\|^2 + \sum_{j=i_s^*+1}^s \omega_j(\alpha_s^*) +\lambda_1 \\
            & = h_{i^*_s}(y^*_{1:i_s^*};z_{1:i_s^*})  + \frac{1}{2}  \|y^*_{i_s^*+1:s} - z_{i_s^*+1:s}\|^2 + \sum_{j=i_s^*+1}^s \omega_j(y^*_j) +\lambda_1 \geq h_{s}(y^*, z_{1:s}),
        \end{align*}
        where the last inequality follows by the definition of $h_s(\cdot, z_{1:s})$. If $i_s^* = 0$, 
        $$H(s) = \min_{i\in[0:s-1],\alpha \in \mathbb{R}} P_s(i,\alpha) = H(0) + \frac{1}{2}  \|y^* - z_{1:s}\|^2 + \sum_{j=1}^s \omega_j(y^*_j) +\lambda_1 = h_s(y^*;z_{1:s}),$$
        Therefore, $H(s) \geq h_s(y^*;z_{1:s}).$ Along with the definition of $H(s)$, $H(s) = h_s(y^*;z_{1:s})$. 
 \end{proof}

 From Lemma \ref{prop-transform} (i), the nonconvex nonsmooth problem \eqref{fusedl0-prox} can be recast as a mixed-integer programming  with objective function given in \eqref{Ps-fun}. Lemma \ref{prop-transform} (ii) suggests a recursive method to obtain an optimal solution to \eqref{fusedl0-prox}. In fact, by setting $s=n$, we can obtain that there exists an optimal solution to \eqref{fusedl0-prox}, says $x^*$, such that $x^*_{i_n^*+1:n} = \alpha_n^* {\bf 1}$, and $x^*_{1:i_n^*}\in \mathop{\arg\min}_{v\in \mathbb{R}^{i_n^*}} h_{i_n^*}(v;z_{1:i_n^*})$. Next, by setting $s=i_n^*$, we are able to obtain the expression of $x^*_{i_{i_n^*}^*+1:i_n^*}$. Repeating this loop backward until $s = 0$,  we can obtain the full expression of an optimal solution to \eqref{fusedl0-prox}. The outline of computing ${\rm prox}_{\lambda_1\|\widehat{B}\cdot\|_0 + \omega(\cdot)}(z)$ is shown as follows.
 \begin{numcases}{}\label{prox-overview}
 {\rm Set\ the\ current\ changepoint}\ s = n. \notag \\{\bf\rm  While} \ s > 0 \  {\bf\rm do} \notag\\ \quad \quad {\rm Find} \ (i_{s}^*, \alpha_{s}^*) \in \mathop{\arg\min}_{i\in [0:s-1], \alpha\in\mathbb{R}} P_{s}(i,\alpha).\\
 \quad \quad {\rm Let} \ x^*_{i_s^*+1:s} = \alpha_{s}^*{\bf 1} \ {\rm and} \ s \gets i_{s}^*.\notag\\   {\rm End}\notag
 \end{numcases}
 
    To obtain an optimal solution to \eqref{fusedl0-prox}, the remaining issue is how to execute the first line in {\bf while} loop of \eqref{prox-overview}, or in other words, for any given $s\in[n]$, how to find $(i^*_s,\alpha_s^*)\in\mathbb{N}\times \mathbb{R}$ appearing in Lemma \ref{prop-transform} (ii). The following proposition provides some preparations.
	\begin{proposition}\label{prop-Ps}
		For each $s\in [n]$, let $P_{s}^*(\alpha):= \min_{i\in[0:s-1]}P_s(i,\alpha)$. 
		\begin{description}
			\item[(i)]  For all $\alpha\in\mathbb{R}$, 
          $$P_s^*(\alpha) = \left\{ \begin{array}{ll}
            \frac{1}{2}(\alpha-z_1)^2 + \omega_1(\alpha) & {\rm if} \ s=1, \\
             \min\Big\{ P_{s-1}^*(\alpha), \min_{\alpha'\in\mathbb{R}} P_{s-1}^*(\alpha') \!+\! \lambda_1\Big\} \!+\! \frac{1}{2}(\alpha\! -\! z_s)^2\! +\! \omega_s(\alpha)  & {\rm if} \ s\in [2\!:\!n].
          \end{array}\right.$$   
  
			\item[(ii)] Let $\mathcal{R}_1^0\!:=\mathbb{R}$, and $\mathcal{R}_s^{i}\!:=\mathcal{R}_{s-1}^{i} \cap (\mathcal{R}_s^{s-1})^c$
			for all $s\in[2\!:\!n]$ and $i\in[0\!:\!s\!-\!2]$, where 
			\begin{equation}\label{Rss-1}
				\mathcal{R}_s^{s-1} := \left\{ \alpha\in\mathbb{R} \ \lvert \ P_{s-1}^*(\alpha) \geq \min_{\alpha'\in\mathbb{R}} P_{s-1}^*(\alpha') + \lambda_1 \right\}.
			\end{equation}
			(a) For each $s\in [2\!:\!n]$, $\bigcup_{i\in [0:s-1]}\mathcal{R}_{s}^i = \mathbb{R}$ and $\mathcal{R}_s^i \cap \mathcal{R}_s^j = \emptyset$ for any $i\ne j \in [0\!:\!s\!-\!1]$. \\
           (b) For each $s\in[n]$ and $i\in [0\!:\!s\!-\!1]$,  $P_s^*(\alpha) = P_s(i,\alpha)$ when $\alpha \in \mathcal{R}_s^{i}$.
		\end{description} 
	\end{proposition}
	\begin{proof}
		{\bf (i)} Note that $P_1^*(\alpha) = P_1(0,\alpha) = H(0) + \frac{1}{2}(\alpha - z_1)^2 + \omega_1(\alpha) + \lambda_1 = \frac{1}{2}(\alpha - z_1)^2 + \omega_1(\alpha).$  
		Now fix any $s\in [2:n]$. By the definition of $P_s^*$, for any $\alpha\in\mathbb{R}$,
		\begin{equation}\label{induction-Ps}
			P_s^*(\alpha) = \min_{i\in[0:s-1]} P_s(i,\alpha) = \min \Big\{ \min_{i \in [0:s-2]} P_{s}(i,\alpha),\,P_s(s\!-\!1,\alpha)\Big\}.
		\end{equation}
		From the definition of $P_s$ in \eqref{Ps-fun}, for each $i\in[0\!:\!s\!-\!2]$ and $\alpha \in \mathbb{R}$, it holds that
		\begin{align*}
			P_s(i,\alpha)  
			&= H(i) + \frac{1}{2} \|\alpha {\bf 1} - z_{i+1:s}\|^2 + \sum_{j=i+1}^s\omega_j(\alpha)+\lambda_1 \\
			&= H(i) + \frac{1}{2} \|\alpha {\bf 1} \!-\! z_{i+1:s-1}\|^2 + \sum_{j=i+1}^{s-1}\omega_j(\alpha)+\lambda_1+\frac{1}{2}( \alpha\!-\!z_s)^2+\omega_s(\alpha) \\
			&= P_{s-1}(i,\alpha) + \frac{1}{2}(\alpha-z_s)^2+\omega_s(\alpha),
		\end{align*}
	   while for any $\alpha\in\mathbb{R}$,
		\(
		P_s(s\!-\!1,\alpha) = H(s\!-\!1) + \frac{1}{2}(\alpha -z_s)^2 +\omega_s(\alpha)+ \lambda_1 .
		\) 
		By combining the last two equalities with \eqref{induction-Ps}, we immediately obtain that 
		\begin{equation}\label{expression-P}
			\begin{aligned}
				P_s^*(\alpha) 
				& = \min\Big\{ \min_{i\in [0:s-2]} P_{s-1}(i,\alpha),\, H(s\!-\!1) +\lambda_1\Big\} + \frac{1}{2}(\alpha -z_s)^2 + \omega_s(\alpha) \\
				& =\min\Big\{P_{s-1}^*(\alpha),  \min_{\alpha'\in\mathbb{R}}P_{s-1}^*(\alpha') +\lambda_1\Big\} + \frac{1}{2}(\alpha -z_s)^2 + \omega_s(\alpha),
			\end{aligned}
		\end{equation}
		where the last equality follows by Lemma \ref{prop-transform} (i). Thus, we get the desired result.
		
		\noindent
		{\bf (ii)} We first prove (a) by induction. When $s = 2$, since $\mathcal{R}_1^0 = \mathbb{R}$ and $\mathcal{R}_2^0 = \mathcal{R}_1^0 \cap (\mathcal{R}_2^1)^c$, we have $\mathcal{R}_2^0 \cup \mathcal{R}_2^1 = \mathbb{R}$ and $\mathcal{R}_2^0 \cap \mathcal{R}_2^1 = \emptyset.$ Assume that the result holds when $s = j$ for some $j \in [2\!:\!n\!-\!1]$. We consider the case $s = j\!+\!1$. Since $\mathcal{R}_s^{i}\!:=\mathcal{R}_{s-1}^{i} \cap (\mathcal{R}_s^{s-1})^c$
		for all $i\in[0\!:\!s\!-\!2]$ and $\bigcup_{i\in[0:s-2]} \mathcal{R}_{s-1}^{i} = \mathbb{R}$, it holds that 
		\begin{align*}
		{\textstyle\bigcup_{i\in[0:s-1]}} \mathcal{R}_{s}^{i} & =\big[{\textstyle \bigcup_{i\in[0:s-2]} (\mathcal{R}_{s-1}^{i} \cap (\mathcal{R}_s^{s-1})^c)}\big] \cup  \mathcal{R}_{s}^{s-1} = \big(\mathbb{R} \cap (\mathcal{R}_s^{s-1})^c\big)\cup  \mathcal{R}_{s}^{s-1} = \mathbb{R}.
		\end{align*}
		Thus we obtain the first part of (a) by deduction. For any $i \in [0\!:\!s\!-\!2]$, by definition, $\mathcal{R}_s^{i}\cap \mathcal{R}_s^{s-1} = \emptyset$. It suffices to show that $\mathcal{R}_s^{i}\cap \mathcal{R}_s^{j} = \emptyset$ for any $i\ne j \in [0\!:\!s\!-\!2]$. By definition, 
		$$\mathcal{R}_s^{i}\cap \mathcal{R}_s^{j} = \big[\mathcal{R}_{s-1}^i \cap (\mathcal{R}^{s-1}_s)^c\big] \cap \big[\mathcal{R}_{s-1}^j \cap (\mathcal{R}^{s-1}_s)^c\big] = \emptyset,$$
		where the last equality is using $\mathcal{R}_{s-1}^i \cap\mathcal{R}_{s-1}^j = \emptyset.$ Thus, the second part of (a) is obtained. 
		
		Next we prove (b). Since for any $\alpha \in \mathbb{R} = \mathcal{R}_1^0$, $P_1^*(\alpha) = P_1(0,\alpha)$, the result holds for $s= 1$. For $s \in [2\!:\!n]$ and $i\!=\!s\!-\!1$, by the definition of $\mathcal{R}_s^{s-1}$, for all $\alpha \in \mathcal{R}_s^{s-1}$, 
		\[
			P_s^*(\alpha)  = \min_{\alpha'\in\mathbb{R}} P_{s-1}^*(\alpha') + \lambda_1 + \frac{1}{2}(\alpha - z_s)^2 + \omega_s(\alpha)  = P_s(s-1,\alpha),
		\]
		where the second equality is using Lemma \ref{prop-transform} (i) and the definition of $P_s$.  Next we consider $s\in[2\!:\!n]$ and $i \in [0\!:\!s\!-\!2]$.
		We argue by induction that $P_s^*(\alpha) = P_s(i,\alpha)$ when $\alpha \in \mathcal{R}_s^{i}$. Indeed, when $s=2$, since $\mathcal{R}_2^0 = \mathcal{R}_1^0\cap (\mathcal{R}_2^1)^c = (\mathcal{R}_2^1)^c$, for any $\alpha\in\mathcal{R}_2^0$, from \eqref{Rss-1} we have $P_{1}^*(\alpha)<\min_{\alpha'\in\mathbb{R}}P_1^*(\alpha')+\lambda_1$, which by part (i) implies that $P_2^*(\alpha) =P_{1}^*(\alpha)+\frac{1}{2}(\alpha-z_2)^2+\omega_2(\alpha)=P_{1}(0,\alpha)+\frac{1}{2}(\alpha-z_2)^2+\omega_2(\alpha)=P_2(0,\alpha)$. Assume that the result holds when $s=j$ for some $j\in [2\!:\!n\!-\!1]$. We consider the case for $s=j\!+\!1$. For any $i\in [0\!:\!s\!-\!2]$, by definition, $\mathcal{R}_{s}^i = \mathcal{R}_{s-1}^i \cap (\mathcal{R}_{s}^{s-1})^c$. Then, from \eqref{expression-P} for any $\alpha \in \mathcal{R}_{s}^i$, 
		\begin{align*}
			P_{s}^*(\alpha) & = P_{s-1}^*(\alpha) + \frac{1}{2}(\alpha -z_{s})^2 + \omega_{s}(\alpha) = P_{s-1}(i,\alpha)+ \frac{1}{2}(\alpha -z_{s})^2 + \omega_{s}(\alpha)\\
			& = H(i)+ \frac{1}{2}\|\alpha {\bf 1} - z_{i+1:s-1}\|^2 + \sum_{j=i+1}^{s-1}w_{j}(\alpha) +\lambda_1  + \frac{1}{2}(\alpha\!-\!z_{s})^2 + \omega_{s}(\alpha) \\
            & = H(i) + \frac{1}{2}\|\alpha {\bf 1} - z_{i+1:s}\|^2 + \sum_{j=i+1}^{s}w_{j}(\alpha)+\lambda_1  = P_{s}(i,\alpha),
		\end{align*}
        where the second equality is using $P_{s-1}^*(\alpha)=P_{s-1}(i,\alpha)$ implied by induction. Hence, the conclusion holds for $s=j+1$ and any $i\in[0\!:\!s\!-\!2]$. The proof is completed.
	\end{proof}
	
   Now we take a closer look at Proposition \ref{prop-Ps}. Part (i) provides a recursive method to compute $P_s^*(\alpha)$ for all $s\in[n]$. For each $s\in [n]$, by the expression of $\omega_s$, $P_s(i,\cdot)$ is a piecewise lower semicontinuous linear-quadratic function whose domain is a closed interval, relative to which $P_s(i,\cdot)$ has an expression of the form $H(i)+ \frac{1}{2}\|\alpha {\bf 1} - z_{i+1:s}\|^2 + (s-i)|\alpha|_0 +\lambda_1$, while $P_s^*(\cdot)=\min\{P_s(0,\cdot),P_s(1,\cdot),\ldots,P_{s}(s-1,\cdot)\}$. Note that for each $i\in [0\!:\!s\!-\!1]$, the optimal solution to $\min_{\alpha\in\mathbb{R}} P_{s}(i,\alpha)$ is easily obtained (in fact, all the possible candidates of the global solutions are $0, \frac{\sum_{j=i+1}^s z_{j}}{s-i}, \max_{j\in[i+1:s]} \{l_{j}\}, \min_{j\in[i+1:s]} \{u_{j}\}$), so is $\mathop{\arg\min}_{\alpha'\in\mathbb{R}} P_s^*(\alpha')$. Part (ii) suggests a way to search for $i_s^*$ such that $P_s^*(\alpha_s^*) = P_s(i_s^*,\alpha_s^*)$ for each $s\in [n]$. Obviously, $P_s(i_s^*,\alpha_s^*) = \min_{i\in[0:s-1],\alpha\in\mathbb{R}}P_s(i,\alpha).$ This inspires us to propose Algorithm \ref{pruning} for solving ${\rm prox}_{\lambda_1 \|\widehat{B}\cdot\|_0 + \omega(\cdot)}(z)$, whose iterate steps are described as follows.
	
	\begin{algorithm}[!h]
		\caption{ (Computing ${\rm prox}_{\lambda_1 \|\widehat{B}\cdot\|_0 + \omega(\cdot)}(z)$) }\label{pruning}  
		1. \textbf{Initialize:} Compute $P_1^*(\alpha)= \frac{1}{2}(z_1-\alpha)^2 + \omega_1(\alpha)$ and set $\mathcal{R}_1^0 = \mathbb{R}.$ 
		
		2. {\bf For} $s= 2,\ldots,n\ $  {\bf do}
		
		3. \quad $P_s^*(\alpha):=\min\{ P_{s-1}^*(\alpha), \min_{\alpha'\in\mathbb{R}} P_{s-1}^*(\alpha') + \lambda_1\} +\frac{1}{2}(\alpha- z_s)^2 + \omega_s(\alpha).$
		
		4. \quad Compute $\mathcal{R}_{s}^{s-1}$ by \eqref{Rss-1}.
		
		5. \quad {\bf For} \ $i = 0,\ldots, s-2$ {\bf do}
		
		6. \quad \quad $\mathcal{R}_s^{i} = \mathcal{R}_{s-1}^{i} \cap (\mathcal{R}_s^{s-1})^c$.
		
		7. \quad {\bf End}
		
		8. {\bf End}
		
		9. Set the current changepoint $s = n$.
		
		10. {\bf While}  $s > 0$  {\bf do}
		
		11.\quad  Find $\alpha_s^* \in \mathop{\arg\min}_{\alpha \in \mathbb{R}} P_{s}^*(\alpha)$, and  $i_s^* = \left\{i\ \lvert\ \alpha_s^* \in \mathcal{R}_{s}^{i} \right\}$.
		
		12. \quad $x^*_{i_s^*+1:s} = \alpha_s^*{\bf 1}$ and $s\gets i_s^*.$
		
		13. {\bf End}
	\end{algorithm}
	
   The main computation cost of Algorithm \ref{pruning} comes from lines 3 and 6, in which the number of pieces of the linear-quadratic functions involved in $P_s^*$ plays a crucial role. The following lemma gives a worst-case estimation for it in the $s$-th iterate.
 \begin{lemma}\label{lemma-pieceofPs}
  Fix any $s\in [2:n]$. The function $P_s^*$ in line 3 of Algorithm \ref{pruning} has at most $O(s^{1+o(1)})$ linear-quadratic pieces.
 \end{lemma}
 \begin{proof}
  Let $h_{i}(\alpha):= H(i)+\frac{1}{2}\|\alpha{\bf 1}-z_{i+1:s}\|^2 + \lambda_1 +(s-i)\lambda_2|\alpha|_0 + \sum_{j=i+1}^{s}\delta_{[l_j,u_j]}(\alpha)$ for $\alpha\in\mathbb{R}$ with $i\in[0\!:\!s\!-\!2]$. For each $i\in [0\!:\!s\!-\!2]$, $h_i$ is a piecewise lower semicontinuous linear-quadratic function whose domain is a closed interval, and every piece is continuous on the closed interval except $\alpha = 0$. Therefore, for each $i\in[0\!:\!s\!-\!2]$, $h_i=\min\big\{h_{i,1},h_{i,2},h_{i,3}\big\}$ with $h_{i,1}(\alpha) := h_i(\alpha)-(s-i)\lambda_2|\alpha|_0 + (s-i)\lambda_2 +\delta_{(-\infty, 0]}(\alpha),\,h_{i,2}(\alpha):=h_i(\alpha) +\delta_{\{0\}}(\alpha)$ and $h_{i,3}(\alpha) := h_i(\alpha)-(s-i)\lambda_2|\alpha|_0 + (s-i)\lambda_2 +\delta_{[0, \infty)}(\alpha)$. Obviously, $h_{i1},h_{i,2}$ and $h_{i,3}$ are  piecewise linear-quadratic functions with domain being a closed interval. In addition, write 
  $h_{s-1}(\alpha) := \min_{\alpha'\in \mathbb{R}} P_{s-1}^*(\alpha')+ \frac{1}{2}(\alpha-z_s)^2 + \lambda|\alpha|_0 + \lambda_1+ \delta_{[l_s, u_s]}(\alpha)$ for $\alpha\in\mathbb{R}$. Obviously, $h_{s-1}$ is a piecewise lower semicontinuous linear-quadratic function whose domain is a closed interval. From the above arguement, $h_{s-1}=\min\{h_{s-1,1},h_{s-1,2},h_{s-1,3}\}$ where each $h_{s-1,j}$ for $j=1,2,3$ is a piecewise linear function whose domain is a closed interval. Combining the above discussion with  line 3 of Algorithm \ref{pruning} and the definition of $P_{s-1}^*$, for any $\alpha\in\mathbb{R}$,  
 \begin{align*}
  P_s^*(\alpha)\!&=\min_{i\in [0:s-2]}\Big\{ P_{s-1}(i, \alpha)\!+\! \frac{1}{2}(\alpha-z_s)^2\!+\!\omega_s(\alpha), \min_{\alpha'\in \mathbb{R}} P_{s-1}^*(\alpha')\!+\! \frac{1}{2}(\alpha-z_s)^2\!+\!\omega_s(\alpha) \!+\!\lambda_1 \Big\}\\
  &=\Big\{h_0(\alpha),h_1(\alpha),\ldots,h_{s-2}(\alpha),h_{s-1}(\alpha)\Big\}=\min_{i\in[0:s-1], j\in [3]}\big\{h_{i,j}(\alpha)\big\}.
 \end{align*}
  Notice that any $h_{i,j}$ and $h_{i',j'}$ with $i\ne i'\in[0:s\!-\!1]$ or $j\ne j'\in[3]$ crosses at most $2$ times. From \cite[Theorem 2.5]{sharir88} the maximal number of linear-quadratic pieces involved in $P_s^*$ is bounded by the maximal length of a $(3s,4)$ Davenport-Schinzel sequence, which by \cite[Theorem 3]{davenport65} is $3c_1 s \exp(c_2\sqrt{\log 3s})$. Here,  $c_1, c_2$ are positive constants independent of $s$. Thus, we conclude that the maximal number of linear-quadratic pieces involved in $P_s^*$ is $O(s^{1+o(1)})$. The proof is finished.
  \end{proof}

  By invoking Lemma \ref{lemma-pieceofPs}, we are able to provide a worst-case estimation for the complexity of Algorithm \ref{pruning}. Indeed, the main cost of Algorithm \ref{pruning} consists in lines 3 and 5-7. Since line 3 involves the computation cost proportional to the pieces of $P_{s-1}^*$, from Lemma \ref{lemma-pieceofPs}, it requires $O(s^{1+o(1)})$ operation. For each $i\in[0\!:\!s\!-\!1]$, from part (b) of Proposition \ref{prop-Ps} (ii), we know that $\mathcal{R}_{s}^{i}$ consists of at most $O(s^{1+o(1)})$ intervals, which means that line 6 requires at most $O(s^{1+o(1)})$ operations and then the complexity of lines 5-7 is $O(s^{2+o(1)})$. Thus, the worst-case complexity of Algorithm \ref{pruning} is $\sum_{s=2}^n O(s^{2+o(1)})=O(n^{3+o(1)})$.
\section{A hybrid of PG and inexact projected regularized Newton methods}\label{sec4}

 In the hybrid frameworks owing to \cite{Themelis18} and \cite{bareilles22}, the PG and Newton steps are alternating. We now state the details of our algorithm, a hybrid of PG and inexact projected regularized Newton methods (PGiPN), for solving problem \eqref{model}, where the introduction of the switch condition \eqref{switch-condition} is due to the consideration that the PG step is more cost-effective than the Newton step when the iterates are far from a stationary point. Let $x^k\in\Omega$ be the current iterate.  It is noted that the PG step is always executed and if condition \eqref{switch-condition} is met, we need to solve \eqref{subp}, which involves constructing $G_k$ to satisfy \eqref{con-Gk}. Such $G^k$ can be easily achieved in the following situations.
 
 For some generalized linear models, $f$ can be expressed as $f(x) = h(Ax-b)$ for some $A\in\mathbb{R}^{m\times n}$, $b\in\mathbb{R}^m$ and twice continuously differentiable, separable $h$. For this case, $\nabla^2h$ is a diagonal matrix, and $\nabla^2 \!f(x) = A^{\top}\nabla^2h(Ax-b)A$. Since $\nabla^2\!f(x^k)$ is not necessarily positive definite, following the method in \cite{liu22}, we construct $G_k := G_k^1$, where
\begin{equation}\label{rnm_liu}
 G_k^1:=\nabla^2 f(x^k)\!+b_2 [-\lambda_{\min}(\nabla^2\!h(Ax^k-b))]_+A^{\top} A + b_1 \|\overline{\mu}_k(x^k\!-\!\overline{x}^k)\|^{\sigma} I
 \end{equation}
 with $b_2 \geq 1$. However, for highly nonconvex $h$, $[-\lambda_{\min}(\nabla^2\!h(Ax^k-b))]_+$ is large, for which $G_k^1$ is a poor approximation to $\nabla^2 f(x^k)$. To avoid this drawback and simultaneously make $G_k$ positive definite, we consider $G_k:=G_k^2$, where
\begin{equation}\label{rnm_new}
    G_k^2:= A^{\top}[\nabla^2\!h(Ax^k-b)]_+ A + b_1 \|\overline{\mu}_k(x^k\!-\!\overline{x}^k)\|^{\sigma} I.
\end{equation}
For the case where $\nabla^2 \!f(x^k) \succeq 0$, $G_k^1 = G_k^2$. If $\nabla^2\!f(x^k) \not\succeq 0$, it is immediate to see that $\|G_k^1 - \nabla^2\! f(x^k)\|_2 \geq \|G_k^2 - \nabla^2\! f(x^k)\|_2$, which means that $G_k^2$ is a better approximation to $\nabla^2 f(x^k)$ than $G_k^1$.
On the other hand, for those $f$'s not owning a separable structure, we form $G_k:=G_k^3$ as in \cite{Ueda10} and \cite{wu22}, where
\begin{equation}\label{rnm_old}
    G_k^3:=\nabla^2 f(x^k)\!+\big(b_2[ -\lambda_{\min}(\nabla^2 f(x^k))]_+ + b_1 \|\overline{\mu}_k(x^k\!-\!\overline{x}^k)\|^{\sigma}\big) I.  
\end{equation}
 It is not hard to check that for $i=1,2,3$, $G_k^i$ meets the requirement in \eqref{con-Gk}. We remark here that the sequel convergence analysis holds for all three $G_k^i$, and we write them by $G_k$ for simplicity.
  
The iterates of PGiPN are described as follows.
 \begin{algorithm}[!ht]
  \caption{(a hybrid of PG and inexact projected regularized Newton methods)}\label{hybrid}
  \textbf{Initialization:} Choose $\epsilon\geq 0$ and parameters $\mu_{\max}>\mu_{\min}>0,\,\tau > 1,\,\alpha>0,b_1>0, b_2\geq 1$, $\varrho \in (0,\frac{1}{2}),\,\sigma\in (0,\frac{1}{2}),\,\varsigma\in(\sigma, 1]$ and $\beta\in(0,1)$. Choose an initial $x^0\in\Omega$ and let $k:=0$. 
	
  \medskip
  {\bf PG Step:}
  	
   (1a) Select $\mu_k \in [\mu_{\min}, \mu_{\max}]$. Let $m_k$ be the smallest nonnegative integer $m$ such that 
    \begin{equation}\label{ls-PG}
     F(\overline{x}^k)\le F(x^k)-\frac{\alpha}{2}\|x^k\!-\!\overline{x}^k\|^2\ \ {\rm with}\ \
     \overline{x}^k\in {\rm prox}_{(\mu_k \tau^m)^{-1}g}(x^k\!-\!(\mu_k \tau^m)^{-1}\nabla f(x^k)).
    \end{equation} 
   
   (1b) Let $\overline{\mu}_k=\mu_k \tau^{m_k}$. If $\overline{\mu}_k\|x^k - \overline{x}^k\| \leq \epsilon,$ stop and output $x^k$; otherwise, go to step (1c). 
	
   (1c) If condition \eqref{switch-condition} holds, go to Newton step; otherwise, let $x^{k+1} = \overline{x}^k$. Set $k \gets k+1$ and return to step (1a).

   \medskip
   {\bf Newton step:}
	
   (2a) Seek an inexact solution $y^k$ of \eqref{subp} satisfying \eqref{inexact-cond1}-\eqref{inexact-cond2}. 
	
   (2b) Set $d^k:= y^k-x^k$. Let $t_k$ be the smallest nonnegative integer $t$ such that
	\begin{equation}\label{ls-NT}
		f(x^k+\beta^{t}d^k) \leq f(x^k) + \varrho\beta^t\langle \nabla f(x^k), d^k\rangle.
	\end{equation}
	
   (2c) Let $\alpha_k = \beta^{t_k}$ with $x^{k+1}= x^k\!+\!\alpha_k d^k$. Set $k \gets k+1$ and return to PG step.    
 \end{algorithm}	
 \begin{remark}\label{remark-hybrid}
  {\bf (i)} Our PGiPN benefits from the PG step in two aspects. First, the incorporation of the PG step can guarantee that the sequence generated by PGiPN remains in a right position for convergence. Second, the PG step helps to identify adaptively the subspace used in the Newton step, and as will be shown in Proposition \ref{prop-supp}, switch condition \eqref{switch-condition} always holds and the supports of $\{Bx^k\}_{k\in\mathbb{N}}$ and $\{x^k\}_{k\in\mathbb{N}}$ keep unchanged when $k$ is sufficiently large, so that Algorithm \ref{hybrid} will reduce to an inexact projected regularized Newton method for solving \eqref{formulated-prob} with $\Pi_k \equiv \Pi_*$. In this sense, the PG step plays a crucial role in transforming the original challenging problem \eqref{model} into a problem that can be efficiently solved by the inexact projected regularized Newton method. 

  \noindent
  {\bf (ii)} When $x^k$ enters the Newton step, from the inexact criterion \eqref{inexact-cond1} and the expression of $\Theta_k$, $ 0 \ge\Theta_k(x^k\!+\!d^k)-\Theta_k(x^k) = \langle \nabla f(x^k), d^k\rangle + \frac{1}{2}\langle d^k, G_kd^k\rangle$, and then 
  \begin{equation}\label{convex_subgrad}
   \langle \nabla f(x^k), d^k\rangle \le -\frac{1}{2} \langle d^k, G_kd^k\rangle \leq -\frac{b_1}{2} \|\overline{\mu}_k(x^k - \overline{x}^k)\|^{\sigma}\|d^k\|^2 < 0,
  \end{equation} 
  where the second inequality is due to \eqref{con-Gk}. In addition, the inexact criterion \eqref{inexact-cond1} implies that $y^k\in\Pi_k$, which along with $x^k\in \Pi_k$ and the convexity of $\Pi_k$ yields that $x^{k} + \alpha d^k\in \Pi_k$ for any $\alpha\in (0,1]$. By the definition of $\Pi_k$, ${\rm supp}(B(x^{k} + \alpha d^k)) \subset {\rm supp}(Bx^k)$ and ${\rm supp}(x^{k} + \alpha d^k) \subset {\rm supp}(x^k)$, so $g(x^{k} + \alpha d^k) \leq g(x^k)$ for any $\alpha\in (0,1]$. This together with \eqref{convex_subgrad} shows that the iterate along the direction $d^k$ will reduce the value of $F$ at $x^k$.

  \noindent
  {\bf(iii)} When $\epsilon=0$, by Definition \ref{def1-Spoint} the output $x^k$ of Algorithm \ref{hybrid} is an $L$-stationary point of \eqref{model}, which is also a stationary point of problem \eqref{formulated-prob} from Proposition \ref{prop-opt} and Lemma \ref{subdiff-hpconvex} (i). Let $r_k\!:\mathbb{R}^n\to\mathbb{R}^n$ be the KKT residual mapping of \eqref{formulated-prob} defined by
  \begin{equation}\label{def-rk}
   r_k(x):= \overline{\mu}_k [x - {\rm proj}_{\Pi_k} (x - \overline{\mu}_k^{-1} \nabla f(x))].
  \end{equation}
  It is not difficult to verify that when $x^k$ satisfies condition \eqref{switch-condition}, the following relation holds 
  \begin{equation}\label{Eq-KKTprox}
   r_k(x^k)=\overline{\mu}_k(x^k- \overline{x}^k),
  \end{equation}
  for which it suffices to argue that $\overline{x}^k\!={\rm proj}_{\Pi_k}(x^k-\overline{\mu}_k^{-1}\nabla f(x^k))$. Indeed, if not, there exists $\overline{z}^k\in\Pi_k$ such that $\widetilde{h}_k(\overline{z}^k)< \widetilde{h}_k(\overline{x}^k)$, where $\widetilde{h}_k(x):=\frac{\overline{\mu}_k}{2}\|x-(x^k - \overline{\mu}_k^{-1}\nabla f(x^k))\|^2$. Since $\overline{z}^k \in \Pi_k$, we have ${\rm supp}(B\overline{z}^k) \subset{\rm supp}(B\overline{x}^k)$ and ${\rm supp}(\overline{z}^k) \subset{\rm supp}(\overline{x}^k)$, which implies that $g(\overline{z}^k) \leq g(\overline{x}^k)$ and then  $\widetilde{h}_k(\overline{z}^k) + g(\overline{z}^k) < \widetilde{h}_k(\overline{x}^k) + g(\overline{x}^k)$, a contradiction to $\overline{x}^k\in {\rm prox}_{\overline{\mu}_k^{-1}g}(x^k - \overline{\mu}_k^{-1}\nabla f(x^k))$. 
  
  \noindent
  {\bf (iv)} By using \eqref{fact1} and the descent lemma \cite[Proposition A.24]{NP97}, the line search in step (1a) must stop after a finite number of backtrackings. In fact, the line search in step (1a) is satisfied when $\mu_k \tau^m \geq L_1+\alpha$, which implies that $\overline{\mu}_k < \widetilde{\mu} :=\tau (L_1\!+\alpha)$ for each $k\in\mathbb{N}$.
 \end{remark}

 By Remark \ref{remark-hybrid} (iv), to show that Algorithm \ref{hybrid} is well defined, we only need to argue that the Newton steps in Algorithm \ref{hybrid} are well defined, which is implied by the following lemma. 
 \begin{lemma}\label{well-defined}
  For each $k\in\mathbb{N}$, define the KKT residual mapping $R_k\!:\mathbb{R}^n\to\mathbb{R}^n$ of \eqref{subp} by 
 \[
  R_k(y):= \overline{\mu}_k[y - {\rm proj}_{\Pi_k}(y- 
 \overline{\mu}_k^{-1}(G_k (y-x^k) + \nabla\! f(x^k)))].
 \]
 Then, for those $x^k$'s satisfying \eqref{switch-condition}, the following statements are true.
 \begin{description}
 \item[(i)] For any $y$ close enough to the optimal solution of \eqref{subp},  $y-\overline{\mu}_k^{-1}R_k(y)$ satisfies \eqref{inexact-cond1}-\eqref{inexact-cond2}.
		
 \item[(ii)] The line search step in \eqref{ls-NT} is well defined, and $\alpha_k \geq \min\big\{1,\frac{(1-\varrho)b_1\beta}{L_1}\|\overline{\mu}_k(x^k-\overline{x}^k)\|^{\sigma}\big\}$. 

 \item[(iii)] The inexact criterion \eqref{inexact-cond2} implies that $\|R_k(y^k)\| \leq \frac{1}{2}\min \left\{\|r_k(x^k)\|, \|r_k(x^k)\|^{1+\varsigma} \right\}$.
 \end{description}
 \end{lemma}
 \begin{proof}
 Pick any $x^k$ satisfying \eqref{switch-condition}. We proceed the proof of parts (i)-(iii) as follows.

 \noindent
 {\bf(i)} Let $\widehat{y}^k$ be the unique optimal solution to \eqref{subp}. Then $\widehat{y}^k \neq x^k$ (if not, $x^k$ is the optimal solution of \eqref{subp} and $0=R_k(x^k)=r_k(x^k)$, which by \eqref{Eq-KKTprox} means that $x^k = \overline{x}^k$ and Algorithm \ref{hybrid} stops at $x^k$). By the optimality condition of \eqref{subp}, $-\nabla \!f(x^k)-G_k(\widehat{y}^k-x^k)\in\mathcal{N}_{\Pi_k}(\widehat{y}^k)$, which by the convexity of $\Pi_k$ and $x^k\in\Pi_k$ implies that $\langle\nabla \!f(x^k)+G_k(\widehat{y}^k-x^k),\widehat{y}^k-x^k\rangle\le 0$. Along with the expression of $\Theta_k$, we have $\Theta_k(\widehat{y}^k)-\Theta_k(x^k) \leq -\frac{1}{2}\langle \widehat{y}^k-x^k,G_k(\widehat{y}^k-x^k)\rangle<0$. Since $\Theta_k$ is continuous relative to $\Pi_k$, for any $z\in\Pi_k$ sufficiently close to $\widehat{y}^k$, $\Theta_k(z)-\Theta_k(x^k) \leq 0$. From $R_k(\widehat{y}) = 0$ and the continuity of $R_k$, $y-\overline{\mu}_k^{-1}R_k(y)$ is close to $\widehat{y}$ when $y$ sufficiently close to $\widehat{y}$, which together with $y-\overline{\mu}_k^{-1}R_k(y)\in \Pi_k$ implies that $y-\overline{\mu}_k^{-1}R_k(y)$ satisfies the criterion \eqref{inexact-cond1} when $y$ is sufficiently close to $\widehat{y}$. In addition, from the expression of $R_k$, for any $y\in\mathbb{R}^n$, 
 \[
   0\in G_k(y-x^k)+\nabla\!f(x^k)-R_k(y)+\mathcal{N}_{\Pi_k}(y-\overline{\mu}_k^{-1}R_k(y)),  
 \]
 which by the expression of $\Theta_k$ implies that 
 $\overline{\mu}_k^{-1}G_kR_k(y)+R_k(y)\in\partial\Theta_k(y-\overline{\mu}_k^{-1}R_k(y))$. Hence, ${\rm dist}(0,\partial\Theta_k(y-\overline{\mu}_k^{-1}R_k(y)))\le \|\overline{\mu}_k^{-1}G_kR_k(y)+R_k(y)\|$. Noting that $R_k(\widehat{y}^k)=0$, we have $\|\overline{\mu}_k^{-1}G_kR_k(\widehat{y}^k)+R_k(\widehat{y}^k)\|= 0<\frac{\min \{\overline{\mu}_k^{-1}, 1\}}{2}\min \left\{\|\overline{\mu}_k(x^k\!-\! \overline{x}^k)\|,\|\overline{\mu}_k(x^k\!-\overline{x}^k)\|^{1+\varsigma}\right\}$. From the continuity of the function $y\mapsto \|\overline{\mu}_k^{-1}G_kR_k(y)+R_k(y)\|$, 
 we conclude that for any $y$ sufficiently close to $\widehat{y}^k$, 
 $y-\overline{\mu}_k^{-1}R_k(y)$ satisfies the inexact criterion \eqref{inexact-cond2}. 

 \noindent
 {\bf (ii)} By \eqref{fact1} and the descent lemma \cite[Proposition A.24]{NP97}, for any $\alpha\in(0,1]$, 
 \begin{align*}
  f(x^k\!+\!\alpha d^k) - f(x^k) -\varrho \alpha \langle \nabla \!f(x^k),  d^k\rangle 
  &\le (1\!-\!\varrho) \alpha \langle \nabla \!f(x^k), d^k\rangle + \frac{L_1\alpha^2}{2} \|d^k\|^2 \\
  &\le-\frac{(1\!-\!\varrho) \alpha b_1}{2} \|\overline{\mu}_k(x^k \!- \!\overline{x}^k)\|^{\sigma}\|d^k\|^2\!+\! \frac{L_1\alpha^2}{2} \|d^k\|^2 \\
  &= \Big(-\frac{(1\!-\!\varrho) b_1}{2} \|\overline{\mu}_k(x^k\! -\! \overline{x}^k)\|^{\sigma}\!+\!\frac{L_1\alpha}{2} \Big)\alpha \|d^k\|^2,
 \end{align*} 
 where the second inequality uses \eqref{convex_subgrad}. Therefore, when $\alpha \leq \min\big\{1,\frac{(1-\varrho)b_1}{L_1}\|\overline{\mu}_k(x^k\!-\! \overline{x}^k)\|^{\sigma}\big\}$, the line search in \eqref{ls-NT} holds, which implies that $\alpha_k\ge\min\big\{1,\frac{(1-\varrho)b_1\beta}{L_1}\|\overline{\mu}_k(x^k\! -\! \overline{x}^k)\|^{\sigma}\big\}$.

 \noindent
 {\bf(iii)} Let $\zeta^k\in\partial \Theta_k(y^k)$ be such that $\|\zeta^k\|={\rm dist}(0, \partial \Theta_k(y^k))$. From $\zeta^k\in\partial \Theta_k(y^k)$ and the expression of  $\Theta_k$, we have $y^k={\rm proj}_{\Pi_k}(y^k+\zeta^k-(G_k(y^k-x^k)+\nabla f(x^k)))$. Along with the nonexpansiveness of ${\rm proj}_{\Pi_k}$, $\|y^k\!-\!{\rm proj}_{\Pi_k}(y^k\!-\!(G_k(y^k\!-\!x^k) + \nabla f(x^k)))\|\le\|\zeta^k\|$. Consequently,
 \begin{align*}
 {\rm dist}(0,\partial\Theta_k(y^k)) 
 &\ge\|y^k\!-\!{\rm proj}_{\Pi_k}(y^k\!-\!(G_k(y^k\!-\!x^k) + \nabla f(x^k)))\| 
  \ge\min \{\overline{\mu}_k^{-1}, 1\} \|R_k(y^k)\|,
 \end{align*}
 where the second inequality follows by Lemma 4 of \cite{sra12} and the expression of $R_k$. Combining the last inequality with \eqref{inexact-cond2} and \eqref{Eq-KKTprox} leads to the desired inequality.
\end{proof}

 When $\overline{\mu}_k=1$, the condition that $\|R_k(y^k)\|\le \frac{1}{2}\min\left\{\|r_k(x^k)\|,\|r_k(x^k)\|^{1+\varsigma} \right\}$ is a special case of the first inexact condition in \cite[Equa (6a)]{Yue19} or the inexact condition in \cite[Equa (14)]{Mordu23}, which by Lemma \ref{well-defined} (iii) shows that criterion \eqref{inexact-cond2} with $\overline{\mu}_k=1$ is stronger than those ones.
 
 To analyze the convergence of Algorithm \ref{hybrid} with $\epsilon=0$, henceforth we assume $x^k \neq \overline{x}^k$ for all $k$ (if not, Algorithm \ref{hybrid} will produce an $L$-stationary point within finite number of steps, and its convergence holds automatically). From the iterate steps of Algorithm \ref{hybrid}, we see that the sequence $\{x^k\}_{k\in\mathbb{N}}$ consists of two parts, $\{x^k\}_{k\in\mathcal{K}_1}$ and $\{x^k\}_{k\in\mathcal{K}_2}$, where 
 \[  \mathcal{K}_1\!:=\!\mathbb{N}\backslash\mathcal{K}_2\ \ {\rm with}\ \ \mathcal{K}_2\!:=\!\big\{k\in\mathbb{N}\ \lvert \ {\rm supp}(Bx^k)\!={\rm supp}(B\overline{x}^k),\ {\rm supp}(x^k)\!= {\rm supp}(\overline{x}^k) \big\}.
 \]
 Obviously, $\mathcal{K}_1$ consists of those $k$'s with $x^{k+1}$ from the PG step, while $\mathcal{K}_2$ consists of those $k$'s with $x^{k+1}$ from the Newton step. 

 To close this section, we provide some properties of the sequences $\{x^k\}_{k\in\mathbb{N}}$ and $\{\overline{x}^k\}_{k\in\mathbb{N}}$.
 \begin{proposition}\label{xk-prop}
  The following assertions are true.
  \begin{description}
  \item[(i)] The sequence $\{F(x^k)\}_{k\in\mathbb{N}}$ is descent and convergent.
 
  \item[(ii)] There exists $\nu > 0$ such that $\lvert B\overline{x}^k\rvert_{\min} \geq \nu $ and $\lvert \overline{x}^k\rvert_{\min} \geq \nu$ for all $k\in\mathbb{N}$.
		
  \item[(iii)]  There exist $c_1,c_2 > 0$ such that $c_1 \|r_k(x^k)\| \leq \|d^k\|\leq c_2 \|r_k(x^k)\|^{1-\sigma}$ for all $k\in\mathcal{K}_2$.
 \end{description}
 \end{proposition}
 \begin{proof}
 {\bf (i)} For each $k\in\mathbb{N}$, when $k\in\mathcal{K}_1$, by the line search in step (1a), $F(x^{k+1})<F(x^k)$, and when $k\in\mathcal{K}_2$, from  \eqref{ls-NT} and \eqref{convex_subgrad}, it follows that $f(x^{k+1})<f(x^k)$, which along with $g(x^{k+1})\le g(x^k)$ by Remark \ref{remark-hybrid} (ii) implies that $F(x^{k+1})<F(x^k)$. Hence, $\{F(x^k)\}_{k\in\mathbb{N}}$ is a descent sequence. Recall that $F$ is lower bounded on $\Omega$, so $\{F(x^k)\}_{k\in\mathbb{N}}$ is convergent.
	
 \noindent
 {\bf (ii)}
 By the definition of $\overline{\mu}_k$ and Remark \ref{remark-hybrid} (iv), $\overline{\mu}_k \in [\mu_{\min},\widetilde{\mu})$ for all $k \in \mathbb{N}$. Note that $\{x^k\}_{k\in\mathbb{N}}\subset \Omega$, so the sequence $\{x^k\!-\overline{\mu}_k^{-1}\nabla f(x^k)\}_{k\in\mathbb{N}}$ is bounded and is contained in a compact set, says, $\Xi$. By invoking Lemma \ref{lemma-lb} with such $\Xi$ and $\underline{\mu}=\mu_{\min},\overline{\mu}=\widetilde{\mu}$, there exists $\nu>0$ (depending on $\Xi,\ \mu_{\min}$ and $\widetilde{\mu}$) such that $\lvert [B; I]\overline{x}^k\rvert_{\min} > \nu$. The desired result then follows by noting that $\lvert B\overline{x}^k\rvert_{\min} \ge\lvert [B; I]\overline{x}^k\rvert_{\min}$ and $\lvert \overline{x}^k\rvert_{\min} \ge\lvert [B; I]\overline{x}^k\rvert_{\min}$.
	
  \noindent
  {\bf (iii)} 
  From the definition of $G_k$, the continuity of $\nabla^2\!f$, $\{x^k, \overline{x}^k\}_{k\in\mathbb{N}}\subset \Omega$ and Remark \ref{remark-hybrid} (iv),  there exists $\overline{c}>0$ such that
  \begin{equation}\label{Gk-bound}
      \|G_k\|_2 \leq \overline{c}\ {\rm for\ all}\ k\in\mathcal{K}_2.
  \end{equation}
 Fix any $k\in\mathcal{K}_2$. By Lemma \ref{well-defined} (iii), $\|R_k(y^k)\| \leq \frac{1}{2} \|r_k(x^k)\|$. Then, it holds that 
  \begin{align*}
   & \frac{1}{2} \|r_k(x^k)\| 
   \leq \|r_k(x^k)\| - \|R_k(y^k)\| \leq \|r_k(x^k) - R_k(y^k)\| \\
   & = \overline{\mu}_k\|x^k - {\rm proj}_{\Pi_k} (x^k - \overline{\mu}_k^{-1} \nabla f(x^k)) - y^k + {\rm proj}_{\Pi_k}(y^k - \overline{\mu}_k^{-1}(G_k (y^k-x^k) + \nabla\! f(x^k)))\| \\
  & \leq (2\overline{\mu}_k+\|G_k\|_2)\|y^k - x^k\|  \leq (2\widetilde{\mu}+\overline{c})\|d^k\|,
  \end{align*}
  where the third inequality is using the nonexpansiveness of ${\rm proj}_{\Pi_k}$, and the last one is due to \eqref{Gk-bound} and $d^k=y^k-x^k$. Therefore, $c_1 \|r_k(x^k)\| \leq \|d^k\|$ with $c_1:=1/(4\widetilde{\mu}\!+2\overline{c})$. 
  For the second inequality, it follows from the definitions of $r_k(\cdot)$ and $R_k(\cdot)$ that $R_k(y^k)-\nabla f(x^k)-G_kd^k \in \mathcal{N}_{\Pi_k}(y^k-\overline{\mu}_k^{-1}R_k(y^k))$ and $r_k(x^k) - \nabla f(x^k) \in \mathcal{N}_{\Pi_k}(x^k - \overline{\mu}_k^{-1}r_k(x^k))$, which together with the monotonicity of the set-valued mapping $\mathcal{N}_{\Pi_k}(\cdot)$ implies that
  \begin{align*}
  \langle d^k,G_kd^k\rangle
   &\le \langle R_k(y^k)\!-\!r_k(x^k),d^k\rangle-\overline{\mu}_k^{-1}\|R_k(y^k)\!-\!r_k(x^k)\|^2 -\overline{\mu}_k^{-1}\langle G_kd^k,-R_k(y^k)+r_k(x^k)\rangle \nonumber\\
   &\le \langle(I+\overline{\mu}_k^{-1}G_k)d^k,R_k(y^k)-r_k(x^k)\rangle.
  \end{align*}
  Combining this inequality with equations \eqref{con-Gk}, \eqref{Eq-KKTprox}  and Lemma \ref{well-defined} (iii) leads to 
  \begin{align} \label{Gkub}  
   b_1\|r_k(x^k)\|^{\sigma}\|d^k\|^2
   &\le(1+\overline{\mu}_k^{-1}\|G_k\|_2)(\|R_k(y^k)\|+\|r_k(x^k)\|)\|d^k\|\\   
   &\le(3/2)(1+\overline{\mu}_k^{-1}\|G_k\|_2)\|r_k(x^k)\|\|d^k\|, \nonumber
  \end{align}
  which along with \eqref{Gk-bound} and $\mu_k \geq \mu_{\min}$ implies that $\|d^k\| \leq \frac{3}{2}(1+\mu_{\min}^{-1}\overline{c})b_1^{-1}\|r_k(x^k)\|^{1-\sigma}.$
	Then, $\|d^k\|\leq c_2 \|r_k(x^k)\|^{1-\sigma}$ holds with $c_2 := \frac{3}{2}(1+\mu_{\min}^{-1}\overline{c})b_1^{-1}$. The proof is completed.
 \end{proof}
\section{Convergence Analysis}\label{sec5}

 Before analyzing the convergence of Algorithm \ref{hybrid}, we show that Algorithm \ref{hybrid} finally reduces to an inexact projected regularized Newton method for seeking a stationary point of a problem to minimize a smooth function over a polyhedral set. This requires the following lemma. 
\begin{lemma}\label{lemma-Fxk}
For the sequences $\{x^k\}_{k\in\mathbb{N}}$ and $\{\overline{x}^k\}_{k\in\mathbb{N}}$ generated by Algorithm \ref{hybrid}, the following assertions are true.
 \begin{description}
  \item[(i)] There exists a constant $\gamma>0$ such that for each $k\in\mathbb{N}$,  
		\begin{equation}\label{sufficient-descent}
			F(x^{k+1}) - F(x^k) 
			\le\left\{ \begin{array}{ll}
				-\gamma\|x^k-\overline{x}^k\|^2 &{\rm if}\ k\in\mathcal{K}_1, \\
				-\gamma\|x^k-\overline{x}^k\|^{2+\sigma}&{\rm if}\ k\in\mathcal{K}_{2}, \ \alpha_k = 1, \\
                -\gamma\|x^k-\overline{x}^k\|^{2+2\sigma}&{\rm if}\ k\in\mathcal{K}_{2}, \ \alpha_k \neq 1.
			\end{array}\right.
		\end{equation}
		
 \item[(ii)] $\lim_{k\rightarrow \infty} \|x^k - \overline{x}^k\|= 0$ and $\lim_{\mathcal{K}_2 \ni k\rightarrow \infty} \|d^k\| = 0$.
		
  \item[(iii)] The accumulation point set of $\{x^k\}_{k\in\mathbb{N}}$, denoted by $\Gamma(x^0)$, is nonempty and compact, and every element of $\Gamma(x^0)$ is an $L$-stationary point of problem \eqref{model}.
 \end{description}
 \end{lemma}
 \begin{proof}
 {\bf (i)} Fix any $k\in \mathcal{K}_2$. From inequalities \eqref{ls-NT} and \eqref{convex_subgrad}, 
 \begin{equation}\label{descent-prop}
 \begin{aligned}
  f(x^{k+1})-f(x^k) 
  &\le -\frac{\varrho b_1 \alpha_k}{2}\|\overline{\mu}_k(x^k\!-\! \overline{x}^k)\|^{\sigma}\|d^k\|^2 
   \le -\frac{\varrho c_1^2b_1\alpha_k}{2}\|\overline{\mu}_k(x^k\!-\! \overline{x}^k)\|^{2+\sigma} \\
  & \leq -\frac{\varrho c_1^2b_1\alpha_k\mu_{\min}^{2+\sigma}}{2}\|x^k\!-\!\overline{x}^k\|^{2+\sigma},
 \end{aligned}
 \end{equation}
 where the second inequality is using Proposition \ref{xk-prop} (iii) and equality \eqref{Eq-KKTprox}. 
 By Remark \ref{remark-hybrid} (ii), we have $g(x^{k+1})\leq g(x^k)$, so that $F(x^{k+1}) - F(x^k) \leq f(x^{k+1}) - f(x^k)$, which along with the last equation yields that $F(x^{k+1})\!-F(x^k) \leq -\frac{\varrho c_1^2b_1\alpha_k\mu_{\min}^{2+\sigma}}{2}\|x^k-\overline{x}^k\|^{2+\sigma}$. Take $\gamma\!:= \min\big\{\frac{\alpha}{2},\frac{\varrho c_1^2b_1\mu_{\min}^{2+\sigma}}{2},\frac{\beta(1-\varrho)\varrho c_1^2b_1^2\mu_{\min}^{2+2\sigma}}{2L_1}\big\}$. The desired result then follows by using Lemma \ref{well-defined} (ii) and recalling that $F(x^{k+1}) - F(x^k) \leq \frac{\alpha}{2}\|x^k-\overline{x}^k\|^2$ for $k\in\mathcal{K}_1$. 

 \noindent
 {\bf(ii)} Let $\widetilde{\mathcal{K}}_2 := \{ k\in\mathcal{K}_2 \ | \ \alpha_k = 1 \}$. Doing summation for inequality \eqref{sufficient-descent} from $i=1$ to any $j\in\mathbb{N}$ yields that
 \begin{align*}
   \sum_{i\in \mathcal{K}_1\cap [j] } \gamma\|x^i-\overline{x}^{i}\|^2 & + \sum_{i\in \widetilde{\mathcal{K}}_2\cap[j]} \gamma\|x^i - \overline{x}^{i}\|^{2+\sigma}
  +\sum_{i\in (\mathcal{K}_2\backslash \widetilde{\mathcal{K}}_2)\cap[j]} \gamma\|x^i - \overline{x}^{i}\|^{2+2\sigma} \\
  & \le \sum_{i=1}^j\big[F(x^i)-F(x^{i+1})\big]= F(x^1)-F(x^{j+1}),  
 \end{align*}
  which by the lower boundedness of $F$ on the set $\Omega$ implies that
  \[
    \sum_{i\in \mathcal{K}_1 } \|x^i-\overline{x}^{i}\|^2 +  \sum_{i\in \widetilde{\mathcal{K}}_2} \gamma\|x^i - \overline{x}^{i}\|^{2+\sigma}
  +\sum_{i\in \mathcal{K}_2\backslash \widetilde{\mathcal{K}}_2} \gamma\|x^i - \overline{x}^{i}\|^{2+2\sigma} < \infty. 
  \]
  Thus, we obtain $\lim_{k\rightarrow \infty} \|x^k-\overline{x}^k\|= 0$. 
  Together with Proposition \ref{xk-prop} (iii), \eqref{Eq-KKTprox} and Remark \ref{remark-hybrid} (iv), it follows that $\lim_{\mathcal{K}_2 \ni k\rightarrow \infty} \|d^k\| = 0$. 
	
  \noindent
  {\bf(iii)} Recall that $\{x^k\}_{k\in\mathbb{N}}\subset\Omega$, so its accumulation point set $\Gamma(x^0)$ is nonempty. Pick any $x^*\in\Gamma(x^0)$. Then, there exists an index set $K\subset\mathbb{N}$ such that $\lim_{K\ni k\to\infty}x^{k}=x^*$. From part (ii), $\lim_{K\ni k\to\infty}\overline{x}^{k}=x^*$. For each $k\in K$, $\overline{x}^{k} \in {\rm prox}_{\overline{\mu}_{k}^{-1} g}\big(x^{k}-\overline{\mu}_{k}^{-1}\nabla f(x^{k})\big)$ with $\overline{\mu}_{k}\in[\mu_{\rm min},\widetilde{\mu})$ by step (1a) and Remark \ref{remark-hybrid} (iv). We assume that $\lim_{K\ni k\to\infty}\overline{\mu}_{k}\to\overline{\mu}_*\in[\mu_{\rm min},\widetilde{\mu}]$ (if necessary by taking a subsequence). Define the function
  \[
    h(z,x,\mu)\!:=\!\left\{\begin{array}{cl}
   \frac{\mu}{2}\big\|z-(x\!-\!\mu^{-1}\nabla\!f(x))\big\|^2+g(z) &{\rm if}\ (z,x,\mu)\in\mathbb{R}^n\times\mathbb{R}^n\times[\mu_{\rm min},\widetilde{\mu}],\\
    \infty &{\rm otherwise},
    \end{array}\right.
  \]
  and write $\mathcal{P}(x,\mu):=\mathop{\arg\min}_{z\in\mathbb{R}^{n}}h(z,x,\mu)$ for $(x,\mu)\in\mathbb{R}^n\times\mathbb{R}$. Note that $h\!:\mathbb{R}^n\times\mathbb{R}^n\times\mathbb{R}\to\overline{\mathbb{R}}$ is a proper and lower semicontinuous function and is level-bounded in $z$ locally uniformly in $(x,\mu)$. In addition, by \cite[Theorem 1.25]{RW09}, the function $\widehat{h}(x,\mu):=\inf_{z\in\mathbb{R}^{n}}h(z,x,\mu)$ is finite and continuous on $\mathbb{R}\times[\mu_{\rm min},\widetilde{\mu}]$. From \cite[Example 5.22]{RW09}, the multifunction $\mathcal{P}\!:\mathbb{R}^n\times\mathbb{R}\rightrightarrows \mathbb{R}^n$ is outer semicontinuous relative to $\mathbb{R}^n\times[\mu_{\rm min},\widetilde{\mu}]$. Note that $\overline{x}^k\in \mathcal{P}(x^k,\overline{\mu}_k)$ for each $k\in K$. Then, $x^*\in{\rm prox}_{\overline{\mu}_*^{-1}g}\big(x^*\!-\!\overline{\mu}_{*}^{-1}\nabla f(x^*)\big)$, so $x^*$ is an $L$-stationary point of \eqref{model}.
 \end{proof}

 Next we use Lemma \ref{lemma-Fxk} (ii) to show that, after a finite number of iterations, the switch condition in \eqref{switch-condition} always holds and the Newton step is executed. 
 To this end, define
\begin{equation}\label{Tk-Sk}
 T_k\!:={\rm supp}(Bx^k),\ \overline{T}_k\!:= {\rm supp}(B\overline{x}^k),\ S_k\!:= {\rm supp}(x^k)\ \ {\rm and}\ \ \overline{S}_k\!:={\rm supp}(\overline{x}^k).   
\end{equation}
\begin{proposition}\label{prop-supp}
 For the index sets defined in \eqref{Tk-Sk}, there exist index sets $T\subset [p],S\subset [n]$ and an index $\overline{k}\in\mathbb{N}$ such that for all $k>\overline{k}$, $T_k = \overline{T}_k = T$ and $S_k = \overline{S}_k = S$, which means that $k\in \mathcal{K}_2$ for all $k>\overline{k}$. Moreover, for each $x^* \in \Gamma(x^0)$,  ${\rm supp}(Bx^*) = T$ and ${\rm supp}(x^*) = S$, where $\Gamma(x^0)$ is defined in Lemma \ref{lemma-Fxk} (iii).
\end{proposition}
\begin{proof}
 We complete the proof of the conclusion via the following three claims:

 \noindent
 {\bf Claim 1:} There exists $\overline{k}\in\mathbb{N}$ such that for $k>\overline{k}$, $\lvert Bx^k\rvert_{\min} \geq \frac{\nu}{2}$, where $\nu$ is the same as the one in Proposition \ref{xk-prop} (ii). Indeed, for each $k\!-\!1\in\mathcal{K}_1$, $x^k = \overline{x}^{k-1}$, and $\lvert Bx^k\rvert_{\min} = \lvert B\overline{x}^{k-1}\rvert_{\min} \geq  \nu > \frac{\nu}{2}$ follows by Proposition \ref{xk-prop} (ii). Hence, it suffices to consider that $k\!-\!1\in\mathcal{K}_2$. By Lemma \ref{lemma-Fxk} (ii), there exists $\overline{k}\in\mathbb{N}$ such that for all $k>\overline{k}$, $\|x^{k-1}-\overline{x}^{k-1}\|<\frac{\nu}{4\|B\|_2}$ and $\|d^{k-1}\| < \frac{\nu}{4\|B\|_2}$, which implies that for $\mathcal{K}_2\ni k-1>\overline{k}-1$, $\|Bx^{k-1}\!-\! B\overline{x}^{k-1}\|<\frac{\nu}{4}$ and $\|Bd^{k-1}\|< \frac{\nu}{4}$. For each $\mathcal{K}_2\ni k\!-1>\overline{k}\!-1$, let $i_{k}\in [p]$ be such that $\lvert (Bx^{k-1})_{i_{k}}\rvert =\lvert Bx^{k-1}\rvert_{\min}$. Since condition \eqref{switch-condition} implies that ${\rm supp}(Bx^{k-1}) = {\rm supp}(B\overline{x}^{k-1})$ for each $k-1\in\mathcal{K}_2$, we have $\lvert (B\overline{x}^{k-1})_{i_k}\rvert\ge\lvert B\overline{x}^{k-1}\rvert_{\rm min}$. Thus, for each $\mathcal{K}_2\ni k-1>\overline{k}-1$, 
 \begin{align*}
 \|Bx^{k-1} - B\overline{x}^{k-1}\| 
  &\ge\lvert (Bx^{k-1})_{i_k}-(B\overline{x}^{k-1})_{i_k}\rvert
   \ge\lvert (B\overline{x}^{k-1})_{i_k}\rvert- \lvert (Bx^{k-1})_{i_k} \rvert \\
  &\ge \lvert B\overline{x}^{k-1}\rvert_{\min} - \lvert Bx^{k-1}\rvert_{\min}.
  \end{align*}
  Recall that $\lvert B\overline{x}^{k-1}\rvert_{\min}\ge\nu$ for all $k\in\mathbb{N}$ by Proposition \ref{xk-prop} (ii). Together with the last inequality and $\|Bx^{k-1}-B\overline{x}^{k-1}\|<\frac{\nu}{4}$, for each $\mathcal{K}_2\ni k-1>\overline{k}-1$, we have $\lvert Bx^{k-1}\rvert_{\min} \geq \frac{3\nu}{4}$. For each $\mathcal{K}_2\ni k-1>\overline{k}-1$, let $j_k\in [p]$ be such that $\lvert (Bx^{k})_{j_k}\rvert = \lvert Bx^{k}\rvert_{\min}$. By Remark \ref{remark-hybrid} (ii), ${\rm supp}(Bx^k) \subset {\rm supp}(Bx^{k-1})$ for each $k-1\in\mathcal{K}_2$, which along with $j_k\in {\rm supp}(Bx^k)$ implies that $\lvert (Bx^{k-1})_{j_k} \rvert\geq \lvert Bx^{k-1}\rvert_{\min}$. Thus, for each $\mathcal{K}_2\ni k-1>\overline{k}-1$,
  \begin{align*}
   \|Bd^{k-1}\|&\geq \|Bx^{k} - Bx^{k-1}\| 
   \geq \lvert (Bx^{k-1})_{j_k}-(Bx^{k})_{j_k} \rvert\\
   &\ge \lvert (Bx^{k-1})_{j_k}\rvert- \lvert (Bx^{k})_{j_k} \rvert  
    \geq \lvert Bx^{k-1}\rvert_{\min} - \lvert Bx^{k}\rvert_{\min},
  \end{align*}
   which together with $\|Bd^{k-1}\|\le\frac{\nu}{4}$ and $\lvert Bx^{k-1}\rvert_{\min} \geq \frac{3\nu}{4}$ implies that $\lvert Bx^{k}\rvert_{\min} \ge\frac{\nu}{2}$. 

  \noindent
  {\bf Claim 2:} $T_k = \overline{T}_k$ for $k>\overline{k}$. From the above arguments, $\|Bx^k\!-B\overline{x}^k\|\leq \frac{\nu}{4}$ for $k>\overline{k}$. If $i\in T_k$, then $\lvert (B\overline{x}^k)_i\rvert \geq  \lvert (Bx^k)_i\rvert - \frac{\nu}{4} \geq \frac{\nu}{4}$, where the second inequality is using $\lvert Bx^k\rvert_{\min} \!>\! \frac{\nu}{2}$ by {\bf Claim 1}. This means that $i\in \overline{T}_k$, so $T_k\subset\overline{T}_k$. Conversely, if $i\in\overline{T}_k$, then $\lvert (Bx^k)_i\rvert \geq \lvert (B\overline{x}^k)_i\rvert - \frac{\nu}{4} \geq \frac{3\nu}{4}$, so $i\in T_k$ and $\overline{T}_k \subset T_k$. Thus, $T_k = \overline{T}_k$ for $k>\overline{k}$. 

   \noindent
  {\bf Claim 3:} $T_k = T_{k+1}$ for $k>\overline{k}$. If $k\in \mathcal{K}_1$, the result follows directly by the result in {\bf Claim 2}. If $k\in\mathcal{K}_2$, from the proof of {\bf Claim 1}, $\|Bx^k\!-Bx^{k+1}\|\leq \|Bd^k\|\leq \frac{\nu}{4}$ for all $k>\overline{k}$. Then, if $i\in T_k$, $\lvert (Bx^{k+1})_i\rvert \geq  \lvert (Bx^k)_i\rvert - \frac{\nu}{4} \geq \frac{\nu}{4}$, where the second inequality is using $\lvert Bx^k\rvert_{\min} \!>\! \frac{\nu}{2}$ by {\bf Claim 1}. This implies that $i\in T_{k+1}$ and $T_k \subset T_{k+1}$. Conversely, if $i\in T_{k+1}$, then $\lvert (Bx^k)_i\rvert \geq \lvert (Bx^{k+1})_i\rvert - \frac{\nu}{4} \geq \frac{\nu}{4}$. Hence, $i\in T_k$ and $T_{k+1} \subset T_k$.
 
  From {\bf Claim 2} and {\bf Claim 3}, there exists $T\subset [p]$ such that $T_k = \overline{T}_k = T$ for $k>\overline{k}$. Using the similar arguments, we can also prove that there exists $S\subset [n]$ such that $S_k = \overline{S}_k = S$ for all $k>\overline{k}$ (if necessary increasing $\overline{k}$). 
	
  Pick any $x^* \in \Gamma(x^0)$. Let $\{x^k\}_{k\in K}$ be a subsequence such that $\lim_{K\ni k\to\infty} x^{k} = x^*$. By the above proof, for all sufficiently large $k\in K$, $\lvert Bx^{k}\rvert_{\min} \geq \frac{\nu}{2}$ and $\lvert x^{k}\rvert_{\min} \geq \frac{\nu}{2}$, which implies that $\lvert Bx^*\rvert_{\min} \geq \frac{\nu}{2}$ and $\lvert x^*\rvert_{\min} \geq \frac{\nu}{2}$. The results ${\rm supp}(Bx^*) = T$ and ${\rm supp}(x^*) = S$ can be obtained by a proof similar to {\bf Claim 3}. The proof is completed.
\end{proof}

By Proposition \ref{prop-supp}, $k\in \mathcal{K}_2$ for all $k>\overline{k}$, i.e., the sequence $\{x^{k+1}\}_{k>\overline{k}}$ is generated by the Newton step. This means that $\{x^{k+1}\}_{k>\overline{k}}$ is identical to the one generated by the inexact projected regularized Newton method starting from $x^{\overline{k}+1}$. Also, since $\Pi_k = \Pi_*:=\Pi_{\overline{k}+1}$ for all $k>\overline{k}$, Algorithm \ref{hybrid} finally reduces to the inexact projected regularized Newton method for solving 
\begin{equation}\label{def-psi}
	\min_{x\in\mathbb{R}^n} \ \phi(x) := f(x) + \delta_{\Pi_*}(x),
\end{equation}
which is a minimization problem of function $f$ over polyhedron $\Pi_*$, much simpler than the original problem \eqref{model}. Consequently, the global convergence and local convergence rate analysis of PGiPN for model \eqref{model} boils down to analyzing those of the inexact projected regularized Newton method for \eqref{def-psi}. The rest of this section is devoted to this.
Unless otherwise stated, the notation $\overline{k}$ in the sequel is always the same as that of Proposition \ref{prop-supp}. In addition, we require the assumption that $\nabla^2\!f$ is locally Lipschitz continuous on $\Gamma(x^0)$, where $\Gamma(x^0)$ is defined in Lemma \ref{lemma-Fxk} (iii).

\begin{assumption}\label{ass-Lip}
 $\nabla^2\!f$ is locally Lipschitz continuous on an open set $\mathcal{O}\supset\Gamma(x^0)$.
\end{assumption}  

  Assumption \ref{ass-Lip} is very standard when analyzing the convergence behavior of Newton-type method. The following lemma reveals that under this assumption, the step size $\alpha_k$ in Newton step takes $1$ when $k$ is sufficiently large. Since the proof is similar to that of \cite[Lemma B.1]{liu22}, the details are omitted here.
\begin{lemma}\label{unit-step}
 Suppose that Assumption \ref{ass-Lip} holds. Then $\alpha_k = 1$ for sufficiently large $k$.
\end{lemma}

Notice that $\Pi_*$ is a polyhedron, which can be expressed as  
\begin{equation}\label{def-pistar}
\Pi_* = \big\{x\in\mathbb{R}^n \ \lvert \ B_{T_{\overline{k}+1}^c\cdot}x= 0, \ x_{S_{\overline{k}+1}^c} = 0, \ x\geq l,\ -x \geq -u\big\}.
\end{equation}
For any $x\in\mathbb{R}^n$, we define multifunction $\mathcal{A}:\mathbb{R}^n\rightrightarrows [2n]$ as 
$$\mathcal{A}(x) := \{i \ \lvert \ x_i = l_i\} \cup \{i+n \ \lvert \ x_i = u_i\}.$$
Clearly, for $x\in \Pi_*$, $\mathcal{A}(x)$ is the active set of constraint $\Pi_*$ at $x$. 
To prove the global convergence for PGiPN, we first show that $\{\mathcal{A}(x^k)\}_{k\in\mathbb{N}}$ remains stable when $k$ is sufficiently large, under the following non-degeneracy assumption.

\begin{assumption}\label{ass-ri}
 For all $x^*\in \Gamma(x^0)$, $0 \in\nabla f(x^*)+ {\rm ri}(\mathcal{N}_{\Pi_*}(x^*)).$
\end{assumption}

It follows from Proposition \ref{prop-opt} and Lemma \ref{lemma-Fxk} (iii) that for each $x^*\in\Gamma(x^0)$, $x^*$ is a stationary point of $F$, which together with Proposition \ref{prop-supp} and Lemma \ref{subdiff-hpconvex} (i) yields that $0 \in \nabla f(x^*)+ \mathcal{N}_{\Pi_*}(x^*)$, so that Assumption \ref{ass-ri} substantially requires that $-\nabla f(x^*)$ does not belong to the relative boundary\footnote{For $\Xi\subset \mathbb{R}^n$, the set difference ${\rm cl }(\Xi)\backslash {\rm ri}(\Xi)$ is called the relative boundary of $\Xi$, see \cite[p. 44]{convexanalysis}.} of $\mathcal{N}_{\Pi_*}(x^*)$. In the next lemma, we prove that under Assumptions \ref{ass-Lip}-\ref{ass-ri}, $\mathcal{A}(x^k) = \mathcal{A}(x^{k+1})$ for sufficiently large $k$.

\begin{lemma}\label{actset-stable}
Let $\{x^k\}_{k\in\mathbb{N}}$ be the sequence generated by Algorithm \ref{hybrid}. Suppose that Assumptions \ref{ass-Lip}-\ref{ass-ri} hold. Then, there exist $\mathcal{A}^*\subset [2n]$ and a closed and convex cone $\mathcal{N}^*\subset \mathbb{R}^n$ such that $\mathcal{A}(x^k) = \mathcal{A}^*$ and $\mathcal{N}_{\Pi_*}(x^k) = \mathcal{N}^*$ for sufficiently large $k$.
\end{lemma}
\begin{proof}
    We complete the proof via the following two claims.

    {\bf Claim 1:} We first claim that $\lim_{k\rightarrow\infty} \|{\rm proj}_{\mathcal{T}_{\Pi_*}(x^k)}(-\nabla\! f(x^k))\| = 0.$
	Since $\Pi_*$ is polyhedral, for any $x\in\Pi_*$, $\mathcal{T}_{\Pi_*}(x)$ and $\mathcal{N}_{\Pi_*}(x)$ are closed and convex cones, and $\mathcal{T}_{\Pi_*}(x)$ is polar to $\mathcal{N}_{\Pi_*}(x)$, which implies that when $k$ is sufficiently large, $z = {\rm proj}_{\mathcal{T}_{\Pi_*}(x^k)}(z) + {\rm proj}_{\mathcal{N}_{\Pi_*}(x^k)}(z)$ holds for any $z\in\mathbb{R}^n$.  Then, for all sufficiently large $k$, 
    $$\|{\rm proj}_{\mathcal{T}_{\Pi_*}(x^k)}(-\nabla\! f(x^k))\| = \|-\!\nabla\!f(x^k)\! -\! {\rm proj}_{\mathcal{N}_{\Pi_*}(x^k)}(-\nabla\! f(x^k))\| = {\rm dist}(0, \partial \phi(x^k)).$$
  Therefore, it suffices to prove that $\lim_{k\rightarrow\infty} {\rm dist}(0, \partial \phi(x^k)) = 0$. By Proposition \ref{prop-supp}, equation \eqref{inexact-cond2}, Remark \ref{remark-hybrid} (iv) and Lemma \ref{lemma-Fxk} (ii), there exists $\{\zeta_k\}_{k>\overline{k}}$ with $\lim_{k\rightarrow \infty} \|\zeta_k\| = 0$ such that $0\in \nabla f(x^k) + G_kd^k + \zeta_k + \mathcal{N}_{\Pi_k}(x^k+d^k)$ for each $k\ge\overline{k}$, which implies that $\nabla f(x^k + d^k) - \nabla f(x^k)- G_k d^k-\zeta_k \in \partial \phi(x^k + d^k)$ for each $k\ge\overline{k}$. This together with Lemma \ref{lemma-Fxk} (ii) and the continuity of $\nabla f$ implies that $\lim_{k\rightarrow\infty} {\rm dist}(0,\partial \phi(x^k+d^k)) = 0$. Thus, by Lemma \ref{unit-step} we obtain the desired result. 
	
 {\bf Claim 2:} We now show that $\mathcal{A}(x^k) \subset \mathcal{A}(x^{k+1})$ for all sufficiently large $k$. If not, there exists $K \subset \mathbb{N}$ such that $\mathcal{A}(x^k) \not\subset \mathcal{A}(x^{k+1})$ for all $k\in K$. If necessary taking a subsequence, we assume that $\{x^k\}_{k\in K}$ converges to $x^*$. By Lemma \ref{lemma-Fxk} (ii), $\{x^{k+ 1}\}_{k\in K}$ converges to $x^*$. In addition, from {\bf Claim 1} it follows that $\lim_{k\rightarrow\infty}\|{\rm proj}_{\mathcal{T}_{\Pi_*}(x^{k+1})} (-\nabla\!f(x^{k+1}))\| = 0$. The two sides along with Assumption \ref{ass-ri} and \cite[Corollary 3.6]{burke1988} yields that $\mathcal{A}(x^{k+1}) = \mathcal{A}(x^*)$ for all sufficiently large $k\in K$. Since $\mathcal{A}(x^k) \subset \mathcal{A}(x^*)$ for sufficiently large $k\in K$, we have $\mathcal{A}(x^{k}) \subset \mathcal{A}(x^{k+1})$ for sufficiently large $k\in K$, contradicting to $\mathcal{A}(x^{k}) \not\subset \mathcal{A}(x^{k+1})$ for $k\in K$. The claimed fact that $\mathcal{A}(x^k) \subset \mathcal{A}(x^{k+1})$ then follows. 
 
 From $\mathcal{A}(x^k) \subset \mathcal{A}(x^{k+1})$ for all sufficiently large $k$, 
 $\{\mathcal{A}(x^k)\}_{k\in\mathbb{N}}$ converges to for some $A^*\subset[2n]$ in the sense of Painlev$\acute{e}$-Kuratowski. From the finiteness of $\mathcal{A}^*$, we conclude that $\mathcal{A}(x^k) = \mathcal{A}^*$ for all sufficiently large $k$. From the expression of $\Pi_{*}$ in \eqref{def-pistar} and $\mathcal{A}(x^k) = \mathcal{A}^*$ for all sufficiently large $k$, we have $\mathcal{N}_{\Pi_*}(x^k) = \mathcal{N}^*$ for all sufficiently large $k$. 
\end{proof}

Our proof for the global convergence of PGiPN additionally requires the following assumption.
\begin{assumption}\label{ass-ratio}
	For every sufficiently large $k$, there exists $\xi_k \in \mathcal{N}_{\Pi_*}(x^k)$ such that  $$\liminf_{k\rightarrow\infty}\frac{-\langle \nabla \!f(x^k)+\xi_k, d^k\rangle}{\|\nabla f(x^k) + \xi_k\| \|d^k\|} > 0.$$
\end{assumption} 

This assumption essentially requires that for every sufficiently large $k$ there exists one element $\xi_k \in \partial \phi(x^k)$ such that the angle between $\nabla\!f(x^k)+\xi_k$ and $d^k$ is uniformly larger than ${\pi}/{2}$. For every sufficiently large $k$, since $x^k+\alpha d^k\in\Pi_*$ for all $\alpha\in[0,1]$, we have $d^k\in\mathcal{T}_{\Pi^*}(x^k)$, which implies that $\langle\xi^k,d^k\rangle\le 0$. Together with \eqref{convex_subgrad},  
for every sufficiently large $k$, the angle between $\nabla\!f(x^k)+\xi_k$ and $d^k$ is larger than $\pi/2$. This means that it is highly possible for Assumption \ref{ass-ratio} to hold. Obviously, when $n=1$, it automatically holds.

Next, we show that if $\phi$ is a KL function and Assumptions \ref{ass-Lip}-\ref{ass-ratio} hold, the sequence generated by PGiPN is Cauchy and converges to an $L$-stationary point. 

\begin{theorem}\label{gconverge}
 Let $\{x^k\}_{k\in\mathbb{N}}$ be the sequence generated by Algorithm \ref{hybrid}. Suppose that Assumptions \ref{ass-Lip}-\ref{ass-ratio} hold, and that $\phi$ is a KL function. Then, $\sum_{k=1}^\infty \|x^{k+1}\!-\!x^k\|<\infty$, and consequently $\{x^k\}_{k\in\mathbb{N}}$ converges to an $L$-stationary point of \eqref{model}.
\end{theorem}
\begin{proof}
	If there exists $\widetilde{k}>\overline{k}$ such that $\phi(x^{\widetilde{k}}) = \phi(x^{\widetilde{k}+1})$, then $F(x^{\widetilde{k}}) = F(x^{\widetilde{k}+1})$ by Proposition \ref{prop-supp}, which together with Lemma \ref{lemma-Fxk} (i) yields
	that $x^{\widetilde{k}} = \overline{x}^{\widetilde{k}}$. Consequently, $x^{\widetilde{k}}$ meets the termination condition of Algorithm \ref{hybrid}, so that $\{x^k\}_{k\in\mathbb{N}}$ converges to an $L$-stationary point of \eqref{model} within a finite number of steps. Thus, we only need to consider the case that $\phi(x^k)>\phi(x^{k+1})$ for all $k>\overline{k}$.
	By \cite[Lemma 6]{Bolte14}, there exist $\varepsilon>0,\eta>0$
	and a continuous concave function $\varphi\in \Upsilon_{\eta}$ such that for all $\overline{x}\in \Gamma(x^0)$
	and $x\in\{z\in\mathbb{R}^n\ \lvert \ {\rm dist}(z, \Gamma(x^0))<\varepsilon\}\cap[\phi(\overline{x})<\phi<\phi(\overline{x})+\eta]$,
	$ \varphi'(\phi(x) -\phi(\overline{x})){\rm dist}(0,\partial \phi(x)) \geq 1$, where $\Upsilon_\eta$ is defined in Definition \ref{KL-Def}, and $\Gamma(x^0)$ is defined in Lemma \ref{lemma-Fxk} (iii).
	Clearly, $\lim_{k\rightarrow\infty} {\rm dist}(x^k, \Gamma(x^0))=0$.  Pick any $x^*\in \Gamma(x^0)$. By the definition of $\phi$, Propositions \ref{xk-prop} (i) and \ref{prop-supp}, we have $\lim_{k\to\infty}\phi(x^k)=\phi(x^*)$. Then, for $k>\overline{k}$ (if necessary by increasing $\overline{k}$), 
	$x^k\in\{z\in\mathbb{R}^n\ \lvert \ {\rm dist}(z, \Gamma(x^0))<\varepsilon\}\cap[\phi(x^*)<\phi<\phi(x^*)+\eta]$. Consequently, for all $k>\overline{k}$, 
	\begin{equation}\label{KLexpre}
		\varphi'(\phi(x^k) -\phi(x^*)) {\rm dist}(0, \partial \phi(x^k)) \geq 1.
	\end{equation}
	By Assumption \ref{ass-ratio}, there exist $c>0$ and $\xi_k \in \mathcal{N}_{\Pi_*}(x^k)$ such that
	for all suffciently large $k$,
	\begin{equation}\label{ass-lbound-coro}
		-\langle \nabla f(x^k) + \xi_k, d^k \rangle >c \|\nabla f(x^k)+\xi_k\|\| d^k\|.
	\end{equation}
	From Lemma \ref{actset-stable} we have $\mathcal{N}_{\Pi_*}(x^k) = \mathcal{N}_{\Pi_*}(x^{k+1})$ for all $k>\overline{k}$ (by possibly enlarging $\overline{k}$), which implies that $\xi_k \in  \mathcal{N}_{\Pi_*}(x^{k+1})$.
	Together with \eqref{ls-NT}, \eqref{ass-lbound-coro} and Lemma \ref{unit-step}, we have that for all $k>\overline{k}$ (if necessary enlarging $\overline{k}$),
	\begin{equation}\label{FS-nabla}
			\frac{\phi(x^{k}) - \phi(x^{k+1})}{{\rm dist}(0,\partial \phi(x^k))}\!\ge\! \frac{-\varrho \langle \nabla f(x^k) + \xi_k, d^k\rangle }{{\rm dist}(0,\partial \phi(x^k))} \geq \frac{\varrho c\|\nabla f(x^k) + \xi_k\|\|d^k\|}{\|\nabla f(x^k) + \xi_k\|} = \varrho c\|x^{k+1}-x^k\|,
	\end{equation}
	where the second inequality follows by $\nabla f(x^k) + \xi_k \in \partial \phi(x^k)$ and \eqref{ass-lbound-coro}. For each $k$, let $\Delta_{k}:=\varphi(\phi(x^k)\!-\!\phi(x^*))$. From \eqref{KLexpre}, \eqref{FS-nabla} and the concavity of $\varphi$ on $[0,\eta)$, for all $k>\overline{k}$,
   \begin{align*}\label{KL-concave}
	\Delta_k-\Delta_{k+1}&=\phi(x^k)-\phi(x^{k+1})
		\geq \varphi'(\phi(x^k)\!-\!\phi(x^*))(\phi(x^k)\!-\!\phi(x^{k+1}))\\
     &\ge\frac{\phi(x^k)-\phi(x^{k+1})}{{\rm dist}(0,\partial \phi(x^k))}
		\geq \varrho c\|x^{k+1}-x^k\|.
    \end{align*}
    Summing this inequality from $\overline{k}$ to any  $k>\overline{k}$ and using $\Delta_k\ge 0$ yields that
	\begin{equation*}
		\sum_{j=\overline{k}}^k\|x^{j+1}\!-\!x^{j}\|
		\le\frac{1}{\varrho c} \sum_{j=\overline{k}}^k(\Delta_j\!-\!\Delta_{j+1})
		= \frac{1}{\varrho c}(\Delta_{\overline{k}}\!-\!\Delta_{k+1})
		\leq \frac{1}{\varrho c}\Delta_{\overline{k}}.
	\end{equation*}
	Passing the limit $k\to\infty$ leads to
	$\sum_{j=\overline{k}}^\infty \|x^{j+1}\!-\!x^j\|<\infty$. Thus, $\{x^k\}_{k\in\mathbb{N}}$ is a Cauchy sequence and converges to  $x^*$. It follows from Lemma \ref{lemma-Fxk} (iii) that $x^*$ is an $L$-stationary point of model \eqref{model}. The proof is completed.
\end{proof}

We now focus on the superlinear rate analysis of PGiPN.
Denote  $$\mathcal{X}^*: = \big\{x\in\mathbb{R}^n \ \lvert \ 0 \in \nabla f(x)+ \mathcal{N}_{\Pi_*}(x), \ \nabla^2 f(x) \succeq 0\big\},$$ which we call by the set of second-order stationary points of \eqref{def-psi}. Based on this notation, we assume that a local H\"{o}lderian error bound condition holds with $\mathcal{X}^*$ in Assumption \ref{ass4}. For more introduction on the H\"{o}lderian error bound condition, we refer the interested readers to \cite{Mordu23} and \cite{liu22}. 

\begin{assumption}\label{ass4}
 The $q$-subregularity of function $r(x):=x-{\rm proj}_{\Pi_*}(x - \nabla f(x))$ holds at $x^*$ for the origin with $\mathcal{X}^*$, i.e., there exist $\varepsilon>0, \ \kappa > 0$ and $q\in (0, 1]$ such that for all $x\in \mathbb{B}(x^*, \varepsilon)\cap \Pi_{*},$
	${\rm dist}(x,r^{-1}(0))={\rm dist}(x, \mathcal{X}^*) \leq \kappa\|r(x)\|^q.$
\end{assumption}

Recently, Liu et al. \cite{liu22} proposed an inexact
regularized proximal Newton method (IRPNM) for solving the problems, consisting of a smooth function and an extended real-valued convex function, which includes \eqref{def-psi} as a special case. They studied the superlinear convergence rate of IRPNM under Assumptions \ref{ass-Lip} and \ref{ass4}. 
By \cite[Lemma 4]{sra12} and $\overline{\mu}_k\in[\mu_{\rm min},\widetilde{\mu})$, $\|r_k(x^k)\| = O(\|r(x^k)\|)$ for sufficiently large $k$. This together with Assumption \ref{ass4} implies that there exists $\widehat{\kappa}>0$ such that for sufficiently large $k$ with $x^k \in \mathbb{B}(x^*, \varepsilon)$,
\begin{equation}\label{eq-errorbound}
    {\rm dist}(x^k, \mathcal{X}^*) \leq \widehat{\kappa} \|r_k(x^k)\|^q.
\end{equation}
Recall that PGiPN finally reduces to an inexact projected regularized Newton method for solving \eqref{def-psi}. From Lemma \ref{well-defined} (iii) and Lemma \ref{unit-step}, we have 
\begin{equation}\label{inexact-cond}
    \Theta_k(x^{k+1}) - \Theta_k(x^k) \leq 0\ \ {\rm and}\  \ \|R_k(x^{k+1})\| \leq \frac{1}{2}\min\{\|r_k(x^k)\|, \|r_k(x^k)\|^{1+\varsigma}\},
\end{equation}
for sufficiently large $k$. Let $\Lambda_k^i := G_k^i\!-\!\nabla^2\!f(x^k) - b_1 \|\overline{\mu}_k(x^k -\overline{x}^k)\|^{\sigma} I$, where $G_k^i$ are those in \eqref{rnm_liu}-\eqref{rnm_old}. Under Assumption \ref{ass4}, from \cite[Lemma 4.8]{wu22} and \cite[Lemma 4.4]{liu22} and the fact that $G_k^1 -G_k^2 \succeq 0$, we have for sufficiently large $k$ with $x^k \in \mathbb{B}(x^*, \varepsilon)$,
\begin{equation}\label{first-regterm}
    \max \big\{\lambda_{\min}(\Lambda_k^1), \lambda_{\min}(\Lambda_k^2), \lambda_{\min}(\Lambda_k^3)\big\} = O({\rm dist}(x^k, \mathcal{X}^*)),
\end{equation}
By using \eqref{eq-errorbound}, \eqref{inexact-cond} and \eqref{first-regterm}, and following a proof similar to \cite[Theorem 4.3]{liu22}, we can obtain the following result. 

\begin{theorem}\label{slconverge}
 Let $\{x^k\}_{k\in\mathbb{N}}$ be the sequence generated by Algorithm \ref{hybrid}. Suppose that Assumption \ref{ass-Lip} holds, and that $\{x^k\}_{k\in\mathbb{N}}$ converges to $x^* \in \Gamma(x^0).$ If Assumption \ref{ass4} holds with $q\in(\frac{1}{1+\sigma},1]$ at  $x^*$, then the sequence $\{x^k\}_{k\in\mathbb{N}}$ converges to $x^*$ with the $Q$-superlinear convergence rate at order $q(1\!+\!\sigma)$.
\end{theorem}
	
  \section{Numerical experiments}\label{sec7}
  
    This section focuses on the numerical experiments of several variants of PGiPN for solving a fused $\ell_0$-norms regularization problem with a box constraint. We first describe the implementation of Algorithm \ref{hybrid} in Section \ref{sec7.1}. In Section \ref{sec7.2}, we make comparison between model \eqref{model} with the least-squares loss function $f$ and the fused Lasso model \eqref{fused-lasso-Lag} by using PGiPN to solver the former and SSNAL \cite{li18b} to solve the latter, to highlight the advantages and disadvantages of our proposed fused $\ell_0$-norms regularization. 
    Among others, the code of SSNAL is available at \url{https://github.com/MatOpt/SuiteLasso}). 
    Finally, in Section \ref{sec7.3}, we present some numerical results toward the comparison among several variants of PGiPN and ZeroFPR and PG method for \eqref{model} in terms of efficiency and the quality of the output. The MATLAB code of PGiPN is available at \url{https://github.com/yuqiawu/PGiPN}.
	\subsection{Implementation of Algorithm \ref{hybrid}} \label{sec7.1}
	\subsubsection{Dimension reduction of \eqref{subp}}\label{sec7.1.1}
	Suppose that $\emptyset\neq S_k^c:=[n]\backslash S_k$. Based on the fact that every $x\in \Pi_k$ satisfies $x_{S_k^c} = 0$, we can obtain an approximate solution to \eqref{subp} by solving a problem in a lower dimension.  Specifically, for each $k\in\mathcal{K}_2$, write
	\[
		H_k\!:=\!(G_k)_{S_kS_k},\, v^k\!:=\!x^k_{S_k},\,\nabla\! f_{S_k}(v^k) \!=\! [\nabla\! f(x^k)]_{S_k},\, \widehat{\Pi}_k\!:=\{v\in\mathbb{R}^{\lvert S_k\rvert} \ \lvert \ \widetilde{B}_kv = 0, \, l_{S_k}\!\leq \!v \!\leq\! u_{S_k}\},
	\]
	where $\widetilde{B}_k$ is the matrix obtained by removing the rows of $B_{T_k^cS_k}$ whose elements are all zero. We turn to consider the following strongly convex optimization problem,
	\begin{equation}\label{reduce_pn}
		\widehat{v}^k \approx\mathop{\arg\min}_{v\in\mathbb{R}^{\lvert S_k\rvert}}  
		\Big\{\theta_k(v) := f(I_{\cdot S_k} v^k) +\langle\nabla\!f_{S_k}(v^k),v-\!v^k\rangle
		+\frac{1}{2}(v\!-v^k)^{\top}H_k(v\!-\!v^k) + \delta_{\widehat{\Pi}_k}(v)\Big\}.
	\end{equation}
	The following lemma gives a way to find $y^k$ satisfying \eqref{inexact-cond1}-\eqref{inexact-cond2} by inexactly solving problem \eqref{reduce_pn}, whose dimension is much smaller than that of \eqref{subp} if $\lvert S_k \rvert \ll n$. 
	\begin{lemma}\label{lemma-subspace}
		Let $y^k_{S_k} =\widehat{v}^k$ and $y^k_{S_k^c} = 0$. Then, $\Theta_k(y^k) = \theta_k(\widehat{v}^k)$ and ${\rm dist}(0, \partial \Theta_k(y^k)) = {\rm dist}(0, \partial \theta_k(\widehat{v}^k))$.
		Consequently, the vector $\widehat{v}^k$ satisfies $$\theta_k(\widehat{v}^k) - \theta_k(v^k) \leq 0,  \ 
			{\rm dist}(0, \partial \theta_k(\widehat{v}^k)) 
			\le\frac{\min \{\overline{\mu}_k^{-1}, 1\}}{2}\min \left\{\|\overline{\mu}_k(x^k\!-\!\overline{x}^k)\|, \|\overline{\mu}_k(x^k\!-\!\overline{x}^k)\|^{1+\varsigma} \right\},
		$$
	if and only if the vector $y^k$ satisfies the inexact conditions in \eqref{inexact-cond1}-\eqref{inexact-cond2}.
	\end{lemma}
	\begin{proof}
		The first part is straightforward. We consider the second part. By the definition of $\Theta_k$, 
		${\rm dist}(0, \partial \Theta_k(y^k)) = {\rm dist}(0, \nabla f(x^k) + G_k (y^k -x^k) + \mathcal{N}_{\Pi_k}(y^k))$.
		Recall that $\Pi_k= \{x\in\Omega\ \lvert \ B_{T_k^c\cdot} x = 0, x_{S_k^c} = 0\}$. Then, $\mathcal{N}_{\Pi_k}(y^k) = {\rm Range}(B_{T_k^c\cdot}^{\top}) + {\rm Range}(I^{\top}_{S_k^c\cdot}) + \mathcal{N}_{\Omega}(y^k)$, and 
		\begin{align*}
		{\rm dist}(0,\partial \Theta_k(y^k)) 
			& = {\rm dist}\big(0, \nabla f(x^k) + G_k (y^k -x^k) + {\rm Range}(B_{T_k^c\cdot}^{\top}) + {\rm Range}(I^{\top}_{S_k^c\cdot}) + \mathcal{N}_{\Omega}(y^k)\big)\\
            & = {\rm dist}\big(0, \nabla f_{S_k}(v^k) + H_k(\widehat{v}^k - v^k) + {\rm Range}(B_{T_k^cS_k}^{\top}) + \mathcal{N}_{[l_{S_k},u_{S_k}]}(\widehat{v}^k)\big)\\
			& = {\rm dist}(0, \nabla f_{S_k}(v^k) + H_k(\widehat{v}^k - v^k) + \mathcal{N}_{\widehat{\Pi}_k}(\widehat{v}^k)) = {\rm dist}(0,\theta_k(\widehat{v}^k)),
		\end{align*}
        where the second equality is using ${\rm Range}(I^{\top}_{S_k^c\cdot})= \{z\in\mathbb{R}^n\ | \ z_{S_k} = 0$\}. 
	\end{proof}
	\subsubsection{Acceleration of Algorithm 1}\label{sec7.1.3}
	
	The switch condition in \eqref{switch-condition} is in general difficult to be satisfied when $\|Bx^k\|_0$ or $ \|x^k\|_0$ is large. Consequently, PGiPN is continuously executing PG steps. This phenomenon is evident in the numerical experiment of the restoration of blurred image, see Section \ref{sec7.3.2}. To further accelerate the iterates of Algorithm \ref{hybrid} into the Newton step, we introduce the following relaxed switch condition:
	\begin{equation}\label{relax-condition}
		\|\lvert {\rm sign}(Bx^k)\rvert - \lvert{\rm sign}(B\overline{x}^k)\rvert\|_1 \leq \frac{\eta_1 n}{k}\ \ {\rm and} \ \  \|\lvert{\rm sign}(x^k) \rvert-\lvert {\rm     sign}(\overline{x}^k)\rvert\|_1 \leq \frac{\eta_2 n}{k},
	\end{equation}
	where $\eta_1,\eta_2$ are two nonnegative constants. Following the arguments similar to those in Lemma \ref{well-defined}, we have that Algorithm \ref{hybrid} equipped with \eqref{relax-condition} is also well defined. Obviously, when $\frac{\eta_i n}{k} \geq 1$, condition \eqref{relax-condition} allows the supports of $Bx^k$ and $B\overline{x}^k$ and $x^k$ and $\overline{x}^k$ have some difference; when $\frac{\eta_i n}{k} < 1 (i=1,2)$, condition \eqref{relax-condition} is identical to \eqref{switch-condition}. This means that as $k$ grows, Algorithm \ref{hybrid} with relaxed switch condition \eqref{relax-condition} will finally reduce to the one with \eqref{switch-condition}.  Since our convergence analysis does not specify the initial point, the asymptotic convergence results also hold for Algorithm \ref{hybrid} with condition \eqref{relax-condition}.
	\subsubsection{Choice of parameters in Algorithm \ref{hybrid}}\label{sec7.1.4}

    We will test the performance of PGiPN with $G_k$ given by $G_k^2$ in \eqref{rnm_new}, and PGiPN(r), which is PGiPN with relaxed switch condition \eqref{relax-condition}. We use Gurobi to solve subproblem \eqref{subp} with such $G_k$, with inexact conditions \eqref{inexact-cond1}, \eqref{inexact-cond2} controlled by options \textsf{params.Cutoff} and
    \textsf{params.OptimalityTol}, respectively.
    Also, we test PGilbfgs, which is the same as PGiPN, except using limited-memory BFGS (lbfgs) to construct $G_k$. In particular, we form $G_k = B_k+b_1\|\overline{\mu}_k(x^k -\overline{x}^k)\|^{\sigma}$, with $B_k$ given by lbfgs. For solving \eqref{subp} with such $G_k$, we use the method introduced in \cite{kanzow22}. We set the parameters of all the variants of PGiPN by
	\( \alpha=10^{-8},\
	\sigma = \frac{1}{2},\ \varrho = 10^{-4},\  \beta = \frac{1}{2},\ \varsigma = \frac{2}{3}.
	\)
	We set $b_1= 10^{-3}$ for PGiPN and PGiPN(r), and $b_1 = 10^{-8}$ for PGilbfgs. 

    We compare the numerical performance of our algorithms with those of ZeroFPR \cite{Themelis18} and the PG method \cite{Wright09}. In particular, ZeroFPR uses the quasi-Newton method to minimize the forward-backward envelope of the objective. The code package of ZeroFPR is downloaded from \url{http://github.com/kul-forbes/ForBES}. We set ``lbfgs'' as the solver of ZeroFPR. On the other hand, the iterate steps of PG are the same as those of PGiPN without the Newton steps, so that we can check the effect of the additional second-order step on PGiPN. For this reason, the parameters of PG are chosen to be the same as those involved in PG Step of PGiPN. We also observe that the sparsity of the output is very sensitive to $\mu_k$ in Algorithm \ref{hybrid}. To be fair, as the default setting in ZeroFPR, in all variants of PGiPN and PG, we set $\mu_k = 0.95^{-1}L_1$ for all $k\in\mathbb{N}$, where $L_1$ is an estimation of the Lipschitz constant of $\nabla\! f$ obtained by computing $\|A\|_2$ from the following MATLAB sentences:

    \noindent
     \textsf{ opt.issym = 1; opt.tol = 0.001; ATAmap = @(x) A'*A*x; L = eigs(ATAmap,n,1,`LM',opt)}.

    For each solver, we set $x^0 = 0$ and terminate at the iterate $x^k$ whenever $k\ge 5000$ or $\overline{\mu}_k\|x^k - {\rm prox}_{\overline{\mu}_k^{-1}g}(x^k- \overline{\mu}_k^{-1}\nabla \! f(x^k))\|_{\infty} < 10^{-4}$. All the numerical tests in this section are conducted on a desktop running on 64-bit Windows System with an Intel(R) Core(TM) i7-10700 CPU 2.90GHz and 32.0 GB RAM.
\subsection{Model comparison with the fused Lasso}\label{sec7.2}
 
 This subsection is devoted to the numerical comparison between the fused $\ell_0$-norms regularization problem with a box constraint (FZNS), i.e., model \eqref{model} with $f = \frac{1}{2}\|A\cdot-b\|^2$ and $B=\widehat{B}$ and the fused Lasso \eqref{fused-lasso-Lag}. We apply PGiPN to solve FZNS and SSNAL to solve \eqref{fused-lasso-Lag}. Since the solved models are different, we only compare the quality of solutions returned by these two solvers, and will not compare their running time.
 
 Our first empirical study focuses on the ability of regression. For this purpose, we use a commonly used dataset, prostate data, which can be downloaded from \url{https://hastie.su.domains/ElemStatLearn/}. There are 97 observations and 9 features included in this dataset. This data was used in \cite{jiang2021} to check the performance of square root fused Lasso. 

 We randomly select 50 observations to form the training set, which composes $A\in\mathbb{R}^{50\times 8}$. The corresponding responses are represented by $b\in\mathbb{R}^{50}$. The reminders are left as testing set, which forms $(\bar{A}, \bar{b})$ with $\bar{A} \in \mathbb{R}^{47\times 8}$ and $\bar{b} \in \mathbb{R}^{47}$. We employ PGiPN to solve FZNS, and SSNAL \cite{li18b} to solve the fused Lasso \eqref{fused-lasso-Lag}, with $(A,b)$ given above, and $l=-1000\times {\bf 1}$, $u = 1000\times{\bf 1}$. For each solver, we select $10$ groups of $(\lambda_1, \lambda_2) \in [0.003, 400] \times [0.0003, 40]$, ensuring that the outputs exhibit different sparsity levels. We record the sparsity and the testing error, where the later one is defined as $\|\bar{A}x^*\!-\!\bar{b}\|$ with $x^*$ being the output. 
 The above procedure is repeated for 100 randomly constructed $(A,b)$, resulting in a total of 1000 recorded outputs for each model. All the sparsity pairs $(\|\widehat{B}x^*\|_0, \|x^*\|_0)$ from PGiPN and SSNAL are recorded in lines 1, 3 and 5 in Table \ref{table-prostate}. For each sparsity pair, the mean testing errors of $\|\bar{A}x^*\!-\!\bar{b}\|$ for PGiPN and SSNAL corresponding to the given pair is recorded in lines 2, 4 and 6 in Table \ref{table-prostate}.
 Among others, since the fused Lasso may produce solutions with components being very small but not equal to $0$, we define $\|y\|_0 := \min\{k\ | \ \sum_{i=1}^k |\hat{y}| \geq 0.999\|y\|_1\}$ as in \cite{li18b}, where $\hat{y}$ is obtained by sorting $y$ in a nonincreasing order, for the outputs of the fused Lasso. As shown in Table \ref{table-prostate},
 it is evident that when $(\|\widehat{B}x^*\|_0, \|x^*\|_0) = (6,6)$, the mean testing error for FZNS is the smallest among all the testing examples. Furthermore,
 in the presented 21 comparative experiments, the fused Lasso outperforms FZNS for only 8 cases. Among these 8 experiments, in 7 cases, $\|\widehat{B}x^*\|_0 \geq 4$ and $\|x^*\|_0 \geq 6$. This indicates that our model performs better when the solution is relatively sparse. 

 \begin{table}[!ht]
	\centering
    \footnotesize
	\caption{Mean testing error (FZNS$|$Fused Lasso) of the outputs.\label{table-prostate}}
	\begin{tabular}{c|ccccccc}
		\toprule
		($\|\widehat{B}x^*\|_0, \|x^*\|_0$)   & (2,1) & (3,2) & (3,3) &(3,4) & (3,5) & (3,6) & (3,8) \\
        \midrule
        Mean testing error & {\color{blue}8.35}$|$8.54 & {\color{blue}7.34}$|$7.36 & 5.45$|${\color{blue}5.15} & {\color{blue}5.15}$|$5.74 & {\color{blue}5.21}$|$6.32 & {\color{blue}5.08}$|$5.27 & {\color{blue}5.11}$|$5.70\\
        \midrule
        ($\|\widehat{B}x^*\|_0, \|x^*\|_0$) & (4,4) & (4,5)& (4,6)& (4,7)& (4,8) & (5,5) & (5,6) \\
        \midrule
        Mean testing error & {\color{blue}5.07}$|$5.52 & {\color{blue}5.31}$|$5.86 & 5.49$|${\color{blue}4.99} & 5.25$|${\color{blue}4.97} & 5.33$|${\color{blue}4.78} & {\color{blue}5.10}$|$5.48 & 5.74$|${\color{blue}5.38}\\
        \midrule
        ($\|\widehat{B}x^*\|_0, \|x^*\|_0$) & (5,7) & (5,8) & (6,6) & (6,7) & (6,8) & (7,7) & (7,8) \\
        \midrule
        Mean testing error & {\color{blue}5.46}$|$5.58 & 5.35$|${\color{blue} 5.19} & {\color{blue}4.41}$|$ 5.26 & 5.34$|${\color{blue}4.95} & {\color{blue}5.25}$|$5.34 & {\color{blue}5.03}$|$5.22 & 5.24$|${\color{blue}5.22}\\
        \bottomrule
	\end{tabular}
\end{table}
 
Our second numerical study is to evaluate the classification ability of the two models using the TIMIT database. As introduced in Section \ref{sec1}, the TIMIT database is a widely used resource for research in speech recognition. Following the approach described in \cite{land97}, we compute a log-periodogram from each speech frame, which is one of the several widely used methods to generate speech data in a form suitable for speech recognition. Consequently, the dataset comprises 4509 log-periodograms of length 256 (frequency). It was highlighted in \cite{land97} that distinguishing between ``aa'' and ``ao'' is particularly challenging. Our aim is to classify these sounds using FZNS and the fused Lasso with $\lambda_2 = 0,\ l=-{\bf 1}$ and $u = {\bf 1}$, or in other words, the zero order variable fusion \eqref{fused-L0} plus a box constraint and the first order variable fusion \eqref{fused-L1}.

In TIMIT, the numbers of phonemes labeled ``aa'' and ``ao'' are 695 and 1022, respectively. As in \cite{land97}, we use the first 150 frequencies of the log-periodograms because the remaining 106 frequencies do not appear to contain any information. We randomly select $m_1$ samples labeled ``aa'' and $m_2$ samples labeled ``ao'' as training set, which together with their labels form $A\in\mathbb{R}^{m\times n}$ and $b\in\mathbb{R}^m$, with $m = m_1+m_2,\ n= 150$, where $b_i = 1$ if $A_{i\cdot}$ is labeled as ``aa'', and $b_i=2$ otherwise. The rest of dataset is left as the testing set, which forms $\bar{A}\in\mathbb{R}^{(1717-m)\times n}, \bar{b}^{1717-m}$, with $\bar{b}_i = 1$ if $\bar{A}_{i\cdot}$ is labeled as ``aa'' and $\bar{b}_i = 2$ otherwise. For $(A,b)$, given 10 $\lambda_1$'s randomly selected within $[2\times 10^{-5}, 300]$ such that the sparsity of the outputs $\|\widehat{B}x^*\|_0$ spans a wide range. If $\bar{A}_{i\cdot}x^* \leq 1.5$, this phoneme is classified as ``aa'' and hence we set $\hat{b}_i = 1$; otherwise, $\hat{b}_i = 2$. If $\hat{b}_i \neq \bar{b}_i$, $A_{i\cdot}$ is regarded as failure in classification. Then the error rate of classification is given by $\frac{\|\bar{b}-\hat{b}\|_1}{1717-m}$. We record both $\|\widehat{B}x^*\|_0$ and the error rate of classification.

The above procedure is repeated for $30$ groups of randomly generated $(A,b)$, resulting in 300 outputs for each solver. 
The four figures in Figure \ref{fig_phoneme} present $\|\widehat{B}x^*\|_0$ and the error rate for each output, with 4 different choices of $(m_1, m_2)$. We can see that, for each figure the output with the smallest error rate is always achieved by the fused $\ell_0$-norms regularization model. It is apparent that in general, FZNS performs better than the fused Lasso when $\|\widehat{B}x^*\|_0\leq 30$, while the mean error rate of the fused Lasso is lower than that of FZNS when $\|\widehat{B}x^*\|_0\geq 60$. This phenomenon is especially evident when $m_1$ and $m_2$ are small.
   \begin{figure}[h]
	\centering
	\setlength{\abovecaptionskip}{2pt}
	\subfigure[$m_1 = 25, m_2 = 50$]{
		\includegraphics[width=0.45\linewidth]{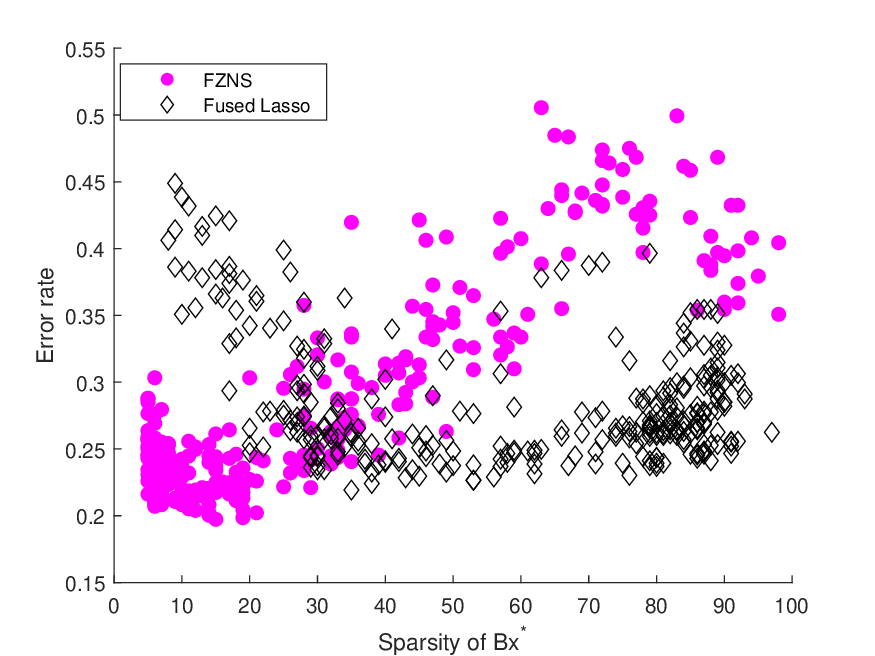}
		\label{scatter25}}
	\subfigure[$m_1 = 50, m_2 = 100$]{
		\includegraphics[width=0.45\linewidth]{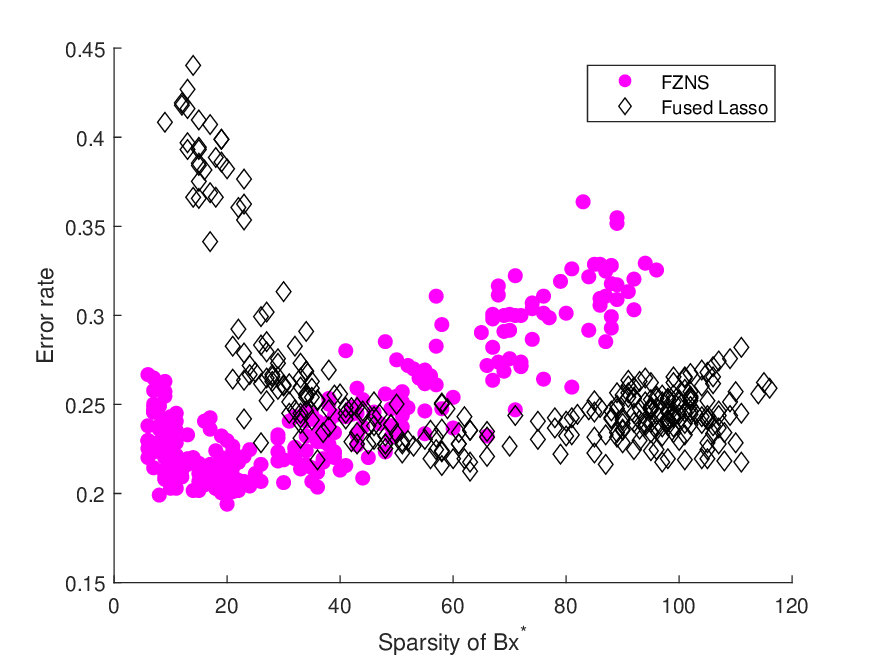}
		\label{scatter50}}
    \subfigure[$m_1=100, m_2 = 200$]{
		\includegraphics[width=0.45\linewidth]{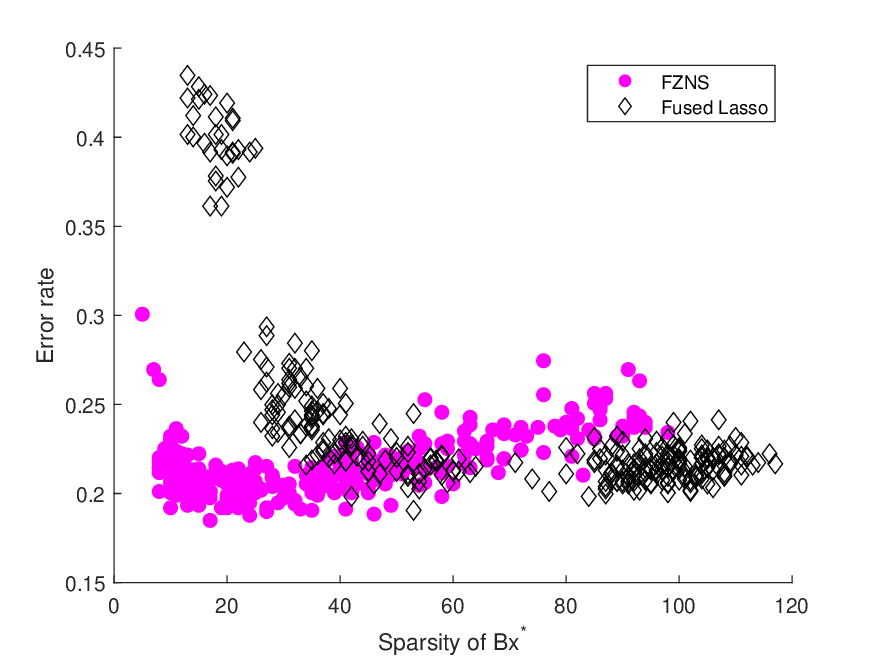}
		\label{scatter100}}
	\subfigure[$m_1=200, m_2 = 400$]{
		\includegraphics[width=0.45\linewidth]{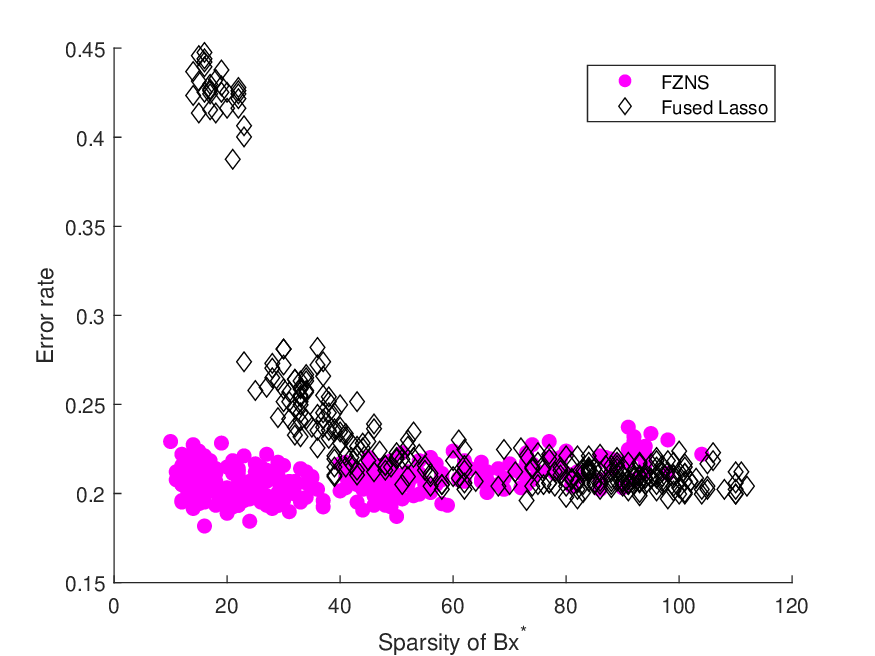}
		\label{scatter200}}
	\caption{$\|\widehat{B}x^*\|_0$ and the classification error rate for the outputs from FZNS and the fused Lasso under different $m_1, m_2$.}
	\label{fig_phoneme}
\end{figure}

Based on the results of these two empirical studies, we deduce that the fused $\ell_0$-norms regularization tends to outperform the fused Lasso regularization model when the output is sufficiently sparse. 
However, it is important to note that the numerical performance of the fused $\ell_0$-norms regularization is not stable if the output is not sparse, especially when the number of observations is small, which suggests that when employing the fused $\ell_0$-norms regularization, careful consideration should be given to selecting an appropriate penalty parameter. Moreover, due to the fact that for some optimal solution $x^*$ of the fused Lasso regularization problem, $|\widehat{B}x^*|_{\min}$ and $|x^*|_{\min}$ may be very small but not equal to zero, which leads to a difficulty in interpreting what the outputs mean in the real world application. This also well matches the statements in \cite{land97} that the $\ell_0$-norm variable fusion produces simpler estimated coefficient vectors.

\subsection{Comparison with ZeroFPR and PG}\label{sec7.3}
 This subsection focuses on the comparison among several variants of PGiPN, ZeroFPR and PG, in terms of efficiency and the quality of the outputs.
 \subsubsection{Classification of TIMIT\label{sec7.3.1}}
 The experimental data used in this part is the TIMIT dataset, the one in Section 6.2. To test the performance of the algorithms on \eqref{model} with nonconvex $f$, we consider solving model \eqref{model} with $f = \sum_{i=1}^m \ \log \Big(1+\frac{(A\cdot-b)_i}{\nu}\Big)$, $B = \widehat{B}$, $l = -{\bf 1}$ and $u = {\bf 1}$, 
 where $A\in\mathbb{R}^{m\times n}$ represents the training data and $b\in\mathbb{R}^m$ is the vector of the corresponding labels. It is worth noting that the loss function is nonconvex, and it was claimed in \cite{aravkin12} that this loss function is effective to process data denoised by heavy-tailed Student's $t$-noise. 

 Following the approach in Section \ref{sec7.2}, we use the first 150 frequencies of the log-periodograms. For the training set, we arbitrarily selecte 200 samples labeled as ``aa'' and 400 samples labeled as ``ao''. These samples, along with their corresponding labels, form the matrices $A\in\mathbb{R}^{m\times n}$ and $b\in\mathbb{R}^m$, with dimensions $m = 600$ and $n = 150$. The remaining samples are designated as the testing set. Given a series of nonnegative $\lambda_c$, we set $\lambda_1 = \lambda_c \times 10^{-7} \|A^{\top}b\|_{\infty}$ and $\lambda_2 = 0.1\lambda_1$. We employ four solvers: PGiPN, PGilbfgs, PG, and ZeroFPR. Subsequently, we record the CPU time and the error rate of classification on the testing set. This experimental procedure is repeated for a total of 30 groups of $(A,b)$, and the mean CPU time and error rate are recorded for each $\lambda_c$, presented in Figure \ref{figure-phonemewithzeroforandpg1}. Motivated by the experiment in Section \ref{sec7.2}, we also draw Figure \ref{figure-phonemewithzeroforandpg2}, recording $\|\widehat{B}x^*\|_0$ and the error rate for all the tested cases for four solvers.

\begin{figure}[ht]
	\centering
	\setlength{\abovecaptionskip}{2pt}
	\subfigure[$\lambda_c$-$\log({\rm time(seconds)})$ plot]{\label{lam_time}
		\includegraphics[width=0.45\linewidth]{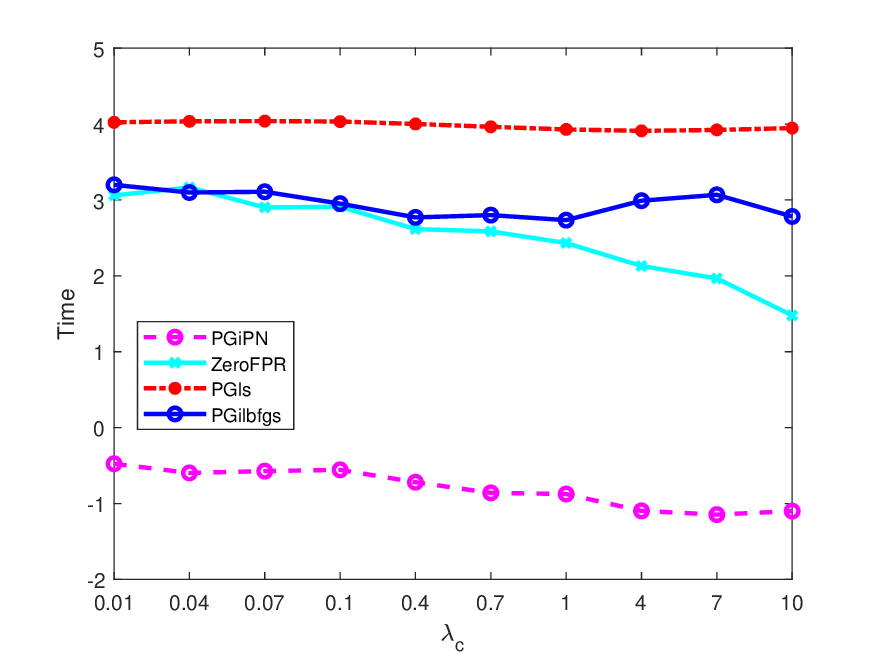}
		}
	\subfigure[$\lambda_c$-error rate plot]{\label{lam_err_rate}
		\includegraphics[width=0.45\linewidth]{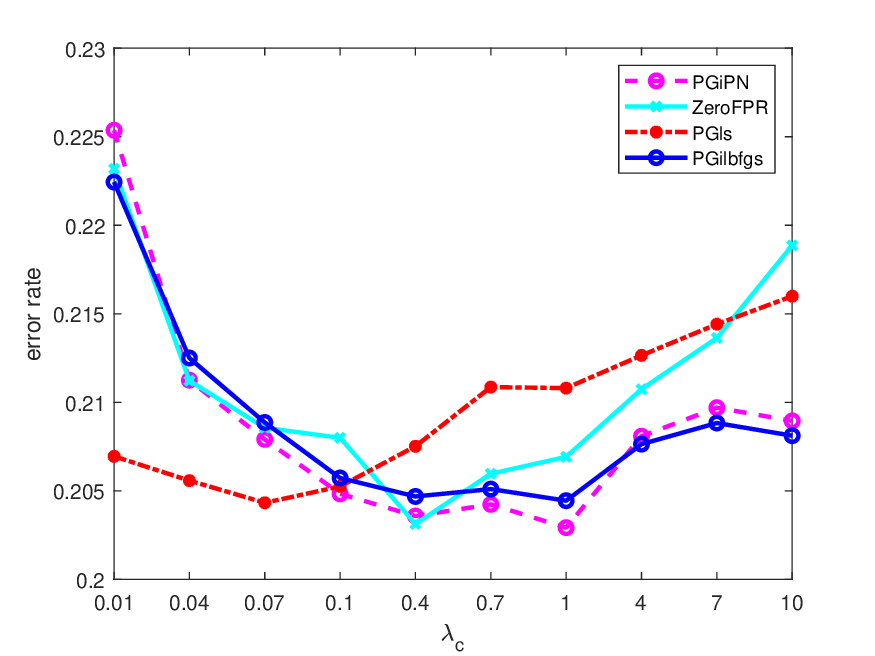}
		}
	\caption{Mean of the cpu time and the error rate on 30 examples for four solvers\label{figure-phonemewithzeroforandpg1}}
\end{figure}
\begin{figure}[ht]
	\centering
	\includegraphics[width=0.8\textwidth]{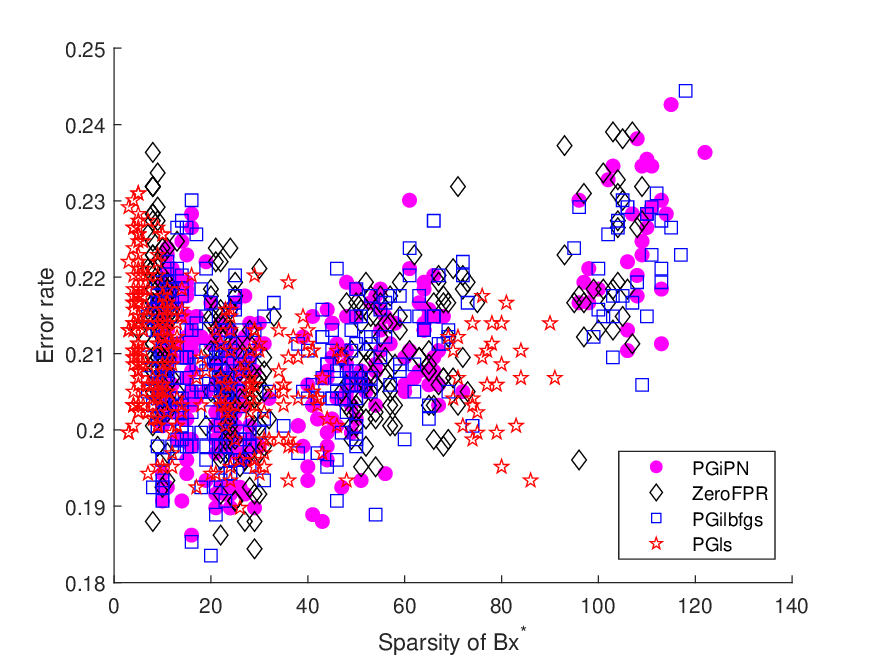}
	\caption{Scatter figure for all tested examples, recording the relationship of sparsity ($\|\widehat{B}x^*\|_0$) and the error rate of classification.}
	\label{figure-phonemewithzeroforandpg2}
\end{figure}

We see from Figure \ref{lam_time} that in terms of efficiency, PGiPN is always the best one, more than ten times faster than the other three solvers. The reason is that the other three solvers depend heavily on the proximal mapping of $g$, and its computation is a little time-consuming, which reflects the advantage of the projected regularized Newton steps in PGiPN. From Figure \ref{lam_err_rate}, when $\lambda_c = 1$, PGiPN reaches the smallest mean error rate among four solvers for 10 $\lambda_c$'s. When $\lambda_c$ is large ($>0.4$), PGiPN and PGilbfgs tend to outperform ZeroFPR and PG. Moreover, when $\lambda_c$ is small $(<0.1)$, the solutions returned by PG have the best error rate among four solvers. This is because $\widehat{B}x^*$ produced by PG is sparser than those of the other three solvers under the same $\lambda_c$, which can be observed from Figure \ref{figure-phonemewithzeroforandpg2}. For small $\lambda_c$, the solutions by the other three solvers are not sparse, leading to high error rate of classification.

\subsubsection{Recovery of blurred images\label{sec7.3.2}}
 Let $\overline{x}\in\mathbb{R}^n$ with $n = 256^2$ be a vector obtained by vectorizing a $256 \!\times\! 256$ image ``cameraman.tif'' in MATLAB, and be scaled such that all the entries belong to $[0,1]$. Let $A \in \mathbb{R}^{n\times n}$ be a matrix representing a Gaussian blur operator with standard deviation $4$ and a filter size of $9$, and the vector $b\in\mathbb{R}^m$ represents a blurred image obtained by adding Gauss noise $e \sim \mathcal{N}(0, \epsilon)$ with $\epsilon > 0$ to $A\overline{x}$, i.e., $b = A\overline{x} + e$. We apply model \eqref{model} with $f= \frac{1}{2}\|A\cdot - b\|^2$, $B = \widehat{B}$, $l = 0$ and $u = {\bf 1}$, to restore the blurred images.  We test five solvers, which are PGiPN, PGiPN(r), PGilbfgs, ZeroFPR and PG.  For PGiPN(r), we set $\eta_1 = 0.01, \eta_2 = 0.01$ in \eqref{relax-condition}. For all these five solvers, we employ $\lambda_1 = \lambda_2 = 0.0005 \times \|A^{\top}b\|_{\infty}$. 
 Under different $\epsilon$'s, we compare the performance of these five solvers in terms of required iterations (Iter), cpu time (Time), $F(x^*)$ (Fval), $\|x^*\|_0$ (xNnz), $\|\widehat{B}x^*\|_0$ (BxNnz) and the highest peak signal-to-noise ratio (PSNR), where
$ {\rm PSNR} := 10\log_{10} \left( \frac{n}{\|\overline{x} - x^*\|^2}\right).$ 
In particular, to check the effect of the Newton step, we record the iterations (or time) in the form $M(N)$, where $M$ means the total iterations (or time) and $N$ means the iterations (or time) in Newton step.
 PSNR measures the quality of the restored images. The higher PSNR, the better the quality of restoration. Table \ref{tab3} presents the numerical results. 

\begin{table}[!ht]
	\centering
	\caption{Numerical comparison of six solvers on recovery of blurred image with $\lambda_1 = \lambda_2 = 0.0005\|A^{\top}b\|_{\infty}$\label{tab3}}
	\begin{tabular}{ccccccccc}
		\toprule
		Noise  & & & PGiPN & PGiPN(r) & PGilbfgs &  PG & ZeroFPR  \\
		\midrule
		\multirow{6}{*}{\begin{tabular}[l]{@{}l@{}}$\epsilon = 0.01$\end{tabular}} &  & Iter & 379(3) & 123(6)  & 529(31)  & 796 & 361  \\
		&  & Time &  1.70e3(10.2) & {\bf\color{blue} 5.61e2(21.4)}  & 2.39e3(6.0)  & 3.43e3 & {\bf\color{red}2.35e4}  \\
		&   & Fval &  37.88 & {\bf \color{red}37.95} & 37.88 & 37.88 & {\bf \color{blue}37.77} \\
		&   & xNnz & 63858 & 63805 & 63858 & 63858& 63717 \\
		&   & BxNnz & 5767 &  5995 & 5776 & 5779 & 5834  \\
		&   & psnr & 25.90 & {\bf\color{red}25.77} & 25.90 & 25.90 & {\bf \color{blue}25.91} \\
		\cline{1-8}
		\multirow{6}{*}{\begin{tabular}[l]{@{}l@{}}$\epsilon = 0.02$\end{tabular}} &  & Iter & 281(4) & 109(6) & 457(38) & 853 & 286 \\
		&  & Time & 1.20e3(8.8) & {\bf\color{blue}4.74e2(13.2)} & 1.95e3(7.4) & 3.62e3 & {\bf\color{red} 1.82e4} \\
		&   & Fval & 45.98 & {\bf\color{red}46.05} & 45.98& 45.98 & {\bf\color{blue}45.83} \\
		&   & xNnz & 63495 & 63440 &63495  & 63495 & 63350 \\
		&   & BxNnz & 6098 & 6320 & 6098 & 6099 & 6143 \\
		&   & psnr &  25.41 &{\bf\color{red} 25.23} & {\bf \color{blue} 25.42} & {\bf \color{blue} 25.42} & 25.33  \\
		\cline{1-8}
		\multirow{6}{*}{\begin{tabular}[l]{@{}l@{}}$\epsilon = 0.03$\end{tabular}} &  & Iter &  234(3) & 94(3)  & 325(15) & 717 & 332 \\
		&  & Time & 9.8e2(6.5) & {\bf\color{blue}3.98e2(6.5)} & 1.36e3(6.7) & 2.97e3 & {\bf\color{red}1.88e4}  \\
		&   & Fval &  60.26 & {\bf\color{red}60.34} & 60.26 & 60.26 &{\bf\color{blue} 60.02} \\
		&   & xNnz & 63006 & 62944 & 63006 & 63006 & 62800  \\
		&   & BxNnz & 6594 & 6844 & 6597 & 6592 &  6710\\
		&   & psnr & {\bf\color{blue}24.90} & {\bf\color{red}24.67} & {\bf\color{blue}24.90} & {\bf\color{blue}24.90} &  24.76 \\
		\cline{1-8}
		\multirow{6}{*}{\begin{tabular}[l]{@{}l@{}}$\epsilon = 0.04$\end{tabular}} &  & Iter &  255(3) & 78(5)  & 360(19) & 526 & 230 \\
		&  & Time & 1.04e3(6.2) & {\bf\color{blue} 3.37e2(20.0)} & 1.45e3(3.9) & 2.18e3 & {\bf\color{red} 1.11e4} \\
		&   & Fval & 77.82 & {\bf \color{red}77.87} & 77.82 & 77.82 &{\bf\color{blue} 77.44} \\
		&   & xNnz & 62103 & 62002 & 62104 & 62104 & 61853 \\
		&   & BxNnz & 7267 & 7553 & 7268 & 7271 & 7427 \\
		&   & psnr & {\bf\color{blue}24.20} & {\bf\color{red} 23.85} & {\bf\color{blue}24.20} & {\bf\color{blue}24.20} & 24.00  \\
		\cline{1-8}
		\multirow{6}{*}{\begin{tabular}[l]{@{}l@{}}$\epsilon = 0.05$\end{tabular}} &  & Iter &  263(3) & 76(11)  & 389(29) & 688  & 168 \\
		&  & Time & 1.05e3(10.3) & {\bf\color{blue}3.30e2(28.5)} & 1.55e3(6.3) & 2.71e3 & {\bf\color{red}5.91e3} \\
		&   & Fval &  99.65 & {\bf\color{red}99.72} & 99.65 & 99.65 & {\bf\color{blue}98.93}  \\
		&   & xNnz & 61376 & 61286 & 61381 & 61381 & 60963 \\
		&   & BxNnz & 7955 & 8283 & 7955 &  7956 & 8240 \\
		&   & psnr & 23.36 & 23.00 & {\bf\color{blue}23.37} & {\bf\color{blue}23.37} & {\bf\color{red}22.87} \\
		\bottomrule
	\end{tabular}
\end{table}

From Table \ref{tab3}, PGiPN(r) always performs the best in terms of efficiency, which verifies the effectiveness of the acceleration scheme proposed in Section \ref{sec7.1.3}. PGiPN is faster than PGilbfgs, and PGilbfgs is faster than PG, supporting the effective acceleration of the Newton steps. However, ZeroFPR is the most time-consuming, even worse than PG, a pure first-order method. The reason is that ZeroFPR requires more line searches, and each line search involves a computation of the proximal mapping of $g$, which is expensive (2-5 seconds). 

Despite the superiority of efficiency, the solutions yielded by PGiPN(r) is not good. We also observe that $\|\widehat{B}x^*\|_0$ of PGiPN(r) is a little higher than those of PGiPN, PGilbfgs and PG, because PGiPN(r) runs few PG steps, so that its structured sparsity is not well reduced. Moreover, the PSNR is closely related to $\|\widehat{B}x^*\|_0$, and this leads to the weakest performance of PGiPN(r) in terms of PSNR. 
On the other hand, although ZeroFPR always outputs solutions with the smallest objective value, its PSNR is not as good as the objective value. The performance of PGiPN, PGilbfgs and PG in terms of the objective value and PSNR are quite similar. Taking the efficiency and the quality of the output into consideration, we conclude that PGiPN is the best solver for this test.

\section{Conclusions}\label{sec7.0}
In this paper, we proposed a hybrid of PG and inexact projected regularized Newton method for solving the fused $\ell_0$-norms regularization problem \eqref{model}. This hybrid framework fully exploits the advantages of PG method and Newton method, while avoids their disadvantages. We proved a global convergence by means of KL property, without assuming the uniformly positive definiteness of the regularized Hessian matrix, and also obtained a superlinear convergence rate under a H\"{o}lderian local error bound on the set of the second-order stationary points, without assuming the local minimality of the limit point. 

All PGiPN, ZeroFPR and PG have employed the polynomial-time algorithm for computing a point in the proximal mapping of $g$ when $B=\widehat{B}$, that we developed in this paper. Numerical tests indicate that our PGiPN not only produce solutions of better quality, but also requires 2-3 times less running time than PG and  ZeroFPR, where the latter mainly attributes to our subspace strategy when applying the projected regularized Newton method to solve the problems.

\bibliographystyle{siam}
\bibliography{references}

\begin{thebibliography}{10}

\bibitem{themelis21}
{\sc M.~Ahookhosh, A.~Themelis, and P.~Patrinos}, {\em A {B}regman
  forward-backward linesearch algorithm for nonconvex composite optimization:
  superlinear convergence to nonisolated local minima}, SIAM Journal on
  Optimization, 31 (2021), pp.~653--685.

\bibitem{aravkin12}
{\sc A.~Aravkin, M.~P. Friedlander, F.~J. Herrmann, and T.~Van~Leeuwen}, {\em
  Robust inversion, dimensionality reduction, and randomized sampling},
  Mathematical Programming, 134 (2012), pp.~101--125.

\bibitem{Attouch10}
{\sc H.~Attouch, J.~Bolte, P.~Redont, and A.~Soubeyran}, {\em Proximal
  alternating minimization and projection methods for nonconvex problems: An
  approach based on the {K}urdyka-{{\L}}ojasiewicz inequality}, Mathematics of
  Operations Research, 35 (2010), pp.~438--457.

\bibitem{Attouch13}
{\sc H.~Attouch, J.~Bolte, and B.~F. Svaiter}, {\em Convergence of descent
  methods for semi-algebraic and tame problems: proximal algorithms,
  forward--backward splitting, and regularized {G}auss-{S}eidel methods},
  Mathematical Programming, 137 (2013), pp.~91--129.

\bibitem{bareilles22}
{\sc G.~Bareilles, F.~Iutzeler, and J.~Malick}, {\em {N}ewton acceleration on
  manifolds identified by proximal gradient methods}, Mathematical Programming,
  200 (2023), pp.~37--70.

\bibitem{bauschke99}
{\sc H.~H. Bauschke, J.~M. Borwein, and W.~Li}, {\em Strong conical hull
  intersection property, bounded linear regularity, {J}ameson’s property (g),
  and error bounds in convex optimization}, Mathematical Programming, 86
  (1999), pp.~135--160.

\bibitem{bertsekas82}
{\sc D.~P. Bertsekas}, {\em Projected {N}ewton methods for optimization
  problems with simple constraints}, SIAM Journal on Control and Optimization,
  20 (1982), pp.~221--246.

\bibitem{NP97}
\leavevmode\vrule height 2pt depth -1.6pt width 23pt, {\em Nonlinear
  programming}, Journal of the Operational Research Society, 48 (1997),
  pp.~334--334.

\bibitem{bian20}
{\sc W.~Bian and X.~Chen}, {\em A smoothing proximal gradient algorithm for
  nonsmooth convex regression with cardinality penalty}, SIAM Journal on
  Numerical Analysis, 58 (2020), pp.~858--883.

\bibitem{blumensath08}
{\sc T.~Blumensath and M.~E. Davies}, {\em Iterative thresholding for sparse
  approximations}, Journal of Fourier Analysis and Applications, 14 (2008),
  pp.~629--654.

\bibitem{blumensath10}
\leavevmode\vrule height 2pt depth -1.6pt width 23pt, {\em Normalized iterative
  hard thresholding: Guaranteed stability and performance}, IEEE Journal of
  selected topics in signal processing, 4 (2010), pp.~298--309.

\bibitem{Bolte14}
{\sc J.~Bolte, S.~Sabach, and M.~Teboulle}, {\em Proximal alternating
  linearized minimization for nonconvex and nonsmooth problems}, Mathematical
  Programming, 146 (2014), pp.~459--494.

\bibitem{burke1988}
{\sc J.~V. Burke and J.~J. Mor{\'e}}, {\em On the identification of active
  constraints}, SIAM Journal on Numerical Analysis, 25 (1988), pp.~1197--1211.

\bibitem{davenport65}
{\sc H.~Davenport and A.~Schinzel}, {\em A combinatorial problem connected with
  differential equations}, American Journal of Mathematics, 87 (1965),
  pp.~684--694.

\bibitem{friedman07}
{\sc J.~Friedman, T.~Hastie, H.~H{\"o}fling, and R.~Tibshirani}, {\em Pathwise
  coordinate optimization}, The Annals of Applied Statistics, 1 (2007),
  pp.~302--332.

\bibitem{herrity06}
{\sc K.~K. Herrity, A.~C. Gilbert, and J.~A. Tropp}, {\em Sparse approximation
  via iterative thresholding}, in 2006 IEEE International Conference on
  Acoustics Speech and Signal Processing Proceedings, vol.~3, IEEE, 2006,
  pp.~III--III.

\bibitem{jewell20}
{\sc S.~W. Jewell, T.~D. Hocking, P.~Fearnhead, and D.~M. Witten}, {\em Fast
  nonconvex deconvolution of calcium imaging data}, Biostatistics, 21 (2020),
  pp.~709--726.

\bibitem{jiang2021}
{\sc H.~Jiang, S.~Luo, and Y.~Dong}, {\em Simultaneous feature selection and
  clustering based on square root optimization}, European Journal of
  Operational Research, 289 (2021), pp.~214--231.

\bibitem{kanzow22}
{\sc C.~Kanzow and T.~Lechner}, {\em Efficient regularized proximal
  quasi-{N}ewton methods for large-scale nonconvex composite optimization
  problems}, arXiv preprint arXiv:2210.07644,  (2022).

\bibitem{killick12}
{\sc R.~Killick, P.~Fearnhead, and I.~A. Eckley}, {\em Optimal detection of
  changepoints with a linear computational cost}, Journal of the American
  Statistical Association, 107 (2012), pp.~1590--1598.

\bibitem{land97}
{\sc S.~R. Land and J.~H. Friedman}, {\em Variable fusion: A new adaptive
  signal regression method}, tech. rep., Technical Report 656, Department of
  Statistics, Carnegie Mellon University, 1997.

\bibitem{Lee14}
{\sc J.~D. Lee, Y.~Sun, and M.~A. Saunders}, {\em Proximal {N}ewton-type
  methods for minimizing composite functions}, SIAM Journal on Optimization, 24
  (2014), pp.~1420--1443.

\bibitem{li18b}
{\sc X.~Li, D.~Sun, and K.-C. Toh}, {\em On efficiently solving the subproblems
  of a level-set method for fused {L}asso problems}, SIAM Journal on
  Optimization, 28 (2018), pp.~1842--1866.

\bibitem{liu09}
{\sc J.~Liu, S.~Ji, and J.~Ye}, {\em {SLEP}: Sparse learning with efficient
  projections}, Arizona State University, 6 (2009), p.~7.

\bibitem{liu10}
{\sc J.~Liu, L.~Yuan, and J.~Ye}, {\em An efficient algorithm for a class of
  fused {L}asso problems}, in Proceedings of the 16th ACM SIGKDD international
  conference on Knowledge discovery and data mining, 2010, pp.~323--332.

\bibitem{liu22}
{\sc R.~Liu, S.~Pan, Y.~Wu, and X.~Yang}, {\em An inexact regularized proximal
  {N}ewton method for nonconvex and nonsmooth optimization}, arXiv preprint
  arXiv:2209.09119v5,  (2023).

\bibitem{Lu14B}
{\sc Z.~Lu}, {\em Iterative hard thresholding methods for $\ell_0$ regularized
  convex cone programming}, Mathematical Programming, 147 (2014), pp.~125--154.

\bibitem{lu13}
{\sc Z.~Lu and Y.~Zhang}, {\em Sparse approximation via penalty decomposition
  methods}, SIAM Journal on Optimization, 23 (2013), pp.~2448--2478.

\bibitem{molinari19}
{\sc C.~Molinari, J.~Liang, and J.~Fadili}, {\em Convergence rates of
  {F}orward--{D}ouglas--{R}achford splitting method}, Journal of Optimization
  Theory and Applications, 182 (2019), pp.~606--639.

\bibitem{Mordu23}
{\sc B.~S. Mordukhovich, X.~Yuan, S.~Zeng, and J.~Zhang}, {\em A globally
  convergent proximal {N}ewton-type method in nonsmooth convex optimization},
  Mathematical Programming, 198 (2023), pp.~899--936.

\bibitem{pan22}
{\sc S.~Pan, L.~Liang, and Y.~Liu}, {\em Local optimality for stationary points
  of group zero-norm regularized problems and equivalent surrogates},
  Optimization, 72 (2023), pp.~2311--2343.

\bibitem{poliquin10}
{\sc R.~A. Poliquin and R.~T. Rockafellar}, {\em A calculus of
  prox-regularity}, J. Convex Anal, 17 (2010), pp.~203--210.

\bibitem{convexanalysis}
{\sc R.~T. Rockafellar}, {\em Convex analysis}, Princeton university press,
  1970.

\bibitem{RW09}
{\sc R.~T. Rockafellar and R.~J.-B. Wets}, {\em Variational analysis},
  vol.~317, Springer, 2009.

\bibitem{rudin92}
{\sc L.~I. Rudin, S.~Osher, and E.~Fatemi}, {\em Nonlinear total variation
  based noise removal algorithms}, Physica D: Nonlinear Phenomena, 60 (1992),
  pp.~259--268.

\bibitem{sharir88}
{\sc M.~Sharir}, {\em Davenport-schinzel sequences and their geometric
  applications}, in Theoretical Foundations of Computer Graphics and CAD,
  Springer, 1988, pp.~253--278.

\bibitem{sra12}
{\sc S.~Sra}, {\em Scalable nonconvex inexact proximal splitting}, Advances in
  Neural Information Processing Systems, 25 (2012).

\bibitem{stella17}
{\sc L.~Stella, A.~Themelis, and P.~Patrinos}, {\em Forward--backward
  quasi-{N}ewton methods for nonsmooth optimization problems}, Computational
  Optimization and Applications, 67 (2017), pp.~443--487.

\bibitem{Themelis19}
{\sc A.~Themelis, M.~Ahookhosh, and P.~Patrinos}, {\em On the acceleration of
  forward-backward splitting via an inexact {N}ewton method}, in Splitting
  Algorithms, Modern Operator Theory, and Applications, Springer, 2019,
  pp.~363--412.

\bibitem{Themelis18}
{\sc A.~Themelis, L.~Stella, and P.~Patrinos}, {\em Forward-backward envelope
  for the sum of two nonconvex functions: Further properties and nonmonotone
  linesearch algorithms}, SIAM Journal on Optimization, 28 (2018),
  pp.~2274--2303.

\bibitem{tibshirani05}
{\sc R.~Tibshirani, M.~Saunders, S.~Rosset, J.~Zhu, and K.~Knight}, {\em
  Sparsity and smoothness via the fused {L}asso}, Journal of the Royal
  Statistical Society: Series B (Statistical Methodology), 67 (2005),
  pp.~91--108.

\bibitem{Ueda10}
{\sc K.~Ueda and N.~Yamashita}, {\em Convergence properties of the regularized
  {N}ewton method for the unconstrained nonconvex optimization}, Applied
  Mathematics and Optimization, 62 (2010), pp.~27--46.

\bibitem{Wright09}
{\sc S.~J. Wright, R.~D. Nowak, and M.~A. Figueiredo}, {\em Sparse
  reconstruction by separable approximation}, IEEE Transactions on Signal
  Processing, 57 (2009), pp.~2479--2493.

\bibitem{wu2020}
{\sc F.~Wu and W.~Bian}, {\em Accelerated iterative hard thresholding algorithm
  for $l_0$ regularized regression problem}, Journal of Global Optimization, 76
  (2020), pp.~819--840.

\bibitem{wu22}
{\sc Y.~Wu, S.~Pan, and X.~Yang}, {\em A regularized {N}ewton method for
  $\ell_q$-norm composite optimization problems}, SIAM Journal on Optimization,
  33 (2023), pp.~1676--1706.

\bibitem{Yue19}
{\sc M.-C. Yue, Z.~Zhou, and A.~M.-C. So}, {\em A family of inexact {SQA}
  methods for non-smooth convex minimization with provable convergence
  guarantees based on the {L}uo-{T}seng error bound property}, Mathematical
  Programming, 174 (2019), pp.~327--358.

\bibitem{zhou21}
{\sc S.~Zhou, L.~Pan, and N.~Xiu}, {\em {N}ewton method for
  $\ell_0$-regularized optimization}, Numerical Algorithms, 88 (2021),
  pp.~1541--1570.

\end{thebibliography}
 \end{document}